\documentclass[11pt]{article}

\usepackage[utf8]{inputenc}
\usepackage{amsmath}
\usepackage{amssymb,amsfonts,amsthm}
\usepackage{mathrsfs}
\usepackage{xcolor}
\usepackage{tikz}
\usepackage{fullpage}
\usepackage{dsfont}
\usepackage{thmtools}
\usepackage{thm-restate}
\usepackage{enumitem}

\usepackage{amsfonts}
\usepackage{amssymb}
\usepackage{fancyhdr}
\usepackage{comment}
\usepackage{bm}
\usepackage{soul}

\usepackage{caption}
\usepackage[titles]{tocloft}
\usepackage{xcolor}

\usepackage{subfig}

\usepackage{wrapfig}

\usepackage{tikz}
\usetikzlibrary{arrows.meta}
\usepackage{floatflt}
\usepackage[ruled,vlined,linesnumbered]{algorithm2e}

\usepackage{hyperref}
\usepackage{xcolor}
\hypersetup{
  colorlinks   = true, %Colours links instead of ugly boxes
  urlcolor     = blue, %Colour for external hyperlinks
  linkcolor    = blue, %Colour of internal links
  citecolor   = blue %Colour of citations
}

\usepackage{amsmath}
\usepackage{dsfont}
\usepackage{graphicx}
\usepackage[hypcap]{caption}
\usepackage{color, soul}
\usepackage{float}
\usepackage{dsfont}

\newtheorem{theorem}{Theorem}[section]
\newtheorem{remark}{Remark}
\newtheorem{lemma}[theorem]{Lemma}
\newtheorem{corollary}[theorem]{Corollary}

\newtheorem{proposition}[theorem]{Proposition}
\newtheorem{fact}[theorem]{Fact}
\theoremstyle{definition}
\newtheorem{definition}{Definition}

\newtheorem{innercustomlem}{Lemma}
\newenvironment{mylem}[1]
  {\renewcommand\theinnercustomlem{#1}\innercustomlem\itshape}
  {\endinnercustomlem}

  \newtheorem{innercustomthm}{Theorem}
\newenvironment{mythm}[1]
  {\renewcommand\theinnercustomthm{#1}\innercustomthm\itshape}
  {\endinnercustomthm}

\newtheorem{innercustomcor}{Corollary}
\newenvironment{mycor}[1]
  {\renewcommand\theinnercustomcor{#1}\innercustomcor\itshape}
  {\endinnercustomcor}

\renewcommand*{\big}[1]{\vcenter{\hbox{\scalebox{1.3}{\ensuremath#1}}}}

\newcommand{\tsq}{\sqrt{\theta}}

\newcommand{\eqdist}{\stackrel{d}=}

\newcommand{\supp}{\mathrm{supp}}

\newcommand{\Rad}{\mathsf{Rad}}
\newcommand{\tr}{^\top}
\renewcommand{\P}{\mathbb{P}}

\newcommand{\sign}{\mathsf{sign}}

\newcommand{\la}{\langle}
\newcommand{\ra}{\rangle}

\newcommand{\R}{\mathbb{R}}
\newcommand{\Z}{\mathbb{Z}}

\newcommand{\B}{\Big}
\renewcommand{\b}{\big}

\newcommand{\up}[1]{^{(#1)}}

\newcommand{\ksparse}[1]{\mathcal{B}_{#1}(k)}
\newcommand{\ksparseflat}[1]{\overline{\mathcal{B}_{#1}}(k)}

\newcommand{\TV}{\tv}
\newcommand{\tv}{d_\mathrm{TV}}

\DeclareMathOperator{\Var}{\mathsf{Var}}

\newcommand{\sw}{\mathsf{SpWig}}
\renewcommand{\sc}{\mathsf{SpCov}}

\newcommand{\inv}{^{-1}}

\newcommand{\bern}[1]{\mathsf{Bern}(#1)}

\newcommand{\cL}{\mathcal{L}}

\newcommand{\mL}{\mathcal{L}}

\newcommand{\KL}{\kl}

\newcommand{\bP}{\mathbb{P}}
\newcommand{\bE}{\mathbb{E}}

\newcommand{\N}{\mathcal{N}}

\newcommand{\GS}{\mathtt{GS}}
\newcommand{\GaussC}{\mathtt{GaussClone}}
\newcommand{\CloneCov}{\mathtt{CloneCov}}
\newcommand{\WishartToSSBM}{\mathtt{SpCovToSpWig}}
\newcommand{\RK}{\mathtt{RK}}
\newcommand{\Denoise}{\mathtt{Denoise}}
\newcommand{\GaussCR}{\mathtt{GaussCloneRep}}

\newcommand{\tO}{\widetilde{O}}

\newcommand{\tOm}{\widetilde{\Omega}}

\newcommand{\Ysym}{Y_{\text{sym}}}

\newcommand{\Zt}{\widetilde{Z}}
\newcommand{\Xt}{\widetilde{X}}

\DeclareMathOperator*{\E}{{\rm I}\kern-0.18em{\rm E}}
\renewcommand{\P}{\,{\rm I}\kern-0.18em{\rm P}}

\newcommand{\law}{\mathcal{L}}

\newenvironment{fminipage}%
  {\begin{Sbox}\begin{minipage}}%
  {\end{minipage}\end{Sbox}\fbox{\TheSbox}}

\newcommand{\rad}{\mathsf{Rad}}

\newcommand{\kl}{\mathsf{KL}}

\newcommand{\lamcomp}{\lambda^\mathrm{comp}}
\newcommand{\lamstat}{\lambda^\mathrm{stat}}

\newcommand{\betacomp}{\beta^\mathrm{comp}}
\newcommand{\betastat}{\beta^\mathrm{stat}}

\newcommand{\thetacomp}{\theta^\mathrm{comp}}
\newcommand{\thetastat}{\theta^\mathrm{stat}}

\newcommand{\Sigmahat}{\widehat \Sigma}

\newcommand{\Xbar}{\bar X}

\newcommand{\fr}{\mathrm{F}}

\newcommand{\cc}[1]{%
  \ifcase#1 c_{\scriptscriptstyle K}% Case 0
  \or c''_{\scriptscriptstyle K}% Case 1
  \or c'_{\scriptscriptstyle K}% Case 2
  \else c_{u}% For all other cases (including non-numeric)
  \fi}

\newcommand{\boldnormal}[1]{\textbf{\textnormal{#1}}}

\newcommand{\Enorm}[1]{\mathcal{E}_{#1}^{\mathrm{norm}}}
\newcommand{\EE}[1]{\mathcal{E}_{#1}}
\newcommand{\Eg}{\mathcal{E}^{g}}
\newcommand{\Em}{\mathcal{E}}

\newcommand{\I}{\mathrm{I}}
\newcommand{\II}{\mathrm{II}}
\newcommand{\III}{\mathrm{III}}

\newcommand{\EIII}[1]{\mathcal{E}_{#1}^\III}

\renewcommand{\bP}{\mathbf{P}}
\newcommand{\equaldist}{\stackrel{d}=}

\newcommand{\FR}{\mathrm{Fr}}

\newcommand{\cU}{\mathcal{U}}

\newcommand{\cM}{\mathcal{M}}
\newcommand{\cQ}{\mathcal{Q}}

\newcommand{\PP}[2]{\mathcal{P}^{#1}_{#2}}
\newcommand{\QQ}[2]{\mathcal{Q}^{#1}_{#2}}

\newcommand{\cP}{\mathcal{P}}

\renewcommand{\tr}{\mathrm{Tr}}

\newcommand{\GOE}{\mathrm{GOE}}

\newcommand{\Omc}{\Omega^\mathsf{Cov}}
\newcommand{\Omw}{\Omega^\mathsf{Wig}}
\newcommand{\CCw}{\mathcal{C}^\mathsf{Wig}}
\newcommand{\CCc}{\mathcal{C}^\mathsf{Cov}}

\newcommand{\polylog}{\mathsf{polylog}}

\newcommand{\as}{a_{\mathsf{sp}}}
\newcommand{\an}{a_{-}}

\newcommand{\ws}{w_{\mathsf{sp}}}
\newcommand{\wn}{w_{-}}

\newcommand{\qquadand}{\qquad\text{and}\qquad}

\usepackage[titles]{tocloft}
\title{Computational Equivalence of Spiked Covariance and Spiked Wigner Models via Gram-Schmidt Perturbation\footnote{Accepted for presentation at the Conference on Learning Theory (COLT) 2025.}}

\author{Guy Bresler\thanks{Supported by NSF Grant CCF-2428619.} \and Alina Harbuzova\thanks{Supported by Vanu Bose Presidential Fellowship.}}

\date{Massachusetts Institute of Technology}

\begin{document}

\maketitle

% \pagenumbering{gobble}

\begin{abstract}
    In this work, we show the first average-case reduction transforming the sparse Spiked Covariance Model into the sparse Spiked Wigner Model and as a consequence obtain the first computational equivalence result between two well-studied high-dimensional statistics models. Our approach leverages a new perturbation equivariance property for Gram-Schmidt orthogonalization, enabling removal of dependence in the noise while preserving the signal.
\end{abstract}

% \begin{keywords}%
%   Computational-statistical gap, average-case reductions, spiked covariance model, sparse PCA, spiked Wigner model%
% \end{keywords}

\setcounter{tocdepth}{2}

\setlength{\cftbeforesecskip}{4pt}

\tableofcontents

\newpage
\setcounter{page}{1}

\section{Introduction}

A fundamental challenge in modern statistics is understanding the trade-off between statistical resource, in the form of data quantity or quality, and computational resource.
For a majority of structured statistical problems of modern interest, the best computationally efficient algorithms require significantly more data than what is theoretically sufficient in the absence of the computational constraints. This phenomenon, also known as ``computational-statistical gap", appears across a wide range of problems, including sparse linear regression \cite{zhang2014lower}, sparse principle component analysis (PCA) \cite{berthet2013optimal}, tensor PCA \cite{richard2014statistical}, submatrix detection \cite{ma2015computational}, community recovery \cite{hajek2015computational}, and many others \cite{brennan2020reducibility}. A thriving line of research aims to determine the fundamental limits of computationally efficient algorithms for the extensive array of modern statistical problems.

Traditional computational complexity theory 
focuses on problems whose input data are worst-case, i.e., adversarially chosen. In contrast, statistics and machine learning problems usually involve data generated by a benign process, often modeled by specific distributions. As a result, traditional worst-case complexity techniques, which primarily rely on reductions that relate one problem to another, do not directly apply to the modern \emph{average-case} or \emph{distributional} problems.\footnote{Early average-case complexity theory, e.g., \cite{levin1986average,bogdanov2006average}, was tailored to a very strong notion of dist-NP-completeness that works for \emph{any} distribution; this theory does not address problems studied in mathematical statistics which assume plausible data generating distributions. } To address this gap, a recent line of work focused on developing the \emph{average-case} reductions instead (\cite{berthet2013complexity,brennan2018reducibility,brennan2020reducibility,brennan2019optimal,ma2015computational,hajek2015computational}). These average-case reductions relate two problems of interest by precisely transforming one data distribution into another. This fundamental relationship between problems goes beyond merely transferring computational feasibility and demonstrates that their associated statistical and computational phenomena has the same root.

In parallel, researchers have sought to analyze the limits of \emph{restricted classes of algorithms}, including sum-of-squares \cite{barak2016nearly,hopkins2017power}, approximate message passing (AMP) \cite{zdeborova2016statistical}, statistical query \cite{diakonikolas2017statistical,kearns1998efficient,feldman2015complexity,feldman2013statistical}, low-degree polynomials \cite{hopkinsThesis,kunisky2019notes,wein2023average}, certain smooth algorithms through overlap gap methodology \cite{gamarnik2014limits,gamarnik2021overlap,huang2022tight}, and many others. While such analyses give insight, each class of algorithms has its own limitations and indeed completely fails to solve (or correctly predict computational hardness for) problems that are known to be easy \cite{li2024some,holmgren2021counterexamples,koehler2022reconstruction}. The limitations of this approach to hardness analysis are clear -- each time a new class of algorithms gains prominence, the entire collection of statistics problems of interest must be analyzed from scratch, one at a time. In this sense, by elucidating connections between problems themselves in an algorithm-independent way, average-case reductions are ``future-proofed".

A central triumph of the worst-case complexity theory is the \emph{classification of problems into equivalence classes}, which constitutes a major conceptual simplification:
one may study any \emph{single} problem from the class in place of them all. In contrast, existing average-case reductions do not establish equivalence classes like NP-complete problems for the worst-case setting. Instead, they create partial order of hardness: webs of reductions starting at planted clique or lattice problems relate diverse statistical problems but lack the \emph{two-way} reductions needed for equivalence (\cite{brennan2018reducibility,brennan2020reducibility,vafa2025symmetric,gupte2024sparse,bruna2021continuous,bangachev2024near}). Some simple equivalences follow from \cite{brennan2019universality}, but only for planted submatrix problems that differ in entry distributions rather than structure. The paper
\cite{bresler2023algorithmic} makes a significant conceptual advance by showing that random graphs with correlations in the form of added triangles constitute an equivalence class. While this result achieves the first equivalence between problems with different dependence structure, the setting remains artificial.

The motivation for this paper is the eventual goal of classification of statistical problems into equivalence classes with the implication that problems in the same class are in a precise sense fundamentally the same, rather than coincidentally similar. In this work we take a step towards such classification by proving several computational equivalence results between two of the most well-studied high-dimensional statistics problems: the Spiked Covariance Model and Spiked Wigner Model.

In the Spiked Covariance Model ($\sc$) of Johnstone and Lu \cite{johnstoneSparse04}, one observes $n$ i.i.d. samples from 
$$
\N(0,I_d + \theta u u^\top)\,,
$$
where $u\in \R^d$ represents the signal vector and $\theta\in \R$ is a signal strength parameter. The Spiked Covariance Model is a generative model for studying Principal Component Analysis (PCA), the task of finding a direction of maximum variance in data. 
PCA is an important dimensionality reduction technique and has been a cornerstone of data analysis since Pearson \cite{Pearson01111901} with modern day applications spanning the full spectrum of data science
in economics, computational biology, social science, and engineering.

In the Spiked Wigner Model ($\sw$), studied in \cite{deshpande2014information,lesieur2015phase}, one observes a $d\times d$ matrix
$$
\lambda u u^\top + W\,,
$$
where $W$ is symmetric with $\N(0,1)$ off-diagonal and $\N(0,2)$ diagonal entries, $u\in\R^d$ is a signal vector, and $\lambda\in \R$ is a signal strength parameter. The Spiked Wigner Model
captures $\mathbb{Z}_2$ synchronization \cite{javanmard2016phase}
and a Gaussian version of the stochastic block model \cite{deshpande2016asymptotic}.

For both models, it is often assumed that the signal vector $u$ is \emph{sparse}, having few nonzero entries~\cite{johnstone2009consistency}.
Sparse PCA has a vast literature and was the first well-studied high-dimensional statistics problem shown to have a computational-statistical gap \cite{berthet2013complexity}, and the sparse 
Spiked Wigner problem has seen a similarly extensive study \cite{deshpande2014information,lesieur2015phase,brennan2018reducibility,hopkins2017power}.

The relationship between the two models remains murky. While both the $\sw$ and the $\sc$ samples' empirical covariance matrix can be viewed as ``noisy observations" of the signal $u u^\top$, the entries of the former are jointly independent, whereas the entries of the latter exhibit the complex dependence structure of the Wishart distribution. This difference is more than superficial and it has led to entirely separate lines of work analyzing each. 
For example, while \cite{hopkins2017power} contains hardness against sum-of-squares algorithms for $\sw$, the analogous result for $\sc$ model remained elusive until \cite{NEURIPS2022_e85454a1} appeared five years later. As far as reductions go, there exist reductions transforming $\sw$ into $\sc$---but only in certain parameter regimes---and in the other direction there simply are no reductions known.

In this paper, we prove the first reductions transforming $\sc$ into $\sw$, and as a consequence obtain the first equivalence result between well-studied statistical problems.

\begin{theorem}[Informal Statement of Main Theorem~\ref{thm:main_equiv}]
    For a big part of their parameter ranges, $\sc$ and $\sw$ are computationally equivalent. 
\end{theorem} 

Average-case reductions changing the dependence structure are rare. As mentioned above, with the exception of \cite{bresler2023algorithmic} which does change dependence for a newly defined problem, the vast majority of reductions in the literature change the planted structure but for the most part leave the noise i.i.d. \cite{ma2015computational,hajek2015computational,brennan2018reducibility,brennan2020reducibility}. In contrast, we develop algorithmic techniques that substantially change the dependence structure of the noise while leaving the signal intact. Our main analysis tool is a novel \emph{perturbation equivariance} result for Gram-Schmidt orthogonalization.

We next define the problems that we address in this paper (Sec.~\ref{sec:intro_models}), followed by a discussion of their computational profiles in Sec.~\ref{subsec:phase_diagram}. We then introduce average-case reductions in Sec.~\ref{sec:intro_avg_case} and discuss the nuanced notion of computational equivalence for statistical problems in Sec.~\ref{sec:intro_equiv}. We present our results in Sec.~\ref{sec:intro_results} and discuss our technical contributions, including our novel techniques for structural and dependence manipulation, in Sec.~\ref{sec:ideas}.

\subsection{Two Spiked Models and Their Similarities}\label{sec:intro_models}

Both Spiked Wigner and Spiked Covariance Models have the general form of signal + noise, where the signal is a $d$-dimensional vector from an allowed collection. 
Throughout the paper we consider sparse versions of the models. We define an allowed collection of $k$-sparse $d$-dimensional unit vectors with all nonzero entries of the same magnitude as:
$$
\ksparseflat d = \b\{v\in \mathbb{R}^d, \|v\|_2=1, \|v\|_0=k, v_i \in\{\pm k^{-1/2}, 0\}
\b\}\,.
$$
Unless specified otherwise, we assume $\ksparseflat d$ as the allowed collection of signals for $\sc$ and $\sw$. For our equivalence results we rely on some prior work that requires a more relaxed version defined for some $c_{u} = O(\polylog(d))$ as
$$
\ksparse d = \b\{v\in \mathbb{R}^d, \|v\|_2=1, \|v\|_0=k, v_i \in  [-c_u k^{-1/2}, c_u k^{-1/2}]\b\}\,.
$$

\paragraph{Spiked Wigner Model, $\sw(d,k,\lambda)$.} In the $d$-dimensional Spiked Wigner Model we observe
\begin{equation}\label{e:Wigner}
    Y = \lambda u u^\top + W\,, \qquad W\sim \GOE(d)\,,
\end{equation}
where $u\in\ksparseflat d$ (or $\ksparse d$) is a $k$-sparse unit signal vector, $\lambda \in \R$ is a signal-to-noise ratio (SNR) parameter, and $\GOE(d)$ denotes the distribution of $\frac{1}{\sqrt{2}} (A+A^\top)$ with $A\sim \N(0, 1)^{\otimes d\times d}$ \cite{deshpande2014information,lesieur2015phase}. 
We denote the law of $Y$ by $\sw(d,k,\lambda)$.

\paragraph{Spiked Covariance Model, $\sc(d,k,\theta,n)$.} Also referred to as the Spiked Wishart Model or Sparse PCA, the Spiked Covariance Model of Johnstone and Lu 
\cite{johnstoneSparse04} consists of
$n$ i.i.d. samples from the $d$-dimensional multivariate Gaussian 
\begin{equation}\label{e:SCM}
\N(0,I_d+\theta uu^\top)\,,
\end{equation}
where again $u\in \ksparseflat d$ (or $\ksparse d$) is a $k$-sparse unit signal vector and $\theta$ is an SNR parameter.

Note that if we arrange the samples in the rows of matrix $Z \in \R^{n\times d}$, the spiked covariance model data has the  simple representation
\begin{equation}\label{e:SCM2}
    Z = X + \tsq g u^\top,
\end{equation}
where $X \sim \N(0, I_n)^{\otimes d}, g\sim \N(0, I_n)$ are independent. 
We denote the law of $Z$ by $\sc(d,k,\theta,n)$.

The detection and recovery problems for the Spiked Wigner and Spiked Covariance Models are formally defined in Sec.~\ref{subsec:detect_problem} and \ref{subsec:recover_problem}. Detection entails distinguishing between the null case ($\lambda=0$ for $\sw$ or $\theta=0$ for $\sc$) and the planted case ($\lambda\geq\lambda^{\star}>0$ or $\theta\geq\theta^{\star}> 0$). The recovery task is to produce estimate $\widehat{u}$ approximating the underlying sparse signal $u$.

\paragraph{Quantitative Relationship Between the Two Models.}

Historically, these problems have largely been studied independently, with separate papers studying their information theoretic limits, algorithms, and computational barriers, which we overview in Section~\ref{subsec:phase_diagram}. However, the following comparison between the empirical covariance matrix of $\sc$ samples and the $\sw$ matrix reveals quantitative similarities between the models.

Given $Z$ sampled from the Spiked Covariance Model~\eqref{e:SCM2}, and the corresponding empirical covariance matrix
$
\Sigmahat = \frac1n Z^\top Z \,,
$
consider the matrix $M=\sqrt{n}( \Sigmahat - I_d)$, where the $\sqrt{n}$ scaling matches the unit variance in the entries of the spiked Wigner model. Observe that 
$$\E M= \theta \sqrt{n}\cdot u u^\top = \E Y, \quad\text{where}\quad Y \sim \sw(d, k, \theta \sqrt{n})\,.
$$
Comparing with \eqref{e:Wigner}, we draw a natural correspondence between SNR parameters $\lambda$ and $\theta$ of the two models: 
\begin{equation}\label{e:param_corr}
  \lambda\quad \longleftrightarrow \quad \theta \sqrt{n}\,.  
\end{equation}

\subsection{Parametrization}

To analyze the complexity of the problems associated with $\sw$ and $\sc$ we make the following choice of parametrization. For $\sw$, consider $\nu = (\alpha, \beta)$ and let
$$
    \sw(\nu) := \b\{\sw(d,k_d = d^{\alpha},\lambda_d = d^{\beta})\b\}_d\,.
    $$
For $\sc$, consider $\mu = (\alpha, \beta, \gamma)$ and let 
$$
    \sc(\mu) := \b\{\sc(d, k_d = d^{\alpha}, \theta_d = d^{\beta}, n_d = d^{\gamma})\b\}_d\,.
$$ 
Here we chose $d$ as a primary scaling parameter and it turns out that expressing parameters $k,\theta,\lambda, n$ as powers $\alpha, \beta,\gamma$ of $d$ for \emph{constant} $\alpha,\beta,\gamma$, captures all complexity phase transitions in the two models up to log factors (even though such granularity ignores subpolynomial $d^{o(1)}$ factors). While one could ask for finer precision, our focus is on obtaining phase diagrams (as described below) capturing behavior at this polynomial scale. For more discussion on the choice of parametrization in statistical problems, see~\cite{bresler2023algorithmic}.

The parameter correspondence  $\lambda\leftrightarrow\theta \sqrt{n}$ in \eqref{e:param_corr} is expressed in the above scaling as
\begin{align}
    \label{eq:param}\sw\b(\nu = (\alpha, \beta + \gamma/2)\b) &\longleftrightarrow \sc\b(\mu = (\alpha, \beta, \gamma)\b)\,.
\end{align}
We call \eqref{eq:param} the \emph{canonical parameter correspondence}.
We focus on the regime $\gamma \geq 1$ where the number of samples in $\sc$  exceeds the dimension. Note that $\alpha\in (0,1)$, since $1\leq k\leq d$.

\subsection{Computational Phase Diagrams}\label{subsec:phase_diagram}

\begin{figure*}[t]
\vspace{-3mm}
\centering
\def\scalesize{0.20} 
\subfloat{
\begin{tikzpicture}[scale=\scalesize]
\tikzstyle{every node}=[font=\scriptsize]
\def\xmin{0}
\def\xmax{11}
\def\ymin{0}
\def\ymax{11}

% ##### SPCA n=d

\coordinate (A) at (0, 1.8);
\coordinate (B) at (0, 6.8);
\coordinate (C) at (5, 6.8);
\coordinate (D) at (10, 6.8);

\node at (5, 0) [below] {$\frac{1}{2}$};
\node at (10, 0) [below] {$1$};
\node at (B) [left] {$0$};
\node at (A) [left] {$-\frac12$};

\filldraw[fill=cyan, draw=cyan] (A) -- (C) -- (D) -- (A);
\filldraw[fill=green!25, draw=green] (A) -- (0, 10) -- (10, 10) -- (D) -- (C) -- (A);
\filldraw[fill=gray!25, draw=gray] (A) -- (D) -- (10, 0) -- (0, 0) -- (A);

\node at (6, 6) {H};
\node at (2, 7.5) {E};
\node at (5.8, 2.5) {I};

\draw[->] (\xmin,\ymin) -- (\xmax,\ymin) node[right] {$\alpha$};
\draw[->] (\xmin,\ymin) -- (\xmin,\ymax) node[above] {$\beta$};

\draw (A) -- (C); 
\draw (C) -- (D);
\draw (D) -- (A);

\node at (5, -4) {\textbf{(a)} $\sc(d,k, \theta,n)$ with};
\node at (5, -5.5) {$\gamma = \log_d n = 1$ and};
\node at (5, -7) {$\alpha = \log_d k, \beta = \log_d \theta$};
\end{tikzpicture}}
\subfloat{
\begin{tikzpicture}[scale=\scalesize]
\tikzstyle{every node}=[font=\scriptsize]
\def\xmin{0}
\def\xmax{11}
\def\ymin{0}
\def\ymax{11}
% ##### SPCA General

\coordinate (A) at (0, 0.8);
\coordinate (B) at (0, 5.8);
\coordinate (C) at (5, 5.8);
\coordinate (D) at (10, 5.8);

\node at (5, 0) [below] {$\frac{1}{2}$};
\node at (3, 0) [below] {$\frac{\gamma}{6}$};
\node at (10, 0) [below] {$1$};
\node at (0, 10) [left] {$0$};
\node at (B) [left] {$\frac{1-\gamma}{2}$};
\node at (A) [left] {$-\frac{\gamma}2$};

\filldraw[fill=cyan, draw=cyan] (A) -- (C) -- (D) -- (A);
\filldraw[fill=green!25, draw=green] (A) -- (0, 10) -- (10, 10) -- (D) -- (C) -- (A);
\filldraw[fill=gray!25, draw=gray] (A) -- (D) -- (10, 0) -- (0, 0) -- (A);
\filldraw[fill=black, draw=black] (C) -- (D) -- (3, 0.8 + 3/10*5) -- (3, 0.8 + 3/10*10) -- (C);

% \node at (6, 7) {H};
\node at (2, 6.5) {E};
\node at (5.8, 1.5) {I};

\node[text=white] at (6, 5) {?};

\draw[->] (\xmin,\ymin) -- (\xmax,\ymin) node[right] {$\alpha$};
\draw[->] (\xmin,\ymin) -- (\xmin,\ymax) node[above] {$\beta$};

\draw (A) -- (C); 
\draw (C) -- (D);
\draw (D) -- (A);

\node at (5, -4) {\textbf{(b)} $\sc(d, k, \theta, n)$ with};
\node at (5, -5.5) {$1 <\gamma = \log_d n <3$ and};
\node at (5, -7) {$\alpha = \log_d k, \beta = \log_d \theta$};

\end{tikzpicture}}
\subfloat{
\begin{tikzpicture}[scale=\scalesize]
\tikzstyle{every node}=[font=\scriptsize]
\def\xmin{0}
\def\xmax{11}
\def\ymin{0}
\def\ymax{11}
% ##### SPCA General

\coordinate (A) at (0, 0.8);
\coordinate (B) at (0, 5.8);
\coordinate (C) at (5, 5.8);
\coordinate (D) at (10, 5.8);

\node at (5, 0) [below] {$\frac{1}{2}$};
\node at (10, 0) [below] {$1$};
\node at (0, 10) [left] {$0$};
\node at (B) [left] {$\frac{1-\gamma}{2}$};
\node at (A) [left] {$-\frac{\gamma}2$};

\filldraw[fill=cyan, draw=cyan] (A) -- (C) -- (D) -- (A);
\filldraw[fill=green!25, draw=green] (A) -- (0, 10) -- (10, 10) -- (D) -- (C) -- (A);
\filldraw[fill=gray!25, draw=gray] (A) -- (D) -- (10, 0) -- (0, 0) -- (A);

\node at (6, 5) {H};
\node at (2, 6.5) {E};
\node at (5.8, 1.5) {I};

\draw[->] (\xmin,\ymin) -- (\xmax,\ymin) node[right] {$\alpha$};
\draw[->] (\xmin,\ymin) -- (\xmin,\ymax) node[above] {$\beta$};

\draw (A) -- (C); 
\draw (C) -- (D);
\draw (D) -- (A);

\node at (5, -4) {\textbf{(c)} $\sc(d,k, \theta,n)$ with};
\node at (5, -5.5) {$\gamma = \log_d n \geq 3$ and};
\node at (5, -7) {$\alpha = \log_d k, \beta = \log_d \theta$};

\end{tikzpicture}}
\subfloat{
\begin{tikzpicture}[scale=\scalesize]
\tikzstyle{every node}=[font=\scriptsize]
\def\xmin{0}
\def\xmax{11}
\def\ymin{0}
\def\ymax{11}

% ##### spWig

\coordinate (A) at (0, 4.2);
\coordinate (B) at (0, 9.2);
\coordinate (C) at (5, 9.2);
\coordinate (D) at (10, 9.2);

\node at (5, 0) [below] {$\frac{1}{2}$};
\node at (10, 0) [below] {$1$};
\node at (B) [left] {$\frac12$};
\node at (A) [left] {$0$};

\filldraw[fill=cyan, draw=cyan] (A) -- (C) -- (D) -- (A);
\filldraw[fill=green!25, draw=green] (A) -- (0, 10) -- (10, 10) -- (D) -- (C) -- (A);
\filldraw[fill=gray!25, draw=gray] (A) -- (D) -- (10, 0) -- (0, 0) -- (A);

\node at (6, 8.5) {H};
\node at (2, 8) {E};
\node at (5.8, 3) {I};

\draw[->] (\xmin,\ymin) -- (\xmax,\ymin) node[right] {$\alpha$};
\draw[->] (\xmin,\ymin) -- (\xmin,\ymax) node[above] {$\beta$};

\draw (A) -- (C); 
\draw (C) -- (D);
\draw (D) -- (A);

\node at (5, -4) {\textbf{(d)} $\sw(d,k, \lambda)$ with};
\node at (5, -5.5) {$\alpha = \log_d k$ and};
\node at (5, -7) {$\beta = \log_d \lambda$};
\end{tikzpicture}}
\vspace{-3mm}\caption{\scriptsize
Subset of \textbf{phase diagrams} from \cite{brennan2018reducibility}, \cite{brennan2019optimal} plotted as signal $\beta = \log_d \theta$ (or $\beta = \log_d \lambda$) vs. sparsity $\alpha = \log_d k$ characterizing feasibility of both detection and recovery.
}
\label{fig:diagrams}
\vspace{-4mm}\end{figure*}
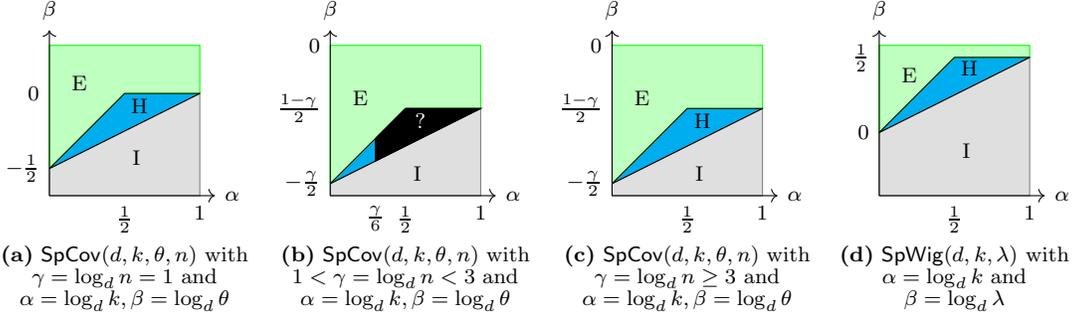

A \emph{phase diagram} captures a problem's feasibility. With every point representing complexity of a problem associated with the corresponding parameter, the diagrams divide the parameter space into the following regions:
\begin{enumerate}
    \item [(E)] Efficiently solvable (``Easy"), where there is a polynomial-time algorithm;
    \item [(H)] Computationally hard (``Hard"), where the planted clique conjecture\footnote{This conjecture is widely believed and has been extensively studied; see e.g. \cite{barak2016nearly} and references therein.} implies that there is no polynomial-time algorithm;
    \item [(I)] Information-theoretically impossible (``Impossible").
\end{enumerate}

Note that for $\sw$ the parameter space $\{\nu = (\alpha, \beta)\}$ is 2-dimensional while for $\sc$ the parameter space $\{\mu = (\alpha, \beta, \gamma)\}$ is 3-dimensional. Virtually all prior work only considered the slice $\gamma=\log_d n=1$ of the 3-dimensional phase diagram. It turns out that the relationship between the number of samples $n$ and the dimension $d$ plays an important role, and for this reason we will consider \textbf{2-dimensional slices} $\sc_{\gamma}$ of the $\sc$ diagram for all fixed values of $\gamma \geq 1$.

Complete phase diagrams are known for $\sw$ as well as for $\sc$ in case $\gamma \in\{1\}\cup[3,\infty)$, and for $\sc$ with $1<\gamma<3$ the diagrams are partially known.
Figure~\ref{fig:diagrams} shows phase diagrams for $\sw$ and for $\sc_{\gamma}$ for cases $\gamma = 1, 1< \gamma< 3,$ and $\gamma \geq 3$. We now discuss prior work that informs these diagrams, starting with the easy regimes.

\paragraph{Algorithms for Detection and Recovery (Easy Regime).}

For both $\sw$ and $\sc$, two simple classes of algorithms -- spectral and thresholding -- achieve the best performance, each for a different range of sparsities.

\emph{Spectral Algorithms:} The BBP transition for the spectrum of the $\sw$ matrix enables efficient recovery and detection for $\lambda = \tOm(\sqrt{d})$ using spectral methods \cite{baik2005phase,paul2007asymptotics,johnstone2009consistency,peche2006largest,feral2007largest,capitaine2009largest,benaych2011eigenvalues}. Separately, efficient detection and recovery spectral algorithms have been developed for $\sc$ for $\theta = \tOm(\sqrt{d}/\sqrt{n})$, where the first eigenvalue of the sample covariance matrix is well-separated \cite{ma2013sparse}, \cite{cai2013sparse}. Approximate Message Passing algorithms were shown in \cite{deshpande2014information} to be optimal for both models in a certain not-very-sparse regime where additionally $n$ is a constant factor of $d$, and spectral algorithms for both models were sharply analyzed in \cite{perry2016optimality}.

\emph{Thresholding Algorithms:} These algorithms are based on the thresholding of the largest matrix entry ($\sw$) or the largest entry of the re-scaled empirical covariance matrix ($\sc$) and succeed whenever $\lambda = \tOm(k)$, $\theta = \tOm(k/\sqrt{n})$  \cite{johnstone2009consistency,amini2009high, berthet2013complexity, berthet2013optimal,deshp2016sparse}.

We summarize the current best computational results, up to log factors, in the following table: thresholding algorithms are better for $k\leq \sqrt{d}$ and spectral for $k\geq d$.
\begin{table}[h]
\small
\centering
\renewcommand{\arraystretch}{1.5}
\begin{tabular}{|c|c|c|}
\hline
\textbf{Model} & \textbf{SNR} & \textbf{log\(_d\) SNR} \\ \hline
\(\sw(d,k,\lambda)\) & \(\lamcomp(d,k) = \min \{k, \sqrt{d}\}\) & \(\betacomp_{\lambda}(\alpha) = \min \{\alpha, 1/2\}\) \\ \hline
\(\sc(d,k,\theta,n)\) & \(\thetacomp(d,k,n) = \min \{k/\sqrt{n}, \sqrt{d/n}\}\) & \(\betacomp_{\theta}(\alpha, \gamma) = \min \{\alpha - \gamma/2, 1/2-\gamma/2\}\) \\ \hline
\end{tabular}
\vspace{-2mm}
\caption{\footnotesize Summary of SNR computational thresholds for $\sw$ and $\sc$.}
\label{tab:snr_summary}
% \vspace{-4mm}
\end{table}
The thresholds\footnote{For simplicity, we sometimes write $\thetacomp$ and $\lamcomp$ in place of $\thetacomp(d,k,n)$ and $\lamcomp(d,k)$, leaving the dependence on $d,k,n$ implicit.} $\thetacomp(d,k,n)$ and $\lamcomp(d,k)$
constitute the boundaries between ``Easy" and ``Hard" regimes in Figure~\ref{fig:diagrams}. The thresholds are consistent with parameter correspondence $\lambda\ \leftrightarrow \ \theta \sqrt{n}$~\eqref{eq:param}.

\paragraph{Statistical Limits (Impossible Regime).}

For $\sw$, it is statistically impossible to solve both detection and recovery whenever $\lambda = \tO(\sqrt{k})$ \cite{montanari2015limitation,perry2016statistical,perry2016optimality,banks2018information} and for $\sc$ whenever $\theta = \tO(\sqrt{k/n})$ \cite{amini2009high,vu2012minimax,berthet2013optimal,birnbaum2013minimax,cai2013sparse,wang2016statistical,cai2015optimal}. This constitutes the threshold between ``Hard" and ``Impossible" regimes, and we summarize it in Table~\ref{tab:stat_summary}.

\begin{table}[h]
\small
\centering
\renewcommand{\arraystretch}{1.5}
\begin{tabular}{|c|c|c|}
\hline
\textbf{Model} & \textbf{SNR} & \textbf{log\(_d\) SNR} \\ \hline
\(\sw(d,k,\lambda)\) & \(\lamstat(d,k) = \sqrt{k}\) & \(\betastat_{\lambda}(\alpha) = \alpha / 2\) \\ \hline
\(\sc(d,k,\theta,n)\) & \(\thetastat(d,k,n) = \sqrt{k/n}\) & \(\betastat_{\theta}(\alpha, \gamma) = \alpha / 2 - \gamma / 2\) \\ \hline
\end{tabular}
\vspace{-2mm}
\caption{\footnotesize Summary of SNR statistical thresholds for \(\sw\) and \(\sc\).}
\label{tab:stat_summary}
\vspace{-4mm}
\end{table}

\paragraph{Hard Regime.}

The remaining regions of the phase diagrams, which are statistically feasible but with no efficient algorithms known, are believed to be computationally hard.

In this paper we focus on the reductions approach to computational hardness. 
\cite{berthet2013complexity} showed a reduction from planted clique to a model for sparse PCA that is different from $\sc$. 
\cite{brennan2019optimal} showed a reduction to the canonical $\sc$ model via average-case reductions from $\sw$, substantiating the entire ``Hard" regime in case $\gamma \in \{1\}\cup[3,\infty]$ and the part of hard regime $\alpha \leq \gamma/6$ in case $1<\gamma<3$. For $\sw$, the entire hard region was obtained by reduction from the biclustering problem by \cite{brennan2018reducibility}.

\subsection{Average-Case Reductions}\label{sec:intro_avg_case}

Any average-case problem is associated with a distribution over instances. An \emph{average-case reduction} is an (efficient) algorithm (or Markov kernel) that maps instances of one average-case problem to another, with the requirement that the \emph{distribution} of the outputs is close in total variation to the target distribution. The existence of such reduction guarantees that any algorithm solving the target problem can be used to solve the source problem.

We formally define average-case reductions between problems like $\sc(\mu)$ and $\sw(\nu)$ for detection in Sec.~\ref{sec:avg_case} and for recovery\footnote{As explained in Sec.~\ref{sec:avg_case_recovery} we slightly restrict the allowed class of recovery algorithms to simplify the analysis.} in Sec.~\ref{sec:avg_case_recovery}. 
A key consequence of these reductions (Lem.~\ref{lem:red_detect} and \ref{lem:red_recover}) is that computational complexities are transferred: if an $O(d^{C})$-time average-case reduction
 from one point to another (e.g. from $\sc(\mu)$ to $\sw(\nu)$) is proved, then the existence of an $O(d^{M})$-time algorithm in the small neighborhood of the latter point ($\sw(\nu)$) yields an $O(d^{M} + d^C)$-time algorithm for the first point ($\sc(\mu)$). All our reductions run in $O(d^C = d^{2+\gamma})$-time.

We pay special attention to what we call ``average-case reductions at the computational threshold", which are formalized in Def.~\ref{def:avg_case_points_threshold}, Sec.~\ref{sec:avg_case}. We note that decreasing the SNR ($\theta$ or $\lambda$) makes detection and recovery harder.\footnote{This can be made formal via a reduction that simply adds noise.} This is why some of our reductions focus on points in the ``Hard'' region that are near the ``Easy" boundary, with $\theta\approx\thetacomp$, $\lambda\approx\lamcomp$: establishing hardness for those implies hardness in the entire ``Hard" region.

As described in Sections~\ref{sec:avg_case} and \ref{sec:avg_case_recovery}, average-case reductions for detection and recovery have slightly different requirements. If a reduction algorithm is shown to satisfy both, we avoid specifying if it is for detection or recovery and just call it ``an average-case reduction".

\subsection{Notions of Equivalence}\label{sec:intro_equiv}
Since the models' computational complexities are well understood for the ``Easy" and ``Impossible" regimes, we focus on the notion of computational equivalence in the ``Hard" regime.\footnote{We note that our definitions as well as equivalence results could be easily extended to the ``Easy" and ``Impossible" regimes; we choose to omit these to simplify the description.} We define the parameters inside the corresponding ``Hard" regimes as $\Omw$ and $\Omc$ with $\Omc_{\gamma}$ representing a slice of the ``Hard" region of $\sc$ at particular $\gamma$:
\begin{align*}
    \Omw &= \b\{ \nu = (\alpha, \beta): \betastat_{\lambda}(\alpha) < \beta < \betacomp_{\lambda}(\alpha) \b\}\,,\\
    \Omc &= \b\{ \mu = (\alpha, \beta, \gamma): \betastat_{\theta}(\alpha,\gamma) < \beta < \betacomp_{\theta}(\alpha, \gamma) \b\}\,,\\
    \Omc_{\gamma^{\star}} &= \b\{ \mu = (\alpha, \beta, \gamma) \in \Omc : \gamma = \gamma^{\star} \b\}\,.
    \vspace{-2mm}
\end{align*}

Establishing equivalence between $\sc$ and $\sw$ involves relating two \emph{families} of problems $\{ \sc(\mu) \}_{\mu \in \Omc}$ and $\{ \sw(\nu) \}_{\nu \in \Omw}$ or, in other words, two different phase diagrams. While \cite{bresler2023algorithmic} gave a general definition for relating different phase diagrams, we focus on a notion of \emph{bijective equivalence} between them. 
Essentially, bijective equivalence means that two phase diagrams are in one-to-one correspondence, with the computational complexities transferring between corresponding points. 

\begin{definition}[Comparison of points]
For $\mu\in \Omc$ and $\nu \in \Omw$, we write $\sc(\mu) \preceq \sw(\nu)$ if there exists an average-case reduction (Def.~\ref{def:avg_case_points}) from $\sc(\mu)$ to $\sw(\nu)$. We use the analogous definition for $\sw(\nu) \preceq \sc(\mu)$.
\end{definition}

\begin{definition}[Equivalent points]\label{def:equiv_points}
    For $\mu\in \Omc$ and $\nu \in \Omw$, we say that $\sc(\mu)$ and $\sw(\nu)$ are \textbf{computationally equivalent points}, denoted $\sc(\mu) \equiv \sw(\nu)$, if $\sc(\mu) \preceq \sw(\nu)$ and $\sw(\nu) \preceq \sc(\mu)$.
\end{definition}

We propose to aim for a bijective equivalence between the $\sw$ phase diagram and 2-dimensional \emph{slices} $\sc_{\gamma}$ of the 3-dimensional $\sc$ phase diagram. In addition to the bijective equivalence, which constitutes a deep connection between two problems,
we also define a weaker notion of equivalence, that already implies that the phase diagram of one problem is a \emph{consequence} of the phase diagram of the other.

\subsubsection{Bijective Equivalence with the Canonical Parameter Map}\label{subsec:bijective_equiv}

\begin{wrapfigure}{r}{0.45\textwidth}
\vspace{-10mm}
\centering
\def\scalesize{0.20} 
\subfloat{
\begin{tikzpicture}[scale=\scalesize]
\tikzstyle{every node}=[font=\scriptsize]
\def\xmin{0}
\def\xmax{11}
\def\ymin{0}
\def\ymax{9}

%###### SWW

\coordinate (A) at (0, 0.5);
\coordinate (C) at (5, 5.5);
\coordinate (D) at (10, 5.5);

\filldraw[fill=cyan, draw=cyan] (A) -- (C) -- (D) -- (A);
\filldraw[fill=green!25, draw=green, opacity = 0.3] (A) -- (0, 8) -- (10, 8) -- (D) -- (C) -- (A);
\filldraw[fill=gray!25, draw=gray, opacity = 0.3] (A) -- (D) -- (10, 0) -- (0, 0) -- (A);

\node at (6.6, 4.7) {H};
\node at (7, 2) {I};
\node at (2, 6) {E};

\draw[->] (\xmin,\ymin) -- (\xmax,\ymin) node[right] {$\alpha$};
\draw[->] (\xmin,\ymin) -- (\xmin,\ymax) node[right] {$\beta$};

\draw[cyan!75!black, very thick] (A) -- (C);
\draw[cyan!75!black, very thick] (C) -- (D);
\draw[cyan!75!black, very thick] (D) -- (A);

% \draw[cyan!75!black, very thick] (0, 0) -- (5, 0);
% \draw[cyan!75!black, very thick] (5, 0) -- (10, 5);
% \draw[cyan!75!black, very thick] (0, 0) -- (10, 5);
% \draw[gray] (0, 0) -- (5, 0);
% \draw[gray] (5, 0) -- (10, 5);

\node at (5, -1) {$\sc$};
\end{tikzpicture}}\quad
\subfloat{
\begin{tikzpicture}[scale=\scalesize]
\tikzstyle{every node}=[font=\scriptsize]
\def\xmin{0}
\def\xmax{11}
\def\ymin{0}
\def\ymax{9}

% ##### SPCA 

% \filldraw[fill=cyan, draw=cyan] (0, 5) -- (10, 0) -- (5, 0) -- (0, 5);
% \filldraw[fill=green!25, draw=green, opacity = 0.3] (0, 5) -- (5, 0) -- (10, 0) -- (10, -1) -- (0, -1) -- (0, 5);
% \filldraw[fill=gray!25, draw=gray, opacity = 0.3] (0, 5) -- (10, 0) -- (10, 6) -- (0, 6) -- (0, 5);

% \node at (6, 1) {H};
% \node at (2, 1) {E};
% \node at (5, 5) {I};

% \draw[->] (\xmin,\ymin) -- (\xmax,\ymin) node[right] {$\alpha$};
% \draw[->] (\xmin,\ymin) -- (\xmin,\ymax) node[right] {$\beta$};

% % \draw[gray] (5, 0) -- (10, 0); 
% % \draw[gray] (0, 5) -- (5, 0);
% \draw[cyan!75!black, very thick] (0, 5) -- (5, 0);
% \draw[cyan!75!black, very thick] (5, 0) -- (10, 0);
% \draw[cyan!75!black, very thick] (0, 5) -- (10, 0);

\coordinate (A) at (0, 2);
\coordinate (C) at (5, 7);
\coordinate (D) at (10, 7);

\filldraw[fill=cyan, draw=cyan] (A) -- (C) -- (D) -- (A);
\filldraw[fill=green!25, draw=green, opacity = 0.3] (A) -- (0, 8) -- (10, 8) -- (D) -- (C) -- (A);
\filldraw[fill=gray!25, draw=gray, opacity = 0.3] (A) -- (D) -- (10, 0) -- (0, 0) -- (A);

\node at (6, 6.3) {H};
\node at (7, 2.8) {I};
\node at (2, 7) {E};

\draw[->] (\xmin,\ymin) -- (\xmax,\ymin) node[right] {$\alpha$};
\draw[->] (\xmin,\ymin) -- (\xmin,\ymax) node[right] {$\beta$};

\draw[cyan!75!black, very thick] (A) -- (C);
\draw[cyan!75!black, very thick] (C) -- (D);
\draw[cyan!75!black, very thick] (D) -- (A);

\node at (5, -1) {$\sw$};

\end{tikzpicture}}

% Add arrows between diagrams
\begin{tikzpicture}[overlay, remember picture]
% \draw[<->, very thick, purple, bend left=25] (-2.9, 0.9) to (0.4, 0.9);
\draw[<->, >=Latex, thick, purple, bend left=35] (-2.9+0.8, 0.9+0.7) to (0.4+0.8, 0.9+0.7 + 0.33);
\end{tikzpicture}
\vspace{-8mm}
\caption{\footnotesize
Bijective equivalence
}
\label{fig:equiv1}
\vspace{-4mm}
\end{wrapfigure}
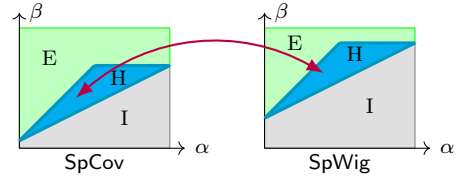

For a \emph{slice} $\sc_{\gamma}$ of the $\sc$ phase diagram at fixed $\gamma = \log_d n$, we say that $\sc_{\gamma}$ and $\sw$ are \emph{bijectively (computationally) equivalent} if there exists a bijective map $f_\gamma:\Omc_\gamma \to \Omw$ between their sets of parameters, such that $\sc(\mu)\equiv \sw(f_{\gamma}(\mu))$ for all $\mu\in \Omc_{\gamma}$.

All of our bijective equivalence results will be for the canonical parameter correspondence from \eqref{eq:param}, which is clearly a bijection. In this work, we thus say that $\sc_{\gamma}$ and $\sw$ are \textbf{bijectively equivalent} if they are bijectively equivalent with the correspondence~\eqref{eq:param}.

\subsubsection{Equivalence of Hardness Thresholds}\label{subsec:comp_thr_equiv}
 
As discussed above in Sec.~\ref{sec:intro_avg_case}, average-case reductions \emph{at the computational threshold} have complexity implications for the whole ``Hard" region. We denote the points at the computational threshold for $\sw, \sc$, and $\sc_{\gamma}$ as follows:
\begin{align*}
    \CCw &= \b\{ \nu = (\alpha, \beta): \beta = \betacomp_{\lambda}(\alpha) = \min\{\alpha, 1/2\} \b\}\,,\\
    \CCc &= \b\{ \mu = (\alpha, \beta, \gamma): \beta = \betacomp_{\theta}(\alpha, \gamma) = \min\{\alpha-\gamma/2, 1/2-\gamma/2\} \b\}\,,\\ \CCc_{\gamma^{\star}} &= \b\{ \mu = (\alpha, \beta, \gamma) \in \CCc : \gamma = \gamma^{\star} \b\}\,.
\end{align*}

We note that the definition of bijective equivalence above naturally extends to the points in $\CCw$ and $\CCc$. Here we define a weaker notion of equivalence at the computational threshold -- surjective in each direction -- that still transfers computational hardness for the whole ``Hard" regime, but does not require a bijective point equivalence between diagrams.\footnote{A toy example illustrates how this can occur. Suppose that the computational threshold for two problems is described by the ray $\{(x,y)\in \R^2: x\geq0\text{ and } y=0\}$ and that reductions existed in both directions, each under parameter correspondence $(x,y)\mapsto (x^3,y)$.} 

We say that $\sc_{\gamma}$ and $\sw$ have \textbf{equivalent computational thresholds} if
\begin{wrapfigure}{r}{0.47\textwidth}
\vspace{-2mm}
\centering
 \def\scalesize{0.20} 
\subfloat{
\begin{tikzpicture}[scale=\scalesize]
\tikzstyle{every node}=[font=\scriptsize]
\def\xmin{0}
\def\xmax{11}
\def\ymin{0}
\def\ymax{9}

%###### SWW

\coordinate (A) at (0, 0.5);
\coordinate (C) at (5, 5.5);
\coordinate (D) at (10, 5.5);

\filldraw[fill=cyan, draw=cyan] (A) -- (C) -- (D) -- (A);
\filldraw[fill=green!25, draw=green, opacity = 0.3] (A) -- (0, 8) -- (10, 8) -- (D) -- (C) -- (A);
\filldraw[fill=gray!25, draw=gray, opacity = 0.3] (A) -- (D) -- (10, 0) -- (0, 0) -- (A);

\node at (5.9, 4.5) {H};
\node at (7, 2) {I};
\node at (2, 6) {E};

\draw[->] (\xmin,\ymin) -- (\xmax,\ymin) node[right] {$\alpha$};
\draw[->] (\xmin,\ymin) -- (\xmin,\ymax) node[right] {$\beta$};

\draw[cyan!45!black, ultra thick] (A) -- (C);
\draw[cyan!45!black, ultra thick] (C) -- (D);

\node at (5, -1) {$\sc$};
\end{tikzpicture}}\quad
\subfloat{
\begin{tikzpicture}[scale=\scalesize]
\tikzstyle{every node}=[font=\scriptsize]
\def\xmin{0}
\def\xmax{11}
\def\ymin{0}
\def\ymax{9}

% ##### SPCA 

\coordinate (A) at (0, 2);
\coordinate (C) at (5, 7);
\coordinate (D) at (10, 7);

\filldraw[fill=cyan, draw=cyan] (A) -- (C) -- (D) -- (A);
\filldraw[fill=green!25, draw=green, opacity = 0.3] (A) -- (0, 8) -- (10, 8) -- (D) -- (C) -- (A);
\filldraw[fill=gray!25, draw=gray, opacity = 0.3] (A) -- (D) -- (10, 0) -- (0, 0) -- (A);

\node at (6, 6.15) {H};
\node at (7, 2.8) {I};
\node at (2, 7) {E};

\draw[->] (\xmin,\ymin) -- (\xmax,\ymin) node[right] {$\alpha$};
\draw[->] (\xmin,\ymin) -- (\xmin,\ymax) node[right] {$\beta$};

\draw[cyan!45!black, ultra thick] (A) -- (C);
\draw[cyan!45!black, ultra thick] (C) -- (D);

\node at (5, -1) {$\sw$};

\end{tikzpicture}}

% Add arrows between diagrams
\begin{tikzpicture}[overlay, remember picture]
% \draw[<->, very thick, purple, bend left=25] (-2.9, 0.9) to (0.4, 0.9);
\draw[->, >=Latex, thick, purple, bend left=30] (-2.9+1.58, 0.9+1.15) to (0.4+0.55, 0.9+0.64 + 0.38);
\draw[<-, >=Latex, thick, purple, bend left=30] (-2.9+1.5, 0.9+1.15) to (0.4+1.5, 0.9+1.45 );
\end{tikzpicture}
 \vspace{-8mm}
\caption{\footnotesize
Equiv. of hardness threshold: every point at threshold in each diagram is mapped from some point at threshold in the other diagram.
}
\label{fig:equiv2}
\vspace{-8mm}
\end{wrapfigure}
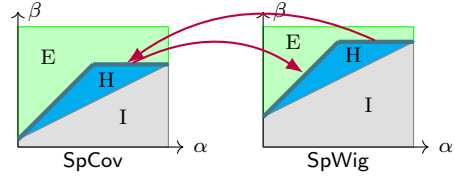
\begin{itemize}
    \vspace{-2mm}
    \item There exists a \emph{surjective average-case reduction} from the computational threshold of $\sw$ onto the entire computational threshold of $\sc_{\gamma}$. In particular, for every $\mu \in \CCc_{\gamma}$, there is a $\nu \in \CCw$, such that $\sw(\nu) \preceq \sc(\mu)$, i.e., there is an average-case reduction from $\sw(\nu)$ to $\sc(\mu)$ per Def.~\ref{def:avg_case_points_threshold}.
    % \vspace{-2mm}
    % \item There exists a \emph{surjective average-case reduction} from the computational threshold of $\sc_{\gamma}$ onto the entire computational threshold of $\sw$. In particular, for every $\nu \in \CCw$, there exists $\mu \in \CCc_{\gamma}$, such that $\sc(\mu) \preceq \sw(\nu)$.
\end{itemize}

\begin{itemize}
    \vspace{-6mm}
    \item There exists a \emph{surjective average-case reduction} from the computational threshold of $\sc_{\gamma}$ onto the entire computational threshold of $\sw$. In particular, for every $\nu \in \CCw$, there exists $\mu \in \CCc_{\gamma}$, such that $\sc(\mu) \preceq \sw(\nu)$.
\end{itemize}

\subsection{Our Results}\label{sec:intro_results}

Our goal is to establish two-way computational equivalence between $\sw$ and $\sc$ models. As discussed above, such an equivalence requires two average case reductions: one from $\sw$ to $\sc$ and the other from $\sc$ to $\sw$. 
While reductions from $\sw$ to $\sc$ exist, none are known in the reverse direction. We address this gap by constructing reductions from $\sc$ to $\sw$, establishing the strongest possible reduction with canonical parameter correspondence~\eqref{eq:param} for a large portion of the parameter space. As a consequence we obtain a surjective reduction onto the computational threshold of $\sw$ in all regimes, i.e., for arbitrary $\gamma \geq 1$. We note that our reductions either keep the sparse signal vector $u$ unchanged or up to entry sign changes, allowing us to obtain strong implications for both detection and recovery.

\begin{theorem}[Main Result 1: $\sc \to \sw$ Reductions]\label{thm:main_reduction}
    $\ $
    % \vspace{-1mm}
    \begin{enumerate}%[itemsep=5mm]
        \item \textbf{(Reductions with Canonical Parameter Correspondence)} For $\mu$ and $\nu$ in accordance to~\eqref{eq:param}, i.e, 
         $\mu = (\alpha, \beta, \gamma) \leftrightarrow \nu = (\alpha, \beta+\gamma/2)$:
         % \vspace{-2mm}
         \begin{enumerate} 
            \item There exists an $O(d^{2+\gamma})$-time average-case reduction from $\sc(\mu)$ to $\sw(\nu)$ for any $\gamma \geq 2$ and $\mu\in\Omc$ (i.e., all points in $\Omc$ for $n \gg d^2$),
            \item There exists an $O(d^{2+\gamma})$-time average-case reduction from $\sc(\mu)$ to $\sw(\nu)$ for any $\gamma\geq 1$ and $\mu\in\CCc_{\gamma}$ with $\alpha \leq 1/2$ (i.e., points on the computational threshold for $n\gg d$ and $k\leq \sqrt{d}$).     
        \end{enumerate}
        % \vspace{-4mm}
        \item \textbf{(Surjective Reduction to the Computational Threshold)} There exists a surjective average-case reduction from the computational threshold of $\sc_{\gamma}$ onto the entire computational threshold of $\sw$  for any $\gamma\geq 1$ (i.e., for all $n \gg d$). In particular, for all $\nu\in \CCw$ and $\gamma\geq 1$, there exists $\mu\in \CCc_{\gamma}$ 
        and an average-case reduction from $\sc(\mu)$ to $\sw(\nu)$.\footnote{This result utilizes existing ``reflection cloning" reductions within $\sw$ from \cite{brennan2018reducibility} which require the use of a relaxed version $\ksparse d$ for the target $\sw$ instance whenever $\alpha>1/2$ and $\gamma < 3$.}  
    \end{enumerate}    
    \end{theorem}

To show equivalence between $\sc$ and $\sw$ we combine this theorem together with existing reductions of \cite{brennan2018reducibility,brennan2019optimal} from $\sw$ to $\sc$. The details are in Section \ref{sec:spcainternal}.

% \vspace{-2mm}
\begin{theorem}[Main Result 2: Equivalence between $\sw$ and $\sc$]\label{thm:main_equiv}
    $\ $
        \begin{enumerate}
            \item (\textbf{Bijective Equivalence}) $\sw(\nu)$ and $\sc(\mu)$ are bijectively equivalent with the canonical parameter map \eqref{eq:param}
            for all $\gamma \geq 3$ and for $\gamma = 1, \alpha = 1/2$, $\beta = \betacomp$ (i.e., everywhere for $n\gg d^3$ or for $n=\Theta(d)$ at the threshold at point $ k = \Theta(d^{1/2})$).
            \item (\textbf{Equivalence of Hardness Thresholds}) $\sw(\nu)$ and $\sc(\mu)$ have equivalent computational thresholds for $\gamma = 1 \text{ or } \gamma \geq 3$ (i.e., entire computational threshold for $n=\Theta(d)$ or  $n\gg d^3$). 
        \end{enumerate}
\end{theorem}

Theorem~\ref{thm:main_equiv} shows the first computational equivalence between any two well-studied problems in high-dimensional statistics. While we do not achieve full bijective equivalence for all $\gamma$, we make progress towards this objective. The limitations of our equivalence results are largely due to the $\sw$ to $\sc$ direction, for which we use existing reductions (see Sec.~\ref{sec:spcainternal}). Future improvements depend on improved reductions from $\sw$ to $\sc$: for example, new average-case reductions from $\sw$ onto the computational threshold of $\sc_{\gamma}$ for some $1<\gamma<3$ will immediately yield equivalence of computational thresholds between $\sw$ and $\sc_{\gamma}$ when combined with our $\sc_{\gamma}$ to $\sw$ reductions. For the $\sc$ to $\sw$ reduction direction, the limiting condition $\alpha \leq 1/2$ for the general case $\gamma \geq 1$ arises from the limitations of the ``flipping” subroutine in our reduction (see Sec.~\ref{subsec:flipping1}): this subroutine effectively squares the Wigner entry-wise signal $\lambda u_i u_j$, which becomes lossy when $\alpha > 1/2$.

Our results also imply that subexponential runtimes above polynomial time are transferred \cite{ding2024subexponential,holtzman2020greedy}. Thus, bijective equivalence is meaningful independently of whether or not hardness holds at the conjectured threshold.

Moreover, our results yield the first reductions \emph{within} $\sc$ that modify the parameter $\gamma = \log_d n$, i.e., number of samples $n$ relative to the dimension $d$ (see Sec.~\ref{subsec:our_results}).

We remark that $\sw$ was shown by \cite{johnstone2015testing} to arise as the large-sample $n\to\infty$ limit (with other parameters fixed) of $\sc$. 
    %%Johnstone and Onatski 
Thus, Item 1 of Theorem~\ref{thm:main_equiv} was essentially proved for $\gamma=\infty$.

\section*{Acknowledgments}
We are grateful to Kiril Bangachev for valuable contributions during the early stages of this project. GB and AH are supported in part by NSF Grant CCF-2428619. AH was supported by Vanu Bose Presidential Fellowship.

\section{Overview of Ideas}\label{sec:ideas}

The $\sc$ and $\sw$ data distributions are different in a fundamental way. While both can be viewed as ``noisy observations" of the rescaled signal $u u^\top$, the entries of the $\sw$ matrix are jointly independent, whereas the entries of the empirical covariance matrix for $\sc$ samples exhibit the complex dependence structure of the Wishart distribution. This distinction makes devising an average-case reduction from $\sc$ to $\sw$ challenging, 
as few techniques are known for removing dependencies while preserving the underlying signal.

To address this challenge, we introduce several novel techniques for dependence removal and structure manipulation. This section overviews the main ideas.

We first use a CLT result for Wishart matrices to show that the two models' distributions are close in total variation if $n\gg d^3$, 
establishing a bijective equivalence between $\sc$ and $\sw$ in this regime.\footnote{Since the models are close in TV, no reduction is needed.} 
For the more challenging $n \ll d^3$ regime, where the models are genuinely different (their total variation distance is nearly 1), we develop new tools to remove the dependencies while preserving the signal. To start, we overcome the $n \gg d^3$ barrier by leveraging a CLT for \emph{subsets} of Wishart matrices, and build an average-case reduction with canonical parameter correspondence \eqref{eq:param} from $\sc$ to $\sw$ for all $n\gg d^2$.  

Next, we introduce a more powerful dependence-removal method based on Gram-Schmidt orthogonalization, which succeeds at the computational threshold in the range $n \gg d$ for sparse signal $k\leq \sqrt{d}$. The key challenge, described below, is showing a certain \emph{perturbation equivariance} property of Gram-Schmidt.

\subsection{Warm Up: Dependence Removal via Wishart CLT when $n\gg d^3$}

\cite{bubeck2016testing,jiang2015approximation} prove that the rescaled Wishart distribution converges to Wigner for $n\gg d^3$: for $X\in \R^{n\times d}$ with i.i.d. $\N(0, 1)$ entries and the empirical covariance matrix $\Sigmahat = \frac{1}{n} X^\top X$, if $n$ and $d$ simultaneously grow large with $n \gg d^3$, then
    \begin{gather}
        \label{e:wishart_wigner_conv}\tv\b(\sqrt{n}(\Sigmahat - I_d), W\b) \to 0\,, 
    \end{gather}
where $W\sim \GOE(d)$. This result shows that for $n\gg d^3$, the distributions of $\sw$ and $\sc$ are close in total variation in the absence of the spike (i.e. $\theta = \lambda = 0$). We extend this result in Thm.~\ref{thm:naive_reduction} to show that the \emph{spiked} models also converge in this regime: letting $Z\in \R^{n\times d}$ be data from $\sc$ in the form of~\eqref{e:SCM2} with the empirical covariance matrix $\Sigmahat_{\mathrm{sp}} = \frac{1}{n} Z^\top Z$, we have for $n\gg d^3$,
\begin{equation}\label{e:nd3}
        \TV\big( \sqrt{n}(\Sigmahat_{\mathrm{sp}} - I_d), \lambda u u^\top + W \big) \to 0 \text{ as } d\to \infty\,. 
\end{equation}
The key observation here is that 
\begin{align*}
    \Sigmahat_{\mathrm{sp}} &= \frac{1}{n}Z^\top Z = \frac{1}{n}\b( X+\tsq g u^\top \b)^\top \b( X+\tsq g u^\top \b) \\
    &= \frac{1}{n}X^\top X + \theta \frac{\|g\|^2}{n} u u^\top + \frac{1}{n}(\tsq X^\top g u^\top + \tsq u g^\top X) \approx \frac{1}{n}X^\top X + \theta u u^\top = \Sigmahat + \theta u u^\top,
\end{align*}
where $\Sigmahat = \frac{1}{n}X^\top X$ is the noise term characterized in \eqref{e:wishart_wigner_conv} and $\theta u u^\top$ is the desired sparse signal. 
The technical content of \eqref{e:nd3} is that the cross-terms $\tsq X^\top g u^\top + \tsq u g^\top X$ can be dropped. However, the limitations of this approach are apparent once we decrease the number of samples to $n \ll d^3$. In the same work, \cite{bubeck2016testing} prove that the models' noise distributions are genuinely different: for $n \ll d^3$, 
\begin{align}
\label{e:wishart_wigner_div}\tv\b(\sqrt{n}(\Sigmahat - I_d), W\b) \to 1 \text{ as } d \to \infty\,.
\end{align}

\subsection{Gaussian Cloning and Inner Products when $n\gg d^2$}

The Gaussian Cloning procedure, which utilizes the rotational invariance property of Gaussians, has been widely applied for average-case reductions \cite{brennan2019optimal,brennan2020reducibility}. In Lemma~\ref{lem:cloning-guarantee}, we describe how it can be used to create two independent copies of $\sc$ data from~\eqref{e:SCM2}. Given $
    Z = X + \tsq g u^\top
    $, 
    we obtain $
        Z\up1 = X\up1 + \tsq g u^\top / \sqrt{2}$ and $Z\up2 = X\up2 + \tsq g u^\top / \sqrt{2}
        $,
        where $X\up1, X\up2 \sim \N(0, I_n)^{\otimes d}$ are independent. That is, $Z\up1, Z\up2$ are independent samples from $ \sc(d,k,\theta/2, n)$ with same signal vectors $u, g$. 

For our average-case reduction, we consider a two-sample analogue of the empirical covariance matrix, the inner product matrix 
\begin{equation}\label{e:innerProd}
   Y = \frac{1}{\sqrt{n}}(Z\up1)^\top Z\up2 = \frac{1}{\sqrt{n}}(X\up1)^\top X\up2 + \frac \theta {2\sqrt{n}} \|g\|^2 u u^\top + \text{ cross-terms}\,.  
\end{equation}

Using the independence of $Z\up1$ and $Z\up2$ along with a recent result of Brennan et al. \cite{brennan2021finetti}, which shows that inner product matrices $\frac{1}{\sqrt{n}}(X\up1)^\top X\up2$ (the top right submatrix of a Wishart) converge to $\GOE$ already for $n\gg d^2$, we prove the average-case reduction guarantees in this regime (Section~\ref{sec:direct_reduction}).

\subsection{Gram-Schmidt Orthogonalization for Dependence Removal when $n\gg d$}

Our techniques above achieved sufficient independence for $n \gg d^2$, but remained algorithmically \emph{passive}. In contrast, our new approach described next introduces an \emph{active} ingredient -- an orthogonalization step that directly creates independence -- succeeding in the much wider range $n \gg d$.

Our approach makes use of the rotational invariance of Gaussian distributions: if $n\geq d$, $X^\top \sim N(0, I_d)^{\otimes n}$ and $\widetilde{R} \in \R^{n\times d}$ is an orthonormal matrix, then $Y = X^\top \widetilde{R} \sim N(0, I_d)^{\otimes d}$. Recall from \eqref{e:innerProd} that the inner product matrix $Y = \frac{1}{\sqrt{n}}(Z\up 1)^\top Z\up2$ of two independent sample copies preserves the desired rank-1 signal $uu^\top$, at least in expectation. 
We combine these ideas and study 
$$
Y = (Z\up 1)^\top \GS(Z\up 2)\,,
$$
where $\GS(X)$ (Algorithm~\ref{alg:GS}) denotes Gram-Schmidt orthogonalization of the column vectors of $X$. 
Note that in the absence of a spike, $Y \sim N(0, I_d)^{\otimes d}$, which is i.i.d. noise as desired.

However, the situation is much more challenging with a spike.
\textbf{Can the underlying signal $u u^\top$ survive the structural manipulation imposed by Gram-Schmidt orthogonalization?} To retain the desired signal $uu^\top$ in $Y$, the strength and sparsity of $u$ must correctly propagate from $Z\up2$ into the basis $\Zt\up 2$. Moreover, we must ensure that the presence of the signal does not interfere with independence of the noise.

If we could show that a sample $Z = X + \tsq g u^\top$ from $\sc$ yields orthonormal basis 
\begin{equation}\label{e:pertEquiv}
    \GS(Z)\approx \GS(X) + \sqrt{\frac\theta n} g u^\top,
\end{equation}
then
\begin{equation}\label{e:Y_GS}
    Y = (Z\up 1)^\top \GS(Z\up 2) \approx \underbrace{(X\up1)^\top\GS( Z\up2)}_{\text{i.i.d. noise}} + \underbrace{\frac \theta {2\sqrt{n}} \|g\|^2 u u^\top}_{\text{signal}} + \underbrace{\sqrt{\frac\theta2} u g^\top \GS(X\up 2)}_{\text{cross-term}}\,,
\end{equation}
addressing both issues.

\begin{wrapfigure}{r}{0.50\textwidth}
\vspace{-2mm}
\centering
 \tikzset{every picture/.style={line width=0.75pt}} %set default line width to 0.75pt        

\begin{tikzpicture}[x=0.75pt,y=0.75pt,yscale=-1,xscale=1]

\def\xposa{160} %left x coord
\def\xposb{325} %right x coord
\def\yposa{120} %bottom y coord
\def\yposb{70} % top y coord

%Straight Lines [id:da9161277204453793] 
\draw    (\xposa,\yposb) -- (\xposa,\yposa) ;
\draw [shift={(\xposa,\yposa)}, rotate = 270] [color={rgb, 255:red, 0; green, 0; blue, 0 }  ][line width=0.75]    (10.93,-3.29) .. controls (6.95,-1.4) and (3.31,-0.3) .. (0,0) .. controls (3.31,0.3) and (6.95,1.4) .. (10.93,3.29)   ;
%Straight Lines [id:da6096662003156208] 
\draw    (\xposb,\yposb) -- (\xposb,\yposa) ;
\draw [shift={(\xposb,\yposa)}, rotate = 270] [color={rgb, 255:red, 0; green, 0; blue, 0 }  ][line width=0.75]    (10.93,-3.29) .. controls (6.95,-1.4) and (3.31,-0.3) .. (0,0) .. controls (3.31,0.3) and (6.95,1.4) .. (10.93,3.29)   ;

%Top Horiz. Line  
\def\yshift{12}
\draw    (\xposa+50,\yposb-\yshift) -- (\xposb-30,\yposb-\yshift) ;
\draw [shift={(\xposb-30,\yposb-\yshift)}, rotate = 180] [color={rgb, 255:red, 0; green, 0; blue, 0 }  ][line width=0.75]    (10.93,-3.29) .. controls (6.95,-1.4) and (3.31,-0.3) .. (0,0) .. controls (3.31,0.3) and (6.95,1.4) .. (10.93,3.29)   ;

%Bottom Horizontal Line 
\def\lineshift{13}
\draw    (\xposa+52,\yposa+\lineshift) -- (\xposb-28,\yposa+\lineshift) ;
\draw [shift={(\xposb-28,\yposa+\lineshift)}, rotate = 180] [color={rgb, 255:red, 0; green, 0; blue, 0 }  ][line width=0.75]    (10.93,-3.29) .. controls (6.95,-1.4) and (3.31,-0.3) .. (0,0) .. controls (3.31,0.3) and (6.95,1.4) .. (10.93,3.29)   ;

% Text Node

% \node at (\xposa-40,\yposb-10) {$\sc_0$};

\draw (\xposa-40,\yposb-20) node [anchor=north west][inner sep=0.75pt]   [align=left] {$\displaystyle \sc(\theta=0)$};
% Text Node
\draw (\xposb-25,\yposb-20) node [anchor=north west][inner sep=0.75pt]   [align=left] {$\displaystyle \sc({\theta})$};
% Text Node
\draw (\xposa-65,\yposa+5) node [anchor=north west][inner sep=0.75pt]   [align=left] {$\displaystyle \GS\b(\sc(\theta={0})\b)$};
% Text Node
\draw (\xposb-25,\yposa+5) node [anchor=north west][inner sep=0.75pt]   [align=left] {$\displaystyle \GS\b(\sc({\theta})\b)$};
% Text Node
\draw (\xposa+4,90) node [anchor=north west][inner sep=0.75pt]   [align=left] {$\GS$};
% Text Node
\draw (\xposb+4,90) node [anchor=north west][inner sep=0.75pt]   [align=left] {$\GS$};
% Stuff under upper horizontal line
\draw (215,\yposb-8) node [anchor=north west][inner sep=0.75pt]   [align=left] {\small $+\sqrt{\theta} gu^\top$};
% Stuff under lower horizontal line
\draw (210,\yposa-8) node [anchor=north west][inner sep=0.75pt]   [align=left] {\small $+\sqrt{\theta/n} gu^\top$};
\end{tikzpicture}
 \vspace{-3mm}
\caption{\footnotesize
Perturbation equivariance of $\GS$. The down then right path corresponds to the RHS of \eqref{e:pertEquiv}, whereas right then down corresponds to LHS of~\eqref{e:pertEquiv}.}
\label{fig:equiv1}
\vspace{-3mm}
\end{wrapfigure}

Eq.~\eqref{e:pertEquiv} is a certain \emph{perturbation equivariance}.  
We want $\GS$ to satisfy two competing desiderata: (1) Wishart dependence is removed, just as when the signal is not present and (2) The signal, which itself is a form of dependence, is not damaged in the reduction.
This is captured by a \emph{commuting diagram} (Figure~\ref{fig:equiv1}), similar to that in \cite{bresler2023algorithmic}.

For an iterative process like Gram-Schmidt, such robustness to perturbation in the form of planted signal is nontrivial: the signal in early vectors affects the subsequent ones, distorting the planted signal. In Lemma~\ref{lem:gZ_strong} we establish a key perturbation result that characterizes the planted signal's behavior under Gram-Schmidt and makes \eqref{e:pertEquiv} precise.

\subsection{Two Novel Procedures: ``Flipping" and Bernoulli Denoising}\label{subsec:flipping1}

\paragraph{``Flipping".} Our matrix $Y$ from \eqref{e:Y_GS} and the $\sw$ distribution differ due to the cross-term.
Observe that the cross-term $u g^\top \GS(X\up2)$ is sparse only in the rows \emph{but not the columns}. To obtain the $\sw$ signal structure, which is sparse in both rows and columns, we propose the following procedure we call ``flipping".

Let $\dot{Z}\up {1}$ and $\ddot{Z}\up {1}$ be the outputs of the cloning procedure applied to $Z\up 1$.
We create two copies of the projection of interest, $Y \up1 = (\dot{Z}\up {1})^\top \GS(Z\up 2)$ and $Y\up 2 = (\ddot{Z}\up {1})^\top \GS(Z\up 2)$. 
Notice that the cross-term $u g^\top \GS(X\up2)$ of $Y\up 1$ is sparse in the rows and the cross-term $\GS(X\up2)^\top g u^\top $ in $\b(Y\up 2\b)^\top$, i.e., the flipped copy, is sparse in the columns. We consider 
$$
Y = \sign \b( Y\up 1 \odot (Y\up 2)^\top \b)\,,
$$
where $\odot$ denotes the entrywise multiplication. 
Observe that conditioned on $u$ and $g$, $Y$ has jointly independent entries. Because $\sign(\N(a,1)\cdot \N(0,1))\sim \textsf{Unif}(\{\pm1\})$ for any $a$, the signal in $Y$ is sparse in both directions. In particular, only entries $Y_{ij}$, where both $i,j \in \supp(u)$, contain the signal.

This ``flipping" procedure is a subroutine of our orthogonalization reduction in Sec.~\ref{sec:gs_reduction}, and is analyzed as part of Theorem~\ref{t:wishSSBM}.

\paragraph{Bernoulli Denoising.} Following the flipping procedure, we use our novel Bernoulli Denoising step described in Section~\ref{subsec:bern_denoise}, which allows us to remove unwanted perturbations in the means of the entries of $Y$ that contain the signal.

Notice that our flipping procedure involves discretization of $Y$ to $\pm 1$ values. It turns out that we can show that entries $Y_{ij}$ for $i,j\in\supp(u)$ are of the form $\rad(a+\Delta)$, where $\rad(t)$ denotes a random variable taking $\pm1$ values such that $\E \rad(t) = t$. Here $\Delta \ll a$ represents an undesired noise whose size we do not know and $a$ is the desired signal level. We propose a denoising procedure designed to reduce $\Delta$ relative to $a$ when given multiple samples from $\rad(a+\Delta)$. The main idea is to construct a polynomial of the noisy samples and fresh $\rad(a)$ samples, such that the resulting variable is distributed as $\Rad\b(C(a^M - (-\Delta)^{M})\b)$ for some constant $C$ and power $M$ depending on the number of available noisy samples:
$$
\text{poly}\b(\underbrace{\rad(a+\Delta),\dots,\rad(a+\Delta)}_{N\text{ noisy inputs}},\underbrace{\rad(a),\dots,\rad(a)}_{\text{fresh r.v.s}}\b) \sim \Rad\b(C(a^M - (-\Delta)^{M})\b)\,.
$$
Then, in a setting when $a \approx 1$ and $\Delta \ll 1$, this procedure decreases the undesired noise relative to the constant mean: $$\B|\frac{\Delta}{a}\B|^{M} \ll \frac{\Delta}{a}\,.$$ 
Equally important, it keeps the signal $a\approx a^M\approx 1$ almost the same. The algorithm is outlined as Algorithm~\ref{alg:denoise} with its guarantees stated in Lemma~\ref{lem:denoise_guarantee}.

The output of these procedures -- flipping and Bernoulli Denoising -- is a matrix of $\pm 1$ entries that resembles the spiked Wigner matrix in that it has independent entries and correct sparse signal strength and structure. Our final step -- Gaussianization of matrix entries, by now a standard procedure in the average-case reductions literature (see Section~\ref{subsec:gaussianize}) -- lifts the entries of $Y$ from discrete $\pm 1$ values to Gaussian-distributed entries while retaining the signal in the entry means, concluding our reduction in Section~\ref{sec:gs_reduction}.

\section{Notation and Organization of Paper}
\subsection{Notation}

For matrices $A$ and $B$ of the same dimension, $A \odot B $ denotes entry-wise matrix multiplication. For matrix $A$ we denote by $A_i$ the $i$th column of $A$. 

We write $\cL(X)$ for the probability law of a random variable $X$, and for two random variables $X$ and $Y$ we often use the shorthand $\TV(X,Y)= \TV(\cL(X),\cL(Y))$. For two random variables $X$ and $Y$ we sometimes use the shorthand $X\approx_{o(1)} Y$ in place of $\TV(X,Y)=o(1)$.

For a vector $v\in \R^n$, $\|v\|=\|v\|_2=\sqrt{\sum_i v_i^2}$ denotes the Euclidean norm.
For a matrix $A\in \R^{n\times n}$, $\|A\|_\FR = \sqrt{\sum_{ij} A_{ij}^2}$ denotes the Frobenius norm.

For $a\in [-1,1]$, we denote by $\rad(a)$ the law of random variable $X\in \{-1,1\}$ with $\E X =a$. We denote $X\sim \N(0,I_d)^{\otimes n}$ as $n$ i.i.d. samples from $\N(0,I_d)$ arranged in the columns of $X\in \R^{d\times n}$.

We sometimes write $\tO(A)$ (corresp. $\tOm(A)$) to denote that the expression is upper (corresp. lower) bounded by $c_1 (\log n)^{c_2} A$ for some known universal constants $c_1, c_2$. For all such bounds the expressions for $c_1, c_2$ are explicitly derived and we only sometimes resort to $\tO, \tOm$ notation to facilitate understanding.

We write $a_n\gg b_n$ to express that $\lim_{n\to\infty}b_n/a_n = 0$.

\subsection{Organization of Rest of the Paper}
The rest of the paper is organized as follows: in Sec.~\ref{sec:spcainternal} we overview existing reductions for $\sc,\sw$ and describe our results; in Sec.~\ref{sec:prelims} and \ref{sec:prelims_new} we describe new and known subroutines for our reductions; in Sec.~\ref{sec:naive_reduction} we prove the total variation equivalence between $\sc$ and $\sw$ in regime $n\gg d^3$; in Sec.~\ref{sec:direct_reduction} we show a reduction from $\sc$ to $\sw$ in regime $n \gg d^2$; in Sec.~\ref{sec:gs_reduction} we show a reduction from $\sc$ to $\sw$ in regime $n \gg d$ via Gram-Schmidt; in Sec.~\ref{sec:gs_perturb} we establish the perturbation equivariance result for the Gram-Schmidt algorithm; formal definitions of average-case reductions for detection and recovery are deferred to Sec.~\ref{sec:avg_case} and \ref{sec:avg_case_recovery}.

\section{Reductions for $\sc$ and $\sw$: Known Results and Our Contributions}\label{sec:spcainternal}

In this section we overview existing internal reductions within $\sc$ and $\sw$ models (Section~\ref{subsec:internal_sparse}, \ref{subsec:internal_dense}) as well as reductions from $\sw$ to $\sc$ (Section~\ref{subsec:sw_to_sc_known}). Then in Section~\ref{subsec:our_results} we describe how to combine these with our novel reductions from $\sc$ to $\sw$ to obtain our main results on equivalence.

\subsection{Internal Reductions within Spiked Models}

\begin{figure*}[h]
\vspace{-3mm}
\centering
\def\scalesize{0.20} 
\subfloat{
\begin{tikzpicture}[scale=\scalesize]
\tikzstyle{every node}=[font=\scriptsize]
\def\xmin{0}
\def\xmax{11}
\def\ymin{0}
\def\ymax{11}
% ##### SPCA General

\coordinate (A) at (0, 0.8);
\coordinate (B) at (0, 5.8);
\coordinate (C) at (5, 5.8);
\coordinate (D) at (10, 5.8);

\node at (5, 0) [below] {$\frac{1}{2}$};
\node at (10, 0) [below] {$1$};
\node at (0, 10) [left] {$0$};
\node at (B) [left] {$\frac{1-\gamma}{2}$};
\node at (A) [left] {$-\frac{\gamma}2$};

\filldraw[fill=cyan, draw=cyan, opacity=0.7] (A) -- (C) -- (D) -- (A);
\filldraw[fill=green!25, draw=green, opacity=0.7] (A) -- (0, 10) -- (10, 10) -- (D) -- (C) -- (A);
\filldraw[fill=gray!25, draw=gray, opacity=0.7] (A) -- (D) -- (10, 0) -- (0, 0) -- (A);

\node at (6, 4.8) {H};
\node at (2, 6.5) {E};
\node at (5.8, 1.5) {I};

\draw[->] (\xmin,\ymin) -- (\xmax,\ymin) node[right] {$\alpha$};
\draw[->] (\xmin,\ymin) -- (\xmin,\ymax) node[above] {$\beta$};

\draw (A) -- (C); 
\draw (C) -- (D);
\draw (D) -- (A);

\draw[->, >=Latex, very thick, purple] (C) -- (2, 5.8-3); 
\draw[->, >=Latex, very thick, purple] (C) -- (8, 5.8); 

\node at (5, -4) {\textbf{(a)} $\sc_{\gamma}(d,k, \theta,n)$ with};
\node at (5, -5.5) {$(\alpha,\beta,\gamma)= (\log_d n,\log_d k,\log_d \theta)$};
% \node at (5, -7) {$\alpha = \log_d k, \beta = \log_d \theta$};

\end{tikzpicture}}\qquad\qquad
\subfloat{
\begin{tikzpicture}[scale=\scalesize]
\tikzstyle{every node}=[font=\scriptsize]
\def\xmin{0}
\def\xmax{11}
\def\ymin{0}
\def\ymax{11}

% ##### spWig

\coordinate (A) at (0, 4.2);
\coordinate (B) at (0, 9.2);
\coordinate (C) at (5, 9.2);
\coordinate (D) at (10, 9.2);

\node at (5, 0) [below] {$\frac{1}{2}$};
\node at (10, 0) [below] {$1$};
\node at (B) [left] {$\frac12$};
\node at (A) [left] {$0$};

\filldraw[fill=cyan, draw=cyan, opacity=0.7] (A) -- (C) -- (D) -- (A);
\filldraw[fill=green!25, draw=green, opacity=0.7] (A) -- (0, 10) -- (10, 10) -- (D) -- (C) -- (A);
\filldraw[fill=gray!25, draw=gray, opacity=0.7] (A) -- (D) -- (10, 0) -- (0, 0) -- (A);

\node at (5.6, 8) {H};
\node at (2, 8) {E};
\node at (5.8, 3) {I};

\draw[->] (\xmin,\ymin) -- (\xmax,\ymin) node[right] {$\alpha$};
\draw[->] (\xmin,\ymin) -- (\xmin,\ymax) node[above] {$\beta$};

\draw (A) -- (C); 
\draw (C) -- (D);
\draw (D) -- (A);

\draw[->, >=Latex, very thick, purple] (C) -- (2, 9.2-3); 
\draw[->, >=Latex, very thick, purple] (C) -- (8, 9.2); 

\node at (5, -4) {\textbf{(b)} $\sw(d,k, \lambda)$ with};
\node at (5, -5.5) {$(\alpha, \beta) = (\log_d k, \log_d \lambda)$};
% \node at (5, -7) {$\beta = \log_d \lambda$};
\end{tikzpicture}}
\vspace{-3mm}\caption{\scriptsize
Internal reductions within $\sc$ and $\sw$: the central point is the \emph{easiest} among those at the computational threshold and reductions map in the directions of the red arrows to all other points on the computational threshold.  Thus, hardness for the entire hard region follows from hardness of the center point.
}
\label{fig:internal_diagrams}
% \vspace{-4mm}
\end{figure*}
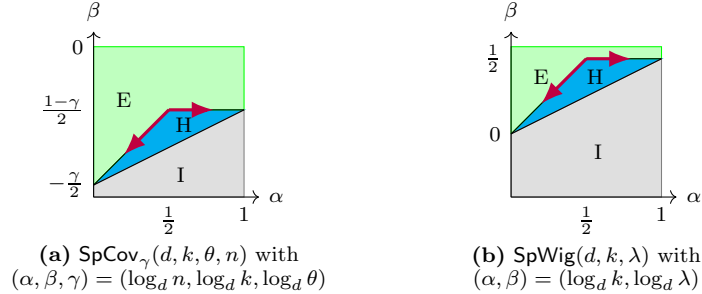

As can be seen in their phase diagrams, both $\sc$ and $\sw$ exhibit a certain change in behavior at $\alpha = 1/2$ (i.e. $k = \Theta(\sqrt{d})$), splitting the computational threshold into two parts. It turns out that for $\sw$ and every fixed slice $\sc_{\gamma}$ of $\sc$, the central point $\alpha = 1/2$ can be shown to be the ``computationally easiest" among those on the threshold. In particular, the results in Lemmas~\ref{lem:spca_internal_subsample} and~\ref{lem:spca_internal_reflection} below can be summarized as follows (see Fig.~\ref{fig:internal_diagrams}):
\begin{enumerate}
\item Let $\nu^\star=(1/2, 1/2)\in \CCw$. Then for every $\nu\in\CCw$,
$$\sw(\nu^\star) \preceq \sw(\nu)\,.$$
\item For every $\gamma\geq 1$, let $\mu_{\gamma}^\star=(1/2, 1/2-\gamma/2,\gamma/2)\in \CCc_{\gamma}$. Then for every $\mu\in\CCc_{\gamma}$,
$$\sw(\mu_{\gamma}^\star) \preceq \sw(\mu)\,.$$

\end{enumerate}

\subsubsection{Changing Sparsity in the More Sparse Regime $k\leq \sqrt{d}$}\label{subsec:internal_sparse}

The following internal reductions for $\sc$ were introduced and analyzed in \cite{brennan2019optimal}. The first one can be easily extended to $\sw$. We present proof sketches for each of these results.

\begin{lemma}[Internal Reduction for both $\sc$, $\sw$]\label{lem:spca_internal_subsample}
$\ $
    \begin{enumerate}
        \item [$\sc$] There exists an average-case reduction for both detection and recovery from $\sc(d,k,\theta,n)$ to $\sc(d,k/2,\theta/2,n)$. Consequently, $\sc(\alpha, \beta, \gamma) \preceq \sc \b(\alpha-t, \beta-t, \gamma\b)$ for any $0<t\leq\alpha$.
        \item [$\sw$] There exists an average-case reduction for both detection and recovery from $\sw(d,k,\lambda)$ to $\sc(d,k/2,\lambda/2)$. Consequently, $\sw(\alpha, \beta) \preceq \sw\b(\alpha-t, \beta-t\b)$ for any $0<t\leq\alpha$.
    \end{enumerate}
    Importantly, for $\alpha\leq 1/2$, if the $\sc$ parameter $\mu=(\alpha,\beta,\gamma)$ is on the computational threshold $\CCc_{\gamma}$, i.e., $\beta = \alpha - \gamma/2$, then $\mu_t = (\alpha-t,\beta-t,\gamma)$ is also on the computational threshold $\CCc_{\gamma}$. Analogous statement holds for $\sw$. 
    
\end{lemma}
\begin{proof}
    Given $Z \in \R^{n\times d}$ -- samples from Spiked Covariance Model in form \eqref{e:SCM2}, the reduction retains every column with probability $1/2$ and replaces it with $N(0,I_n)$ otherwise. This preserves parameters $d$ and $n$, but decreases parameters $k, \theta$ by a factor of $\approx 2$.

    Similarly, given a sample $Y\in \R^{d \times d}$ from the Spiked Wigner Model, the reduction retains every row and column $i$ with probability $1/2$ and replaces all entries in row and column $i$ with i.i.d. $N(0,1)$ symmetrically otherwise. This preserves parameter $d$, but decreases parameters $k, \lambda$ by a factor of $\approx 2$. For details, see \cite{brennan2019optimal}.
\end{proof}

\begin{lemma}[Internal Reduction for $\sc$]\label{lem:spca_internal_noise}
There exists an average-case reduction for both detection and recovery from $\sc(d,k,\theta,n)$ to $\sc(2d,k,\theta,n)$. It follows that $\sc(\alpha, \beta, \gamma) \preceq \sc \b(\alpha/t, \beta/t, \gamma/t\b)$ for any constant $0< t$.

    Importantly, for $\alpha \leq 1/2$, if $\mu=(\alpha, \beta,\gamma)$ is on the computational threshold, i.e., $\beta = \alpha - \gamma/2$, then $\mu_t = (\alpha/t,\beta/t,\gamma/t)$ is also on the computational threshold. 
\end{lemma}
\begin{proof}
    Given $Z \in \R^{n\times d}$ -- samples from $\sc$ in form \eqref{e:SCM2}, the reduction adds $d$ new columns to $Z$, each sampled from $N(0,I_n)$ and then randomly permutes the columns. This preserves parameters $k, \theta, n, u$ and increases the parameter $d$ by a factor~of~2.
\end{proof}

\subsubsection{Changing Sparsity in the More Dense Regime $k\geq \sqrt{d}$}\label{subsec:internal_dense}

The following ``reflection cloning" reduction for $\sw$ was introduced in \cite{brennan2018reducibility} and can be easily generalized to an internal reduction for $\sc$. We note that it requires the use of $\ksparse d$ (rather than $\ksparseflat d$) as the collection of allowed sparse vectors for the target instance of the reduction. 

\begin{lemma}[Internal Reduction for both $\sc$ and $\sw$] \label{lem:spca_internal_reflection}
    $\ $
    \begin{enumerate}
        \item [$\sc$] There exists an average-case reduction for both detection and recovery from $\sc(d,k,\theta,n)$ to $\sc(d,2k,\theta,n)$. Hence,
        $\sc(\alpha, \beta, \gamma) \preceq \sc \b(\alpha+t , \beta, \gamma\b)$ for any $t > 0$.
        \item [$\sw$] There exists an average-case reduction for both detection and recovery from $\sw(d,k,\lambda)$ to $\sw(d,2k,\lambda)$. Hence, 
        $\sw(\alpha, \beta) \preceq \sw\b(\alpha+t, \beta\b)$ for any $t > 0$.  
    \end{enumerate}
    Importantly, for $\alpha\geq 1/2$, if the $\sc$ parameter $\mu=(\alpha, \beta, \gamma)$ is on the computational threshold $\CCc_{\gamma}$, i.e., $\beta = 1/2 - \gamma/2$, then $\mu_t = (\alpha+t,\beta,\gamma)$ is also on the computational threshold. The analogous statement holds for $\sw$.
\end{lemma}
\begin{proof}
    Given $Z \in \R^{n\times d}$ -- samples from Spiked Covariance Model in form \eqref{e:SCM2}, the reduction splits our matrix into two parts $A, B \in R^{n\times (d/2)}$ and recombines them as $\frac{A+B}{\sqrt{2}}, \frac{A-B}{\sqrt{2}}$, i.e., does the ``reflection" step. For spiked Wigner, a similar reflection step is done twice -- along both the $x$ and $y$-axes. The analysis of the Spiked Covariance version is completely analogous to that of the Spiked Wigner version in \cite{brennan2018reducibility}.
\end{proof}

\subsection{Overview of Existing $\sw \to \sc$ Reductions}\label{subsec:sw_to_sc_known}

We summarize the state-of-the-art average-case reductions from $\sw$ to $\sc$ (\cite{brennan2018reducibility,brennan2019optimal}) as follows: 
\begin{theorem}[Prior Average-Case Reductions From $\sw \to \sc$]\label{thm:wig_to_cov}
    $\ $
    \begin{enumerate}
        \item \textbf{(Reduction with Canonical Parameter Correspondence)} For $\mu$ and $\nu$ in accordance to \eqref{eq:param}, i.e., $\mu = (\alpha, \beta, \gamma)\leftrightarrow \nu = (\alpha, \beta+\gamma/2)$:  
        \begin{enumerate}
             \item There exists an average-case reduction from $\sw(\nu)$ to $\sc(\mu)$ for all $\nu\in \Omw$ and $\gamma \geq 3$ (i.e., whole ``hard regime" for $n\gg d^3$).
            \item There exists an average-case reduction from $\sw(\nu)$ to $\sc(\mu)$ for $\gamma = 1$ and any $\nu \in \CCw$ with $\alpha\geq 1/2$ (i.e., computational threshold of the ``the spectral regime" $k\geq \sqrt{d}$ for $n = \Theta(d)$),
        \end{enumerate}
    \item \textbf{(Reduction to the Computational Threshold)} There exists an average-case reduction from the computational threshold $\sw(\nu)$ (i.e., for some $\nu\in \CCw$) to $\sc(\mu )$ for the following regimes of $\mu \in \CCc_\gamma$:
        \begin{enumerate}
            \item $\gamma = 1 \text{ or } \gamma \geq 3$, $\alpha \in (0,1)$ (i.e., whole computational threshold
        for $n=\Theta(d) \text{ or } n\gg d^3$).
            \item $1 < \gamma < 3$ and $\alpha \leq \gamma/6$ (i.e., a part of the computational threshold, where $k\leq d^{\gamma/6}$).
        \end{enumerate}
      \end{enumerate}
\end{theorem}

\subsection{Our Results: $\sc \to \sw$ Reductions and Equivalence}\label{subsec:our_results}

All our reductions are aligned with the canonical parameter correspondence \eqref{eq:param} 
$$
\mu = (\alpha, \beta, \gamma) \quad\longleftrightarrow \quad \nu = f(\mu)=(\alpha, \beta+\gamma/2)\,.
$$
In Section~\ref{sec:direct_reduction} we give a reduction from $\sc(\mu)$ to $\sw(\nu = f(\mu))$ for all $\gamma \geq 2, \mu \in \Omc_{\gamma}, \nu \in \Omw$ (see Corollary~\ref{cor:clone_reduction}), establishing the strongest possible reduction in this regime. In Section~\ref{sec:gs_reduction} we give a bijective reduction from $\sc(\mu)$ to $\sw(\nu = f(\mu))$ for all $\gamma \geq 1, \alpha \leq 1/2, \mu \in \CCc_{\gamma}, \nu \in \CCw$ (see Corollary~\ref{cor:gs_reduction}), establishing a bijective reduction on the computational threshold for $\alpha \leq 1/2$. Combined with the internal $\sw$ reduction (Lemma~\ref{lem:spca_internal_reflection}) this gives a surjective reduction onto the whole computational threshold of $\sw$ in the general case $\gamma\geq 1$. We summarize these results in the following Theorem:

\begin{mythm}{\ref{thm:main_reduction}}\textnormal{(Main Result 1: $\sc \to \sw$ Reductions)}
    $\ $
    \begin{enumerate}
        \item \textbf{(Reductions with Canonical Parameter Correspondence)} For $\mu$ and $\nu$ in accordance to~\eqref{eq:param}, i.e, 
         $\mu = (\alpha, \beta, \gamma) \leftrightarrow \nu = (\alpha, \beta+\gamma/2)$:
         \begin{enumerate} 
            \item There exists an $O(d^{2+\gamma})$-time average-case reduction from $\sc(\mu)$ to $\sw(\nu)$ for any $\gamma \geq 2$ and $\mu\in\Omc$ (i.e., all points in $\Omc$ for $n \gg d^2$),
            \item There exists an $O(d^{2+\gamma})$-time average-case reduction from $\sc(\mu)$ to $\sw(\nu)$ for any $\gamma\geq 1$ and $\mu\in\CCc_{\gamma}$ with $\alpha \leq 1/2$ (i.e., points on the computational threshold for $n\gg d$ and $k\leq \sqrt{d}$). 
            
        \end{enumerate}
        \item \textbf{(Surjective Reduction to the Computational Threshold)} There exists a surjective average-case reduction from the computational threshold of $\sc_{\gamma}$ onto the entire computational threshold of $\sw$  for any $\gamma\geq 1$ (i.e., for all $n \gg d$). In particular, for all $\nu\in \CCw$ and $\gamma\geq 1$, there exists $\mu\in \CCc_{\gamma}$ 
        and an average-case reduction from $\sc(\mu)$ to $\sw(\nu)$.\footnote{This result utilizes existing ``reflection cloning" reductions within $\sw$ from \cite{brennan2018reducibility} which require the use of a relaxed version $\ksparse d$ for the target $\sw$ instance whenever $\alpha>1/2$ and $\gamma<3$.}  
    \end{enumerate}    \end{mythm}
    
To establish equivalence between $\sc$ and $\sw$ we combine our result with existing $\sw \to \sc$ reductions in Theorem~\ref{thm:wig_to_cov}. Observe that our equivalence results are limited by incompleteness of existing $\sw \to \sc$ reductions and any improvement in this direction would immediately yield even stronger equivalence results.

    \begin{mythm}{\ref{thm:main_equiv}}\textnormal{(Main Result 2: Equivalence between $\sw$ and $\sc$)}
    $\ $
        \begin{enumerate}
            \item (\textbf{Bijective Equivalence}) $\sw(\nu)$ and $\sc(\mu)$ are bijectively equivalent with the canonical parameter map \eqref{eq:param}
            for all $\gamma \geq 3$ and for $\gamma = 1, \alpha = 1/2$, $\beta = \betacomp$ (i.e., everywhere for $n\gg d^3$ or for $n=\Theta(d)$ at the threshold at point $ k = \Theta(d^{1/2})$).
            \item (\textbf{Equivalence of Hardness Thresholds}) $\sw(\nu)$ and $\sc(\mu)$ have equivalent computational thresholds for $\gamma = 1 \text{ or } \gamma \geq 3$ (i.e., entire computational threshold for $n=\Theta(d)$ or  $n\gg d^3$).\footnote{To establish the hardness threshold equivalence for $\gamma=1$ we use a reduction from $\sc$ to $\sw$. This requires the target instance of $\sw$ to have a signal from collection $\ksparse d$ as opposed to $\ksparseflat d$. We believe this to be a technical challenge rather than a conceptual problem.}
        \end{enumerate}

    \end{mythm}

Finally, we note that our results yield the first reductions \emph{within} $\sc$ that modify the parameter $\gamma = \log_d n$, which controls the sample size $n$ relative to the dimension $d$. Specifically, we show that 
\begin{enumerate}
    \item The following sets of parameters form \textbf{computational equivalence classes}:
    \begin{enumerate}
        \item $\mathcal{S}_{\alpha} = \b\{ \mu = (\alpha, \betacomp_{\theta}(\alpha, \gamma), \gamma)\b\}_{\gamma\geq 3}$ for every $\alpha\in (0,1)$ (that is, for all $\alpha \in (0,1)$ and $\mu, \mu' \in \mathcal{S}_{\alpha}$, we have $\sc(\mu) \equiv \sc(\mu')$).
        \item $\mathcal{S}_{\alpha} = \b\{ \mu = (\alpha, \betacomp_{\theta}(\alpha, \gamma), \gamma)\b\}_{\gamma \in \{1\}\cup[3,\infty)}$ for every $\alpha\in [0.5,1)$.
    \end{enumerate}
    \item Slices of $\sc$ for $\gamma=1, \gamma\geq 3$ have \textbf{pairwise equivalent computational thresholds}.
\end{enumerate}

\section{Preliminaries}\label{sec:prelims}
\subsection{Properties of Total Variation} The analysis of our reductions will make use of the following well-known facts and inequalities concerning total variation distance (see, e.g., \cite{Polyanskiy_Wu_2025}). Recall that the total variation distance between distributions $P$ and $Q$ on a measure space $(\mathcal{X}, \mathcal{B})$ is given by $\tv(P,Q) = \sup_{A\in \mathcal{B} } |P(A) - Q(A)|$.

\begin{fact}[TV Properties] \label{tvfacts}
The distance $\TV$, satisfies the following properties:
\begin{enumerate}
\item (Triangle Inequality) For any distributions $P, Q, S$ on a measurable space $(\mathcal{X}, \mathcal{B})$, it holds that 
$$
\tv (P, Q) \leq \tv(P, S) + \tv(S, Q)\,.
$$
\item (Conditioning on an Event) For any distribution $P$ on a measurable space $(\mathcal{X}, \mathcal{B})$ and event $A \in \mathcal{B}$, it holds that
$$\TV\b( P(\,\cdot\, | A), P \b) = 1 - P(A)\,.$$
\item (Conditioning on a Random Variable, Convexity) For any two pairs of random variables $(X, Y)$ and $(X', Y')$ each taking values in a measurable space $(\mathcal{X}, \mathcal{B})$, it holds that
\begin{align*}\TV\big( \mL(X), \mL(X') \big) &\le \TV\big( \mL(Y), \mL(Y') \big)
+ \bE_{y \sim Y} \big[ \TV\big( \mL(X | Y = y), \mL(X' | Y' = y) \big)\big]\,,\end{align*}
where we define $\TV\b( \mL(X | Y = y), \mL(X' | Y' = y) \b) = 1$ for all $y \not \in \textnormal{supp}(Y')$.
\item (DPI) For any two random variables $X, X'$ on a measurable space $(\mathcal{X}, \mathcal{B})$ and any Markov Kernel $A: \mathcal{X} \to \mathcal{Y}$, it holds that 
$$
\tv\b( \mL(A(X)), \mL(A(X')) \b) \leq \tv\b(\mL(X), \mL(X')\b)\,.
$$
\end{enumerate}
\end{fact}

The following simple consequence of the data processing inequality will be useful. 
\begin{lemma}\label{e:DPIapplication}
    If $A$ is independent of $X$ and $X'$, then
    $$
    \tv(X+A,X'+A)\leq \tv(X,X')\,.
    $$
\end{lemma}

Recall that the KL-divergence between two distributions $P$ and $Q$ on a measure space $(\mathcal{X},\mathcal{B})$ is defined by
$\kl(P\|Q) = \E_{X\sim P} \log \frac{dP}{dQ}(X)$.

\begin{lemma}[KL-Divergence Between Gaussians]
\label{l:KLGauss}
Let $\mu_1,\mu_2\in \R^n$ and $\Sigma_1,\Sigma_2\in \R^{n\times n}$ be positive definite matrices. Then
$$
\kl(\N(\mu_1,\Sigma_1)\| \N(\mu_2,\Sigma_2))
= \frac12\Big[ 
\log\frac{\det \Sigma_2}{\det \Sigma_1} - n + \mathrm{Tr}(\Sigma_2\inv\Sigma_1) +(\mu_2 - \mu_1)^\top \Sigma_2\inv (\mu_2 - \mu_1)
\Big]\,.
$$
In particular, by Pinsker's inequality,
\begin{equation}\label{eq:tv_gauss_pinsker}
    \tv(\N(\mu_1,I_n), \N(\mu_2,I_n)) \leq \sqrt{\frac12 \KL(\N(\mu_1,I_n)\| \N(\mu_2,I_n))} = \frac12 \|\mu_2-\mu_1\|_2\,. 
\end{equation}
\end{lemma}

Combining the convexity of total variation (Item 3 in Fact~\ref{tvfacts}) and the bound~\eqref{eq:tv_gauss_pinsker} above, we obtain the following lemma. 

\begin{lemma}[TV Between Gaussians with Shifted Means]
\label{l:GaussianShift}
Let $X, A, B\in \R^d$ be random variables with $X\sim \N(0,I_d)$ independent of $A$ and $B$. Then
$$
\tv(X + A + B, X + B) \leq \E \|A\|_2\,.
$$
\end{lemma}

Finally, we will need the following standard lemma bounding the $\chi^2$-divergence between isotropic Gaussians with a random shift. Recall that the $\chi^2$-divergence between two distributions $P$ and $Q$ on a measure space $(\mathcal{X},\mathcal{B})$ is defined as $\chi^2(P\|Q) = \E_{X\sim Q} \b(\frac{dP}{dQ}(X) - 1\b)^2$. 

\begin{lemma}[$\chi^2$-Divergence Between Mixtures of Isotropic Gaussians]
\label{l:GuassianChi2}
    Let $X\sim \N(0,I_d)$ and $A\in \R^d$ a random variable independent of $X$. Then
$$
1+\chi^2(X + A\| X) = \E \exp(\la A, A'\ra )\,, 
$$
where $A'$ is an i.i.d. copy of $A$. 
\end{lemma}
\begin{proof} The proof is a simple application of the so-called ``Ingster trick". 
    Note that 
$$
X+A \sim \E_A \bP_A\,,
$$
where $\bP_A$ is the law $\N(A, I_d)$.
We start by writing the likelihood ratio    
    $$
\frac{d\bP_A}{d\bP_0}(X) = \frac{\exp(-(X-A)^\top (X-A)/2)}{\exp(-X^\top X / 2)} = \exp(X^\top A- A^\top A/2)\,.
    $$
    This now implies, from the definition of $\chi^2$-divergence and Fubini's theorem,
    \begin{align*}
      \chi^2(X + A\| X) +1 &= \E_{X\sim \bP_0}\B( \E_{A}\frac{d\bP_A}{d\bP_0}(X)\B)^2
      =\E_{A,A'}\E_{X\sim\bP_0}\B( \frac{d\bP_A}{d\bP_0}(X) \frac{d\bP_{A'}}{d\bP_0}(X)\B)\,.
    \end{align*}
To evaluate the inner expectation w.r.t. $X$, one can complete the square, obtaining
$$
\E_{X\sim \bP_0}\B( \frac{d\bP_A}{d\bP_0} (X)\frac{d\bP_{A'}}{d\bP_0}(X)\B)
=\exp(\la A, A'\ra)\,.
$$
This completes the proof.    
\end{proof}

\subsection{Gaussian Cloning}\label{subsec:gaus_clone}

% \begin{mdframed}
\vspace{3mm}
\begin{algorithm}\SetAlgoLined\SetAlgoVlined\DontPrintSemicolon%\RestyleAlgo{boxed}
    \KwIn{
    $X_1, \dots, X_N \in \R^M$.
    }
    \KwOut{$X_1\up1,\dots,X_N\up1, X_1\up2,\dots,X_N\up2 \in \R^M$}
    
    \BlankLine

    Sample i.i.d. $Z_1,\dots,Z_N \sim \N(0, I_M)$
    
    \BlankLine

    For each $ i\in [N]$,
    \begin{align*}
        X_i\up1 \gets \frac{X_i + Z_i}{\sqrt{2}}\qquadand
        X_i\up2 \gets \frac{X_i - Z_i}{\sqrt{2}}
    \end{align*}
 
    Return $\{X_i\up1\}_i, \{X_i\up2\}_i$
    \caption{$\GaussC(X_1,\dots, X_N)$}
\label{alg:gaus-clone}
\end{algorithm}
\vspace{3mm}
% \end{mdframed}

\vspace{3mm}
\begin{algorithm}\SetAlgoLined\SetAlgoVlined\DontPrintSemicolon%\RestyleAlgo{boxed}%\LinesNumbered
    \KwIn{
    $X_1, \dots, X_N \in \R^M, K \in \Z^{+}$.
    }
    \KwOut{$\{X_1\up l,\dots,X_N\up l\}_{l=1}^{K}, \ X_i\up l \in \R^M$}
    
    \BlankLine

    $X_1\up 1,\dots,X_N\up 1 \gets X_1, \dots, X_N$

    \BlankLine

    \For{$t =1,\dots,\lceil \log K \rceil$}
       { \For{$i \in [2^{t-1}]$}
       {$X_1\up {2i-1},\dots,X_N\up {2i-1}, X_1\up {2i},\dots,X_N\up {2i} \gets \GaussC(X_1\up i, \dots, X_N\up i)$}} 

    \BlankLine
 
    Return $\{X_1\up l,\dots,X_N\up l\}_{l=1}^{K}$
    \caption{$\GaussCR(X_1,\dots, X_N, K)$}
\label{alg:gaus-clone-rep}
\end{algorithm}
\vspace{3mm}

We will frequently make use of a (by-now standard, see e.g. \cite{brennan2018reducibility}) procedure to create two samples with the same planted signal and independent noise, each at a mildly lower signal strength than the original.

\begin{lemma}[Single Cloning Guarantee]\label{lem:cloning-gen}
    Let $X \sim \N(\mu, I_M)^{\otimes N}$ and $X\up1, X\up2 = \GaussC(X)$ the output of of Algorithm \ref{alg:gaus-clone}.
    Then $X\up1$ and $X\up2$ are two i.i.d. samples of $\N(\mu / \sqrt{2}, I_M)^{\otimes N}$.
\end{lemma}

\begin{proof}
        For each $i \in [N]$, 
    $$
    \begin{pmatrix}
        X_i\up1 \\
        X_i\up2 
    \end{pmatrix}  = 
    \begin{pmatrix}
        X_i \\
        Z_i 
    \end{pmatrix}
    \begin{pmatrix}
        1/\sqrt{2} & 1/\sqrt{2} \\
        1/\sqrt{2} & -1/\sqrt{2} 
    \end{pmatrix}\,.
    $$
    The lemma follows from the rotational invariance property of Gaussian distribution.
\end{proof}
\begin{lemma}[Multiple Cloning Guarantee]\label{lem:cloning-gen-mult}
    Let $X \sim \N(\mu, I_M)^{\otimes N}$ and $X\up1, X\up2, \dots, X\up K = \GaussCR(X, K)$ the output of of Algorithm \ref{alg:gaus-clone-rep}.
    Then $X\up1, X\up2, \dots, X\up K$ are $K$ independent samples of 
    $\N(\mu / \sqrt{2^{\lceil \log K\rceil}}, I_M)^{\otimes N}$.
\end{lemma}
\begin{proof}
    The lemma statement follows from applying the guarantee for one cloning step (Lemma~\ref{lem:cloning-gen}) $\log K$ times.
\end{proof}
We apply the single and multiple cloning procedures' guarantees (Lemmas~\ref{lem:cloning-gen} and \ref{lem:cloning-gen-mult}) to samples from the Spiked Covariance Model to obtain the following lemma.
\begin{lemma}[Spiked Covariance Cloning]\label{lem:cloning-guarantee}
    Let $Z\in \R^{n\times d}$ be data from $\sc$ in the form ~\eqref{e:SCM2}, i.e., 
    $
    Z = X + \tsq g u^T,$
    where $X \sim \N(0, I_n)^{\otimes d}, g\sim \N(0, I_n)$ and $u$ is the unit signal vector. We have the following two cloning guarantees:
    \begin{enumerate}
        \item Let $Z_1\up1,\dots,Z_d\up1, Z_1\up2,\dots,Z_d\up2 = \GaussC(Z_1,\dots,Z_d)$ given in Algorithm \ref{alg:gaus-clone}.
        Then 
        $$
        Z\up1 = X\up1 + \tsq g u^T / \sqrt{2}, \quad Z\up2 = X\up2 + \tsq g u^T / \sqrt{2}\,,
        $$
        where $X\up1, X\up2 \sim \N(0, I_n)^{\otimes d}$ are independent. 
        
        \item Let $\{Z_1\up l,\dots,Z_d\up l\}_{l=1}^K = \GaussCR(Z_1,\dots,Z_d,K)$ be the output of Algorithm \ref{alg:gaus-clone-rep}.
        Then for all $l\in[K]$, 
        $$
        Z\up l = X\up l + \tsq g u^T / \sqrt{2^{\lceil \log K\rceil}}\,,
        $$
        where $X\up1, X\up2,\dots, X\up l \sim \N(0, I_n)^{\otimes d}$ are independent. 
    \end{enumerate}
    Thus, the output of $\GaussC(Z)$ (or $\GaussCR(Z)$) consists of independent samples from $\sc$ with SNR parameter $\theta / 2$ (or $\theta / \b(2^{\lceil \log K\rceil}\b)$) with the same signal vectors $g, u$.
    
\end{lemma}

\subsection{Gaussianization}\label{subsec:gaussianize} Here we describe the guarantees for a subroutine that maps a pair of Bernoulli random variables with different means to a pair of Gaussian random variables with means differing by almost the same amount. Such entry-wise transformations of measure were used in \cite{ma2015computational}, \cite{hajek2015computational}, and later generalized in \cite{brennan2018reducibility}. The following is a special case of Lemma 5.4 in \cite{brennan2018reducibility}, applied to $\bern {1/2}, \bern {1/2+p}$ random variables:

\begin{lemma}[Gaussianization of $\bern {1/2}, \bern {1/2+p}$]\label{cor:gaussianize_}
   Let $n$ be a parameter and suppose that $p = p(n)$ satisfies $n^{-O_n(1)} \leq p < 1/2$.
Define $\mu = \mu(n) \in (0,1)$ to be
$$
\mu = \frac{p}{2 \sqrt{6 \log n + 2 \log p^{-1}}}.
$$
Then there exists a map $\mathtt{RK}_G$ that can be computed in $O(\log n)$-time and satisfies
$$
\tv\big(\mathtt{RK}_G(\bern{1/2+p}), \N(\mu,1)\big) = O_n(n^{-3}), \text{  and  } \tv\big(\mathtt{RK}_G(\bern{1/2}), \N(0,1)\big) = O_n(n^{-3})\,.
$$
\end{lemma}
\begin{remark}[Gaussianization of Rademacher Variables]
    Applying a trivial mapping between Rademacher and Bernoulli random variables, we can obtain similar mapping guarantees for $\rad(2p), \rad(0)$ in place of $\bern{1/2+p}, \bern{1/2}$. We will denote the corresponding map $\mathtt{RK}_G^{\rad}$.
\end{remark}

\subsection{Gram-Schmidt}\label{subsec:gs} One of our subroutines is the classic Gram-Schmidt Process described in Algorithm \ref{alg:GS} below. Given vectors $X_1,\dots,X_N \in \R^M$ it outputs an orthonormal basis $\Xt_1,\dots,\Xt_N$ in $\R^M$.

\vspace{3mm}
\begin{algorithm}\SetAlgoLined\SetAlgoVlined\DontPrintSemicolon%\RestyleAlgo{boxed}%\LinesNumbered
    \KwIn{
    $X_1, \dots, X_N \in \R^M$.
    }
    \KwOut{$\Xt_1,\dots,\Xt_N \in \R^M$}
    
    \BlankLine
    For $i = 1, \dots, N$,
    $$
    \Xt_i \gets \frac{X_i - \sum_{j=1}^{i-1} \la X_i, \Xt_j\ra \Xt_j}{\|X_i - \sum_{j=1}^{i-1} \la X_i, \Xt_j\ra \Xt_j\|}
    $$
    \BlankLine

    Return $\Xt_1,\dots,\Xt_N$
    \caption{$\GS(X_1,\dots, X_N)$}
\label{alg:GS}
\end{algorithm}
\vspace{3mm}

\section{Two Simple New Reduction Primitives}\label{sec:prelims_new}

\subsection{Bernoulli Denoising}\label{subsec:bern_denoise} 

Recall that $\rad(a)$ denotes a random variable taking $\pm1$ values, such that $\E \rad(a) = a$. In our reduction, we work with Rademacher random variables of the form $\rad(a+\Delta)$, where $\Delta \ll a$ represents an undesired noise whose size we do not know. We propose a denoising procedure designed to reduce $\Delta$ relative to $a$ when given multiple samples from $\rad(a+\Delta)$. The main idea is to construct a polynomial of the noisy samples and fresh $\rad(a)$ samples, such that the resulting variable is distributed as $\Rad\b(C(a^M - (-\Delta)^{M})\b)$ for some constant $C$ and power $M$ depending on the number of available noisy samples:
$$
\text{poly}\b(\underbrace{\rad(a+\Delta),\dots,\rad(a+\Delta)}_{N\text{ noisy inputs}},\underbrace{\rad(a),\dots,\rad(a)}_{\text{fresh r.v.s}}\b) \sim \Rad\b(C(a^M - (-\Delta)^{M})\b)\,.
$$
Then, in a setting when $a \approx 1$ and $\Delta \ll 1$, this procedure decreases the undesired noise relative to the constant mean: $$\B|\frac{\Delta}{a}\B|^{M} \ll \frac{\Delta}{a}\,.$$ 
Equally important, it keeps the signal almost the same. The algorithm is outlined in Algorithm~\ref{alg:denoise} with its guarantees stated in Lemma~\ref{lem:denoise_guarantee}. 
%\gb{The basic ideas is ???}

\vspace{3mm}
\begin{algorithm}\SetAlgoLined\SetAlgoVlined\DontPrintSemicolon%\RestyleAlgo{boxed}%\LinesNumbered
    \KwIn{
    $X_1,\dots,X_N \in \{\pm1\}, a\in[-1, 1]$.
    }
    \KwOut{$Y \in \{\pm1\}$}

    \BlankLine
    $M \gets \lfloor \frac{\sqrt{1+8N} - 1}{2} \rfloor$

    \BlankLine
    For $i = 0,1,\dots,M-1$, let $k_i = \frac{(2M+1)i - i^2}{2}$ and generate 
    $$
    W_{i1}, \dots, W_{ii} \sim \rad\B(-a {\binom{M}{i}}^{1/i}\B)^{\otimes i}\,,
    $$
    and compute 
    $$
    Y_i \gets \prod_{k =k_i+1}^{k_i + M-i} X_k \prod_{k=1}^i W_{ik}\,.
    $$

    \BlankLine 

    $$
    Y \gets \begin{cases}
        Y_0 &\text{ w.p. } 1/M,\\
        Y_1 &\text{ w.p. } 1/M,\\
        \,\,\vdots\\
        Y_{M-1} &\text{ w.p. } 1/M.
    \end{cases}
    $$

    \BlankLine 
    Return $(-1)^{M+1}Y$.
    
    \caption{$\Denoise(X_1,\dots, X_N, a)$}
\label{alg:denoise}
\end{algorithm}
\vspace{3mm}

\begin{lemma}\label{lem:denoise_guarantee}
    Fix $N, M = \lfloor \frac{\sqrt{1+8N} - 1}{2} \rfloor,$ and $ a \in [-1,1]$, such that $|a| \leq M^{-1}$ and let $X_1,\dots,X_N \sim \rad(p)^{\otimes N}$. Let $Y = \textsc{Denoise}(X_1,\dots,X_N, a)$ be the output of Algorithm~\ref{alg:denoise}. 
    \begin{enumerate}
        \item If $p = a + \Delta$, where $|\Delta| \leq |a|$, then $\law(Y) \eqdist \rad(M^{-1}a^M + (-1)^{M+1}M^{-1} \Delta^M)$;
        \item If $p = 0$, then $\law(Y) \eqdist \rad(0)$.
    \end{enumerate}
\end{lemma}

\begin{proof}
First, we verify that given as input $X_1,\dots,X_N \sim \rad(p)^{\otimes N}$, where $p \in \{a+\Delta, 0\}$ and $|\Delta| \leq |a| \leq M^{-1}$, the algorithm contains valid operations. Since $|a| \leq M^{-1}$, 
$$
\B|-a \binom{M}{i}^{1/i}\B| \leq |a| M \leq 1\,,
$$
so $\rad\b(-a \binom{M}{i}^{1/i}\b)$ is a well-defined distribution. Moreover, since 
$$
k_{M-1} + M - (M-1) = \frac{M^2 + M - 2N + 2N}{2} \leq N\,,
$$
the calculation of $Y$ indeed involves at most $N$ samples $X_i$. 
We verify the two cases in the statement.
\paragraph{Case 1: $p = a + \Delta$.}
    Since $X_1,\dots, X_n, W_{i1},\dots, W_{ii}$ are jointly independent, $Y_i$ are jointly independent satisfying
    $$
    \E Y_i = \binom{M}{i}(a + \Delta)^{M-i}(-a)^i\,,
    $$
    and therefore, 
    \begin{align*}
        \E Y &= (-1)^{M+1}M^{-1} \sum_{i=0}^{M-1} \E Y_i\\
        &= (-1)^{M+1}M^{-1}\B(\sum_{i=0}^{M-1} \binom{M}{i}(a + \Delta)^{M-i}(-a)^i\B) \\
        &= (-1)^{M+1}M^{-1} \Delta^M + M^{-1}a^M\,.
    \end{align*}
    
    \paragraph{Case 2: $p = 0$.} In this case we similarly have $Y_i$ jointly independent with $\E Y_i = 0$ for all $i = 0, \dots,M-1$, and therefore, $\E Y= 0$.
\end{proof}

\subsection{Removing Randomness in the Size of the Spike}

Let $Z\in \R^{n\times d}$ be a sample from the Spiked Covariance Model~\eqref{e:SCM2}, i.e.,
$Z= X+\tsq g u^\top,$ where $X\sim \N(0,I_n)^{\otimes d}$, $g\sim \N(0,I_n)$ and $u$ is a unit signal vector.

From Gaussian Norm Concentration (Lemma \ref{lem:gausnorm}), with high probability, $\|g\| = \sqrt{n} + O(1)$. 
The following lemma shows that in the analysis of the Spiked Covariance Model~\eqref{e:SCM2} we may assume that $\|g\| = \sqrt{n}$. 
\begin{lemma}\label{l:gNorm}
    In the setup above, for $\theta \leq \thetacomp$,
$$
\tv\b(X + \tsq g u^\top, X + \tsq (\sqrt{n}/\|g\|) g u^\top\b)  = O(\tsq)\,.
$$
\end{lemma}
\begin{proof}
By convexity of total variation, the quantity in the lemma statement is upper bounded by
$$
\E_g\tv\b(\law(X+ \tsq g u^\top|g), \law (X + \tsq (\sqrt{n}/\|g\|) g u^\top|g)\b)\,.
$$
Treating $g$ as fixed for a moment, the total variation is preserved under projection of the columns of both matrices onto $g/\|g\|$, so 
\begin{align*}
    &\tv\b(\law(X+ \tsq g u^\top|g), \law (X + \tsq (\sqrt{n}/\|g\|) g u^\top|g)\b) \\&= \tv(\N(0,I_d) + \tsq \|g\| u^\top,\N(0,I_d) + \tsq \sqrt{n} u^\top) \\
    &= O\b(\tsq\b|\|g\|-\sqrt{n}\b|\b)\,,
\end{align*}
where in the last line we used Lemma~\ref{l:GaussianShift} -- that is, $\tv\b(\N(\mu_1,I_d), \N(\mu_2,I_d)\b) = O(\|\mu_1-\mu_2\|)$. Now
$$
\b(\E \b|\|g\|-\sqrt{n}\b|\b)^2 \leq \E \b(\|g\|-\sqrt{n}\b)^2  = \Var \|g\| + \b(\E \|g\| - \sqrt{n}\b)^2  =  O(1)\,,
$$
where we used the fact that $\E\|g\|- \sqrt{n}= O(1)$ and $\Var \|g\|=O(1) $ (see, e.g., Exercises 3.1.4 and 3.1.5 of \cite{vershynin2018high}). Therefore, the quantity in the lemma is upper bounded by 
$$
O(\tsq) \E_g \b|\|g\|-\sqrt{n}\b| = O(\tsq) \leq O(\sqrt{\thetacomp}) = o(1)\,,
$$
which concludes the proof.
\end{proof}

\section{Total Variation Equivalence between $\sc$ and $\sw$ for $n\gg d^3$}\label{sec:naive_reduction}

In this section, we obtain a two-way equivalence result for the regime $n\gg d^3$, building on the noise distribution convergence \cite{bubeck2016testing,jiang2015approximation} and the $\sw$ to $\sc$ reduction of \cite{brennan2018reducibility}. Note that an average-case reduction from $\sc$ to $\sw$ in this regime is a direct consequence of our computational equivalence result in Section~\ref{sec:direct_reduction}, which holds for $n\gg d^2$. However, here we establish a stronger statement: not only does there exist an efficiently computable transformation of one into the other, but this transformation \emph{can be taken to be the identity map}, because the two distributions are already close in total variation. This highlights the deep connection between the two models and justifies our focus on the more challenging regime of $n\ll d^3$.

\begin{theorem}[Strong Bijective Equivalence between $\sw$ and $\sc$ for $n \gg d^3$] \label{thm:naive_reduction}
Let $Z\in \R^{n\times d}$ be a sample from the Spiked Covariance Model~\eqref{e:SCM2} with a unit signal vector $u\in \ksparseflat d$\footnote{The result holds for $u\in\ksparse d$ with the proof unchanged.} and let $Y = \sqrt{n} \b(\frac1{n} Z^T Z - I_{d} \b)$ be the corresponding rescaled empirical covariance matrix.  Let $\theta< \thetacomp=\min\{k/\sqrt{n},\sqrt{d/n}\}$, $\lambda = \theta\sqrt{n}$, and $W\sim \GOE(d)$. 
    If $n\gg d^3$, then 
    $$
\TV\big( Y, \lambda u u^\top + W \big) \to 0 \quad \text{as} \quad d\to\infty\,. 
    $$
\end{theorem}

\begin{proof}
    The proof is identical to that of Theorem~\ref{t:bipartite}, with a few key differences. To avoid repetition we only mention those differences:
    \begin{enumerate}
        \item We do not clone the input matrix $Z$, rather simply consider an empirical covariance matrix $Z^{\top} Z$. This removes the need for additional symmetry discussion.
        \item We substitute Proposition~\ref{prop:BBH-bipartite} with the result \eqref{e:wishart_wigner_conv} of Bubeck et al. and Jiang and Li \cite{bubeck2016testing,jiang2015approximation}.
    \end{enumerate}
    These changes adapt the proof of Theorem~\ref{t:bipartite} to the present theorem. 
\end{proof}

\section{Empirical Covariance Matrix $\&$ Cloning: Reduction for $n\gg d^2$}\label{sec:direct_reduction}

In this section we establish an average-case reduction from $\sc_{\gamma}$ to $\sw$ in the regime $\gamma \geq 2$, i.e. $n\gg d^2$. This reduction utilizes the Gaussian Cloning procedure (Section~\ref{subsec:gaus_clone}) to improve upon the total variation equivalence result in Section~\ref{sec:naive_reduction}.

We analyze the reduction Algorithm~\ref{alg:naive2} in Theorem~\ref{t:bipartite} just below to obtain the following Corollary (for proof, see Sec.~\ref{proof:clone_reduction}):
\begin{corollary}[Canonical Reduction from $\sc$ to $\sw$]\label{cor:clone_reduction} With canonical parameter correspondence $\mu = (\alpha, \beta, \gamma) \leftrightarrow \nu = f(\mu)=(\alpha, \beta + \gamma/2)$~\eqref{eq:param}, there exists a $O(d^{2+\gamma})$-time average-case reduction for both detection and recovery from $\sc(\mu)$ to $\sw(\nu=f(\mu))$ for $\gamma \geq 2, \mu\in\Omc_{\gamma}$.
\end{corollary}

\vspace{3mm}
\begin{algorithm}\SetAlgoLined\SetAlgoVlined\DontPrintSemicolon%\RestyleAlgo{boxed}%\LinesNumbered
    \KwIn{
    $Z_1,\dots, Z_N\in \R^M$
    }
    \KwOut{$Y\in \R^{{N\times N}}$}
    
    \BlankLine

    \tcp{Clone data into two copies}

    $\{Z^{(1)}_1,\dots, Z^{(1)}_N\}, \{Z^{(2)}_1,\dots, Z^{(2)}_N\} \gets \textsc{GaussClone}(\{Z_1,\dots, Z_N\})$
    
    \BlankLine
    
    \tcp{Compute the inner products}

    For every $i,j\in [N]$, let
    $$Y_{ij} = \frac1{\sqrt{M}} \la Z_i\up1,Z_j\up2\ra $$

    \tcp{Reflect for symmetry}

    $\Ysym \gets \frac{Y + Y^T}{\sqrt{2}}$
    
    Return $\Ysym\in \R^{N\times N}$
    \caption{$\CloneCov(Z_1,\dots, Z_N)$}
\label{alg:naive2}
\end{algorithm}
\vspace{3mm}

\begin{theorem}[TV Guarantees for Algorithm~\ref{alg:naive2}] \label{t:bipartite}

Let $Z= X + \tsq g u^\top\in \R^{n\times d}$ 
be data from the Spiked Covariance Model~\eqref{e:SCM2} with a unit signal vector $u\in \ksparseflat d$\footnote{The result holds for $u\in\ksparse d$ with the proof unchanged.}. Let
 $\Ysym\in \R^{d\times d}$ be the output of $\CloneCov(Z_1,\dots,Z_d)$ (Algorithm~\ref{alg:naive2}). Let $\theta\leq \thetacomp=\min\{k/\sqrt{n},\sqrt{d/n}\}$, $\lambda = \theta\sqrt{n}$, and $W\sim \GOE(d)$. 
    If $n \gg d^2$, then 
    $$
\TV\big( \Ysym, \lambda u u^\top + W \big) \to 0 \quad \text{as}\quad d\to\infty\,.
    $$
    Moreover, $\CloneCov(Z_1,\dots,Z_d)$ can be computed in $O(d^2 n)$-time.
\end{theorem}

In our proof of Theorem~\ref{t:bipartite}, we will use the following result from \cite{brennan2021finetti}:
\begin{proposition}[Theorem 2.6 of Brennan-Bresler-Huang \cite{brennan2021finetti}]\label{prop:BBH-bipartite}
    Let $X,Y\in \R^{n\times d}$ and $G\in \R^{d\times d}$ have i.i.d. $\N(0,1)$ entries. 
    If $n\gg d^2$, then 
    $$
\tv(n^{-1/2}X^\top Y, G) \to 0\,,
    $$
    as $n\to \infty$.
\end{proposition}

\begin{proof}(of Theorem~\ref{t:bipartite})

First, $\GaussC (Z_1,\dots,Z_d)$ works in $O(dn)$-time and the inner product step takes $O(d^2 n)$, resulting in total runtime $O(d^2 n)$ of $\CloneCov$.

We prove 
$$
\TV\big( Y, \lambda u u^\top + G \big) \to 0\,, 
    $$
where $Y$ is the output before the last reflection step and $G\sim N(0,I_d)^{\otimes d}$. Note that given this bound the statement for $\Ysym$ immediately follows from the DPI for total variation (Lemma~\ref{tvfacts}).

For each column $i\in [d]$,
$Z_i = X_i + \tsq u_i g\in \R^n$, where $X_i\sim \N(0,I_n)$ and $g\sim \N(0,I_n)$. By Lemma~\ref{l:gNorm} and the triangle inequality for total variation we may assume that $\|g\|=\sqrt{n}$.
Throughout the proof we condition on $g$ and let the columns of $\bar X\in \R^{n\times d}$ be defined by $\Xbar_i = \Pi_{g^\perp} X_i\in \R^n$, while the entries of $X_0\in \R^d$ are defined by $X_{0,i} = \la X_i, g/\|g\|\ra$, so that $X_0=(X_{0,i})_i\sim \N(0,I_d)$. Note that $\Xbar$ and $X_0$ are conditionally independent given $g$ and that $\Xbar^\top g=0$. 
We decompose $Z\in \R^{n\times d}$ as 
$$
Z = \Xbar + g(\|g\|\inv X_{0}+ \tsq u)^\top\,.
$$
Decomposing both $Z\up 1$ and $Z\up 2$ this way we obtain
\begin{align}
Y &=
   \frac1{\sqrt{n}} (Z\up 1)^\top Z\up 2  \nonumber \\&=  \frac1{\sqrt{n}}(\Xbar \up 1)^\top \Xbar \up 2 +  \frac1{\sqrt{n}}\big[ \theta\|g\|^2 u u^\top + (X_0\up 1)(X_0\up 2)^\top  + \|g\|(X_0\up 1) \tsq u^\top + \|g\|\tsq u (X_0\up 2)^\top\big] \nonumber
   \\&\approx_{o(1)}
   G+ \sqrt{n}\theta u u^\top +  \frac1{\sqrt{n}} (X_0\up 1)(X_0\up 2)^\top  + (X_0\up 1) \tsq u^\top + \tsq u (X_0\up 2)^\top \nonumber
   \\&=:G+ \lambda  u u^\top + E + F_1 + F_2^\top
   \,, \label{e:innerProducts1}
\end{align}
where $G\sim N(0,I_d)^{\otimes d}$ is independent of $X_0\up 1, X_0\up 2, u$.
In the third line we used Prop.~\ref{prop:BBH-bipartite}, which states that $\frac1{\sqrt{n}}(\Xbar \up 1)^\top \Xbar \up 2 \approx G$, and
Lemma~\ref{e:DPIapplication}, valid due to independence of $\Xbar$ and $X_0$, as well as
that $\|g\|=\sqrt{n}$. What remains is to remove $E,F_1,F_2$, which we do in a sequence of steps.

\paragraph{Step 1: Removing $E=\frac1{\sqrt{n}} (X_0\up 1)(X_0\up 2)^\top$.}
Note that $X_{0,i}=\la X_i, g/\|g\|\ra\sim \N(0,1)$, so
\begin{align*}
    \E \|E\|_\fr^2 = \frac{d^2}{n} \E (X_{i,0}\up 1)^2\E (X_{j,0}\up 2)^2
= \frac{d^2}{n} 
=o(1) 
    \,,
\end{align*}
and $\E \|E\|_\fr \leq \sqrt{\E \|E\|_\fr^2}=o(1)$.
With this, Lemma \ref{l:GaussianShift} (TV between Gaussians of different means) implies that the quantity in \eqref{e:innerProducts1} is approximated in TV as 
$$G+\lambda  u u^\top + E + F_1 + F_2
   \approx_{o(1)} G + \lambda  u u^\top + F_1 + F_2\,.
   $$

 \paragraph{Step 2: Removing $F_1=(X_0\up 1) \tsq u^\top$ (in the case $n\gg d^3$).}

 Before showing how to remove $F_1$ using a $\chi^2$ argument, we show a simple way to bound these terms in case $n\gg d^3$. 
 % $F_1 = (X_0\up 1) \tsq u^\top$. Note that $X_{0,i}=\la X_i, g/\|g\|\ra\sim \N(0,1)$, and hence,
 We have that
 $$
 \E \|F_1\|_\FR^2 = \theta d \E (X_{0,i}\up 1)^2
= \theta d\leq (d^3 / n)^{1/2}\,, 
 $$
so $\E \|F_1\|_\FR \leq \sqrt{\E \|F_1\|_\FR^2}=o(1)$ when $n \gg d^3$.
Thus, we can similarly remove $F_1$ using Lemma \ref{l:GaussianShift} (TV between Gaussians of different means):
$$G+\lambda  u u^\top + F_1 + F_2
   \approx_{o(1)} G + \lambda  u u^\top + F_2\,.
   $$
 
\paragraph{Step 2': Removing $F_1=(X_0\up 1) \tsq u^\top$.}
We now show how to remove $F_1$ whenever $n\gg d^2$.
We apply the data processing inequality (to the operation of adding $F_2$ and using the fact that $F_2$ is independent of $G$ and $F_1$) and then Lemma~\ref{l:GuassianChi2} to obtain that
\begin{align*}
    \tv(G + \lambda  u u^\top + F_1 + F_2, G + \lambda  u u^\top + F_2) 
    &\leq 
     \tv(G  + F_1, G )\leq \sqrt{ \E \exp\la F_1, F_1'\ra -1 }\,,
\end{align*}
where $F_1'$ is an independent copy of $F_1$. 
Letting $A,B\stackrel{\mathrm{i.i.d.}}\sim\N(0,I_d)$ (so that $A \stackrel{d}= X_0\up1$ and $B= X_0\up2$), we have that
\begin{align*}
    \theta^{-1}\la F_1,F_1'\ra
    \equaldist \sum_{ij}  A_i u_j B_i u_j =
    \la  A,  B\ra \,.
\end{align*}
Recall that the MGF  of a 
$U\sim N(0,\sigma^2)$ random variable is $\E e^{tU}=e^{t^2\sigma^2/2}$ and that of a
$\chi^2$ random variable (such as $B_i^2$) is
$\E \exp(t B_i^2) = (1 - 2t)^{-1/2}$ for $|t|<1/2$. Using this and that $\la A, x\ra\sim N(0,\|x\|^2) $ for any fixed vector $x$, we obtain
\begin{align*}
 \E \exp \la  F_1,F_1'\ra = \E\b[\E[\exp (\theta \la A, B\ra)|B]\b] = \E[\exp(\theta^2\|B\|^2 /2)] &= (1-\theta^2)^{-d/2} \,.
\end{align*}
We obtain $ \E \exp \la A, B\ra \leq (1-\theta^2)^{-d/2} \leq 1 + O(\theta^2 d) \leq 1 + O(d^2/ n) = 1+o(1)$.

   \paragraph{Step 3: Removing $F_2$.} An identical argument to the previous one shows that 
   $$
G + \lambda  u u^\top + F_2 \approx_{o(1)} G + \lambda  u u^\top\,.
   $$
   Combining the steps via the triangle inequality completes the proof. 
\end{proof}

\section{Orthogonalization via Gram-Schmidt: Reduction for $n \gg d$}\label{sec:gs_reduction}

In this section we establish an average-case reduction from $\sc$ to $\sw$ for the general case $n \gg d$, i.e. for arbitrary $\gamma \geq 1$.
We combine Gaussian Cloning (Sec.~\ref{subsec:gaus_clone}), orthogonalization through Gram-Schmidt process (Sec.~\ref{subsec:gs}), Bernoulli Denoising procedure (Sec.~\ref{subsec:bern_denoise}), and the ``flipping" procedure in the reduction Algorithm~\ref{alg:wish-ssbm}. In the analysis of Algorithm~\ref{alg:wish-ssbm} in Theorem~\ref{t:wishSSBM} we rely on the key perturbation Lemma~\ref{lem:gZ_strong} for Gram-Schmidt process. Theorem~\ref{t:wishSSBM} implies the following corollary (for proof, see Sec.~\ref{proof:gs_reduction}).

\begin{corollary}[Canonical Reduction from $\sc$ to $\sw$ at Comp. Thresh.]\label{cor:gs_reduction}

With canonical parameter correspondence $\mu = (\alpha, \beta, \gamma) \leftrightarrow \nu = f(\mu)=(\alpha, \beta + \gamma/2)$~\eqref{eq:param}, there exists an $O(d^{2+\gamma})$-time average-case reduction at the computational threshold for both detection and recovery\footnote{As described in Section~\ref{sec:avg_case}, we prove the reduction for recovery for a slightly restricted class of potential recovery algorithms for $\sw$, which we believe to be a very mild restriction. The recovery algorithm in this case also finishes in $O(d^{2+\gamma})$ and given a recovery algorithm for $\sw$ with loss $<{\ell'}^{\star}$, is guaranteed to output an estimate for $\sc$ with loss at most ${\ell'}^{\star} + o(1)$.} from $\sc(\mu)$ to $\sw(\nu=f(\mu))$ for $\alpha \leq 1/2, \gamma\geq1, \mu \in \CCc_{\gamma}$.

\end{corollary}

\vspace{3mm}
\begin{algorithm}\SetAlgoLined\SetAlgoVlined\DontPrintSemicolon%\RestyleAlgo{boxed}%\LinesNumbered
    \KwIn{
    $Z_1,\dots, Z_N\in \R^M, 2K \in \Z^{+}, \psi \in \R$
    }
    \KwOut{$Y\in \R^{N\times N}$}

    \BlankLine

    \tcp{Clone data into $2K+1$ copies}
    $$
    \{Z_1\up0,\dots, Z_N\up0\}, \{Z_1\up1,\dots, Z_N\up1\} \gets \GaussC(\{Z_1,\dots, Z_N\})
    $$
    $$
    \b\{Z_1\up l,\dots, Z_N\up l\b\}_{l=1}^{2K} \gets
    \mathtt{GaussCloneRep}(\{Z_1\up1,\dots, Z_N\up1\}, 2K)
    $$

    \BlankLine
    
    \tcp{Derive a basis from the first copy}
    
    $\Zt_1\up0,\dots, \Zt_N\up0 \gets \GS(Z_1\up0,\dots, Z_N\up0)$
    
    \BlankLine
    
    \tcp{Compute the Gram-Schmidt coefficients}

    For every $i,j\in [N]$ and $l\in [2K]$, 
    \begin{align*}
        Y_{ij}\up l &\gets \la Z_i\up l,\Zt_j\up0\ra
    \end{align*}

    \BlankLine
    
    \tcp{Combine the entries via "flipping" technique}

    For $i \leq j$ and $l\in [K]$,
    $$Y_{ji}\up l = Y_{ij}\up l\gets\sign(Y_{ij}\up l \cdot Y_{ji}\up {l+K})$$
    
    \BlankLine

    \tcp{Denoise the entries}
    For every $i,j\in [N]$,
    $$
    Y^{\text{Rad}}_{ij} \gets \Denoise(Y_{ij}\up1,\dots, Y_{ij}\up K, \psi)
    $$

    \BlankLine

    \tcp{Lift entries to Gaussian distribution}
    
    $$Y \gets \RK_G^{\text{Rad}}\B(Y^{\text{Rad}}\B)$$

    \BlankLine

    Return $Y\in \mathbb{R}^{N \times N}$ 
    \caption{$\WishartToSSBM(Z_1,\dots, Z_N, 2K, \psi)$}
\label{alg:wish-ssbm}
\end{algorithm}
\vspace{3mm}

\begin{theorem}[TV Guarantee for $\WishartToSSBM$]\label{t:wishSSBM}
Fix $\epsilon>0$ and $\alpha\in (0,1/2]$. Define $A_{\alpha, \epsilon}=\max\{2\alpha \epsilon^{-1}, 4\alpha (1+\epsilon)^{-1}\}$ and let $K_{\alpha, \epsilon} = \lceil A_{\alpha, \epsilon}^2 + 3A_{\alpha, \epsilon} + 4\rceil$,
$C_{\alpha, \epsilon}= 2^{\lceil \log(2K_{\alpha, \epsilon})\rceil + 1}$, and $\psi = \frac{1}{2C_{\alpha, \epsilon}} \theta^2 n k^{-2}$. Let $C'_{\alpha, \epsilon} = C_0 A_{\alpha, \epsilon}^{-1}\cdot2^{-(\log(2K_{\alpha,\epsilon}) + 3)A_{\alpha, \epsilon}}$ for a small enough universal constant $C_0$.

Let $d,k,n$ be such that $n = d^{1+\epsilon}$ and $k = d^{\alpha}$ and let $Z\in \R^{n\times d}$ be a sample from $\sc$~\eqref{e:SCM2} with unit signal vector $u \in \ksparseflat d$. Let
 $Y\in \R^{d\times d}$ be the output of $\WishartToSSBM(Z_1,\dots, Z_d, 2K_{\alpha, \epsilon} , \psi)$ (Alg.~\ref{alg:wish-ssbm}).  
 Choose any $\delta\leq (\log n)^{-1}$ and take $\theta$ in the range
$$\delta \thetacomp(d,k,n) \leq \theta< \thetacomp(d,k,n) = k/\sqrt{n}\,.$$
    Then there is a $\lambda$ in the range
    $$
\lamcomp(d,k) \delta^{2A_{\alpha, \epsilon}} \cdot C'_{\alpha, \epsilon} \log^{-1/2} n \leq \lambda \leq \lamcomp(d,k)\,,
$$ 
such that
    $$
\TV\big( Y, \lambda u' u'^\top + W \big) \to 0\,, 
    $$
where $W\sim \GOE(d)$ and the signal $u' \in \ksparseflat d$ is such that $u'_i = |u_i|$ for every $i\in [d]$.
Moreover, the algorithm $\WishartToSSBM(Z_1,\dots, Z_d, 2K_{\alpha, \epsilon} , \psi)$ (Alg.~\ref{alg:wish-ssbm}) runs in $O(d^2n)$-time.

\end{theorem}

\paragraph{Proof of Theorem~\ref{t:wishSSBM}.}\label{subsec:wishSSBM_proof}
First, the cloning step, Gram-Schmidt algorithm, inner product calculations, denoising, and Gaussianization can all be computed in $O(d^2 n)$ time.

Recall from the cloning guarantee (Corollary~\ref{lem:cloning-guarantee}) that
$$Z\up0 = {X\up0} + \frac{1}{\sqrt{2}}\tsq g u^\top,\quad Z\up l = {X\up l} + \frac{1}{\sqrt{C_{\alpha,\epsilon}}}\tsq g u^\top, \  \forall l \in [2K_{\alpha,\epsilon}]\,,$$ 
where $C_{\alpha,\epsilon} = 2^{\lceil \log(2K_{\alpha,\epsilon}) \rceil + 1}$ and $X\up0, X\up1, X\up2, \dots, X\up{2K_{\alpha,\epsilon}} \sim N(0,I_n)^{\otimes d}$ 
are jointly independent and also independent of $g$. Then for every $l \in [2K_{\alpha,\epsilon}]$ 
$$Y\up l = {Z\up l}^\top \Zt\up 0=\b({X\up l}^\top + \frac{1}{\sqrt{C_{\alpha,\epsilon}}}\tsq u g^\top\b) \Zt\up0 = {X\up l}^\top \Zt\up0 + \frac{1}{\sqrt{C_{\alpha,\epsilon}}}\tsq u g^\top \Zt\up0\,.$$
Given $g, u$, and $\Zt\up 0$, we have ${X\up l}^\top \Zt\up0\eqdist N(0, I_d)^{\otimes d}$, so we can rewrite
$$
Y\up l \eqdist Q\up l + \frac{1}{\sqrt{C_{\alpha,\epsilon}}}\tsq u g^\top \Zt\up0\,,
$$
where $Q\up l \sim N(0, I_d)^{\otimes d}$ are jointly independent and independent of $g, \Zt\up 0$. In particular, for every $l\in [2K_{\alpha,\epsilon}]$,
$$
Y_{ij}\up l = \begin{cases}
    Q_{ij}\up l,\quad i\not\in S\\
    Q_{ij}\up l + \frac{1}{\sqrt{C_{\alpha,\epsilon}}}\tsq u_i g^\top \Zt_j\up0 , \quad i \in S\\
\end{cases}
$$
Then for $i\leq j$ and  $l \in [K_{\alpha,\epsilon}]$,
$$
Y_{ji}\up l = Y_{ij}\up l=\sign(Y_{ij}\up l \cdot Y_{ji}\up {l + K_{\alpha,\epsilon}}) \eqdist \begin{cases}
    \Rad(0), \quad i\not\in S \text{ or } j\not\in S\,,\\
    \Rad(\frac{1}{C_{\alpha,\epsilon}} \theta u_i g^\top \Zt_j\up0 \cdot u_j g^\top \Zt_{i}\up0), \quad i,j \in S\,.
\end{cases}
$$
Since $Q\up1, Q\up2\sim N(0, I_d)^{\otimes d}$, $Y\up l$ are independent symmetric matrices with jointly independent entries. 
Define an event 
$$
\Em = \B\{\forall i\in S, \ \b|g^T\Zt_i\up 0 - \frac{1}{\sqrt{2}}\tsq u_i \sqrt{n}\b| = \tO\b(1 + \theta n^{1/2}k^{-1} + \tsq n^{-1/2}d k^{-1/2} + \theta^{3/2} n^{1/2} k^{-1/2}\b)\B\}\,.
$$
From the GS-Perturbation Lemma~\ref{lem:gZ_strong} we have that $\P(\Em^{c}) \leq n^{-10}$. Then, for $\theta = O(\thetacomp)$ we have that given $\Em$, 
\begin{align*}
    g^T\Zt_i\up 0 &= \frac{1}{\sqrt{2}}\tsq u_i \sqrt{n} + \tO(1 + \theta n^{1/2}k^{-1} + \tsq n^{-1/2}d k^{-1/2} + \theta^{3/2} n^{1/2} k^{-1/2})\\
    &= \frac{1}{\sqrt{2}}\tsq u_i \sqrt{n} \B(1+ \tO(n^{-1/4} + d n^{-1} + d^{1/2}n^{-1/2})\B) \\&= \frac{1}{\sqrt{2}}\tsq u_i \sqrt{n} \B(1+ \tO(n^{-1/4} + d^{1/2}n^{-1/2})\B)\,,
\end{align*}
and therefore,
$$
u_i g^T\Zt_i\up 0 = \frac{1}{\sqrt{2}}\tsq u_i^2 \sqrt{n} \B(1+ \tO(n^{-1/4} + d^{1/2} n^{-1/2})\B) = \mathds{1}_{i\in S} \frac{1}{\sqrt{2}}\tsq \sqrt{n} k^{-1} \B(1+ \tO(n^{-1/4} + d^{1/2} n^{-1/2})\B)\,.
$$
In this case, for every $i,j\in [d], l \in [K_{\alpha,\epsilon}]$,
$$
\E Y_{ij}\up l = \begin{cases}
    0, \ i\not\in S \text{ or } j\not\in S,\\
    \frac{1}{2C_{\alpha,\epsilon}} \theta^2 n k^{-2} \b(1+ \tO(n^{-1/4} + d^{1/2} n^{-1/2})\b), \ i,j \in S.
\end{cases}
$$
Let $Y^{\text{Rad}}_{ij} \gets \Denoise(Y_{ij}\up1,\dots, Y_{ij}\up K, \psi = \frac{1}{2C_{\alpha,\epsilon}} \theta^2 n k^{-2})$. From the denoising procedure guarantees (Lemma~\ref{lem:denoise_guarantee}), $Y^{\text{Rad}}$ has jointly independent $\pm$ entries and satisfies 
$$
\E Y^{\text{Rad}}_{ij} = \begin{cases}
    0, \ i\not\in S \text{ or } j\not\in S,\\
    M^{-1}(\frac{1}{2C_{\alpha,\epsilon}} \theta^2 n k^{-2})^M \B(1+ \tO(n^{-M/4} + d^{M/2} n^{-{M/2}})\B), \ i,j \in S,
\end{cases}
$$
where $M = \lfloor \frac{\sqrt{1+8K_{\alpha,\epsilon}}-1}{2}\rfloor \geq A_{\alpha, \epsilon} = \max\{2\alpha \epsilon^{-1}, 4\alpha (1+\epsilon)^{-1}\}$. It is easy to verify that for this choice of $M$, $$
M^{-1}\b(\frac{1}{2C_{\alpha,\epsilon}} \theta^2 n k^{-2}\b)^M \B(1+ \tO(n^{-M/4} + d^{M/2} n^{-M/2})\B) = M^{-1}\b(\frac{1}{2C_{\alpha,\epsilon}} \theta^2 n k^{-2}\b)^M + o(k^{-1})\,.
$$
Define $Y^{\text{denoised}}\in \R^{d\times d}$ to be a symmetric matrix of jointly independent Rademacher Variables satisfying $\E Y^{\text{denoised}}_{ij}=0$ for $i\not\in S$ or $ j\not\in S$ and $\E Y^{\text{denoised}}_{ji} = \E Y^{\text{denoised}}_{ij} = M^{-1}\b(\frac{1}{2C_{\alpha,\epsilon}} \theta^2 n k^{-2}\b)^M$ for $i,j\in S$. Applying triangle inequality for TV distance, conditioning on event $\Em$ in TV (Fact \ref{tvfacts}) and using the expression for TV distance between Gaussians (Lemma~\ref{l:GaussianShift}), we have
$$
\tv(Y^{\text{Rad}}, Y^{\text{denoised}}) \leq \P(\Em^{c})+\tv(Y^{\text{Rad}}|\Em, Y^{\text{denoised}}) \leq n^{-10} + k^2 o(k^{-2}) = o(1)\,.
$$ 
Then, by DPI (Fact~\ref{tvfacts}),
$$
\tv\b(Y=\textsc{RK}_G^{\text{Rad}}(Y^{\text{Rad}}), \textsc{RK}_G^{\text{Rad}}(Y^{\text{denoised}})\b) \leq \tv(Y^{\text{Rad}}, Y^{\text{denoised}}) = o(1)\,,
$$
where $Y$ is the final output of $\WishartToSSBM$. Note that we can rewrite 
$$
\E Y^{\text{denoised}}_{ij} = M^{-1}\b(\frac{1}{2C_{\alpha,\epsilon}} \theta^2 n k^{-2}\b)^M \mathds{1}_{i,j\in S} = \lambda' u'_i u'_j\,,
$$
where the signal $u'\in \ksparseflat d$ is such that $u_i' = |u_i|,\ \forall i$ and $\lambda' = M^{-1}\b(\frac{1}{2C} \theta^2 n k^{-2}\b)^M k$. Finally, from the guarantees for $\textsc{RK}_G^{\text{Rad}}$ (Lemma~\ref{cor:gaussianize_}), we have 
$$
\tv(\textsc{RK}_G^{\text{Rad}}(Y^{\text{denoised}}), W + \lambda u' {u'}^\top) \leq O(n^{-3})\,, 
$$
where $$\lambda = \frac{\lambda'}{4\sqrt{6\log n + 2 \log {\lambda'}^{-1}}}\,.
$$ 
Substituting $\theta \geq \delta \thetacomp = \delta k/\sqrt{n}$ we obtain for a universal constant $C_0$, 
$$
\lambda \geq \lamcomp \delta^{2A} \cdot C_0 A_{\alpha, \epsilon}^{-1}\frac{1}{(2C_{\alpha,\epsilon})^{A_{\alpha, \epsilon}}} \log^{-1/2} n \geq \lamcomp \delta^{2A} \cdot C'_{\alpha, \epsilon} \log^{-1/2} n\,.
$$
From the triangle inequality, we conclude that
$$\tv\b(Y=\textsc{RK}_G^{\text{Rad}}(Y^{\text{Rad}}), W + \lambda u' {u'}^\top\b) \leq o(1) + O(n^{-3}) = o(1)\,,$$
which finishes the proof.

\section{Gram-Schmidt Perturbation}\label{sec:gs_perturb}

In this section we analyze the result of applying Gram-Schmidt orthogonalization (Algorithm~\ref{alg:GS}) to data from the Spiked Covariance Model. In particular, fix $\epsilon > 0$ and let $n \geq d^{1+\epsilon}$. Let $Z = X + \tsq gu^\top \in \R^{n\times d}$ be data from the Spiked Covariance Model~\eqref{e:SCM2}, where $X\sim \N(0,I_n)^{\otimes d}$, $g\sim \N(0,I_n)$ and $u \in \ksparse d$\footnote{The result in this section is for the more general case $u\in \ksparse d$ and automatically holds for $u\in\ksparseflat d$.} is a unit signal vector such that $|u_i|\leq c_u k^{-1/2} \ \forall i$ for some $c_u = \tO(1)$. Let $\theta$ be in the range 
$$\sqrt{k/n} = \thetastat \leq \theta< \thetacomp=\min\{k/\sqrt{n},\sqrt{d/n}\}\,.$$
 We prove the following result for basis vectors $\Xt = \GS(\{X_1,\dots, X_d\}), \Zt = \GS(\{Z_1,\dots, Z_d\})$:
 
\begin{lemma}[Spike in $\Zt$]\label{lem:gZ_strong}
    Fix any constant $K> 0$. 
    For $Z$ as described above, with probability $\geq 1-n^{-K}$, simultaneously for all $j \in [d]$,\footnote{In this section we write $\tO(A)$ (corresp. $\tOm(A)$) to denote that the expression is upper (corresp. lower) bounded by $c_1 (\log n)^{c_2} A$ for some known constants $c_1, c_2$. For all such bounds we explicitly derive the expressions for $c_1, c_2$ and only sometimes resort to $\tO, \tOm$ notation to provide intuition.} 
    $$\b|\la g, \Zt_j\ra - \la g, \Xt_j\ra - \tsq n^{1/2} u_j\b| \leq \begin{cases}
         \tO(\tsq) = o(1)
        &\text{if } j\not\in S\,,\\
        \tO(\theta n^{1/2} k^{-1} + \tsq n^{-1/2} d k^{-1/2}+\theta^{3/2} n^{1/2} k^{-1/2})\\
        \qquad = o(\tsq n^{1/2} k^{-1/2})=o(\tsq n^{1/2} \max_i u_i)
        &\text{if } j\in S\,,
    \end{cases}
    $$
    where the constants within the $\tO$ depend on $K$.
\end{lemma}

A natural interpretation of this result is that given $Z= X+\tsq g u^\top$ the resulting basis vectors $\Zt$ approximately behave as $\Zt\approx \Xt+\tsq n^{-1/2} g u^\top$, preserving the original Spiked Covariance signal structure. In order to obtain this result, we generalize the statement of the lemma to the following theorem, which is then proved by induction.

\begin{theorem}[GS Perturbation]\label{thm:gs-pert}
Fix $K > 0$. Denote
\begin{enumerate}
    \item $\Xt \gets \GS(\{X_1,\dots, X_d\})$ and $ \Zt \gets \GS(\{Z_1,\dots, Z_d\})$, 
    \item $r_i = \|Z_i-\sum_{j=1}^{i-1}\la Z_i, \Zt_j\ra \Zt_j\|$ and $s_i = \|X_i-\sum_{j=1}^{i-1}\la X_i, \Xt_j\ra \Xt_j\|$.
\end{enumerate}
For all $d, k, \theta, n$ satisfying the above and $d \geq d_0 = O(1)^{O(1/\epsilon)}$, there exist $\an,\as,\wn,\ws = h(n,d,k,\theta)$ and event $\Em$, such that
$\P\b[\Em^c] \leq n^{-K}$, 
$$\an = \tO(n^{-1/2}), \as = o(k^{-1/2}), \wn = o(1), \ws = o(\min\{n^{1/4}, n^{1/2}k^{-1/2}\})\,,$$ and
given $\Em$ we can write 
\begin{equation}\label{eq:Zt}
    \Zt_i = \frac{s_i}{r_i} \b(\Xt_i + \tsq n^{-1/2}(u_i+\alpha_i)g + W_i\b)\,,
\end{equation}
where
    $$
    |\alpha_i| \leq \begin{cases}
        \an, & i \not\in S, \\
        \as, & i\in S\,,
    \end{cases}
\qquad \text{and} \qquad
    \max\b\{|\la g, W_i\ra|, \|W_i\|\b\} \leq \begin{cases}
        \wn, & i \not\in S, \\
        \ws, & i\in S\,.
    \end{cases}
    $$
\end{theorem}

\subsection{Sketch of Proof of Theorem~\ref{thm:gs-pert}}

To analyze the basis vectors $\Zt = \GS(Z_1,\dots,Z_d)$ we first choose to express them as
$$
\Zt_i = \frac{s_i}{r_i} \b(\Xt_i + \tsq n^{-1/2}(u_i+\alpha_i)g + W_i\b)\,.
$$
The basis vector $\Zt_i$ contains the corresponding basis vector $\Xt_i$ in the absence of spike with appropriate rescaling by $s_i r_i^{-1}$. Moreover, while $Z_i = X_i + \tsq u_i g $, the unit vector $\Zt_i$ contains $n^{-1/2} \tsq u_i g$, due to normalization, and with some perturbation $\alpha_i$. We put what remains into $W_i$ -- a vector of (hopefully) small magnitude and small correlation with the spike $g$.

With this representation of $\Zt_i$, in Lemma~\ref{lem:rec}, we identify a recursive relation that expresses $\alpha_i$ and $W_i$ in terms of $s_i, r_i, X_i, u_i$ and terms from the previous induction steps. These expressions follow directly from the iterative formula for Gram-Schmidt procedure.

Now, being able to express $\alpha_i$ and $W_i$ through quantities from previous steps, %and the fresh randomness from $X_i$
we establish high probability magnitude bounds. Using the fresh randomness of $X_i\sim N(0,I_n)$ at step $i$ and the concentration properties of Gaussian and $\chi^2$ random variables,
% (Lemmas~\ref{lem:gausnorm} and \ref{lem:gaustail}), 
we obtain high probability bounds on $|\alpha_i|$ and $\|W_i\|, \la g, W_i\ra$ in terms of parameters $d,k,\theta,n, \an,\as,\wn,\ws$ (Lemmas~\ref{lem:alpha} and~\ref{lem:W}). 

We extend these results to obtain a high probability bound on the quantity of our interest -- $\la g, \Zt_i\ra$ in Lemmas~\ref{lem:gZ}, \ref{lem:gZ_strong}. This is the crucial consequence of the whole perturbation theorem and it allows us to explicitly quantify the spike in the basis vectors $\Zt$.

The high probability bounds in Lemmas~\ref{lem:alpha}, \ref{lem:W} and \ref{lem:gZ} are expressed in terms of both $d,k,\theta,n$ and $\an,\as,\wn,\ws$, where the letter appear from applying the induction assumptions. We finish the proof of the theorem by showing in Lem.~\ref{lem:const_exist} that we can find values for $\an, \as, \wn, \ws$ in terms of $d,k,\theta,n$ that satisfy the conditions of the Theorem~\ref{thm:gs-pert}. This extra step is introduced to simplify the recursive calculations -- plugging in \emph{a priori} values for $\an, \as, \wn, \ws$ in terms of $d,k,n,\theta$ results in complicated expressions.

\subsection{Preliminary Lemmas}

The following lemma is standard and can be found, for example, in \cite{spokoiny2023concentration}.

\begin{lemma}[General Norm Concentration of Gaussian rv]\label{lem:gausnorm0}
    For $X\sim \N(0, \Sigma),$ where $\Sigma\in \mathbb{R}^{n\times n}:$
    $$
    \Pr\b[\b|\|X\|^2 - \tr (\Sigma)\b|\geq 2\sqrt{t \tr(\Sigma^2)} + 2t \|\Sigma\|\b] \leq e^{-t}\,.
    $$
\end{lemma}
Lemma~\ref{lem:gausnorm} and Corollary~\ref{lem:gausidnorm} are direct corollaries of Lemma~\ref{lem:gausnorm0}.
\begin{lemma}[Norm Concentration of Gaussian rv with Independent Coordinates]\label{lem:gausnorm}
    For $X\sim \N(0, \Sigma),$ where $\Sigma\in \mathbb{R}^{n\times n}$ is diagonal and $t\geq 1$,
    $$
    \Pr\b[\b|\|X\|^2 - \tr(\Sigma)\b|\geq t \sqrt{\tr(\Sigma^2)} \b] \leq e^{-t/4}\,.
    $$
    In particular,
    $$
    \Pr\b[\b|\|X\|^2 - \tr(\Sigma)\b|\geq 4 K \log n \sqrt{\tr(\Sigma^2)} \b] \leq n^{-K}\,. 
    $$
\end{lemma}
\begin{corollary}[Norm Concentration of Standard Gaussian rv]\label{lem:gausidnorm}
    For $X\sim \N(0, I_n)$ and $t\geq 1$:
    $$
    \Pr\b[\b|\|X\|^2 - n\b|\geq t n^{1/2} \b] \leq e^{-t/4}\,,
    $$
    and in particular,
    $$
    \Pr\b[\b|\|X\|^2 - n\b|\geq 4 K \log n \cdot n^{1/2} \b] \leq n^{-K}\,.
    $$
\end{corollary}
\begin{lemma}[Gaussian Tail Bound (Proposition 2.1.2 in \cite{vershynin2018high})]\label{lem:gaustail}
    For $X\sim \N(0, \sigma^2)$ and $t\geq 1$:
    $$
    \Pr\b[ |X| \geq t \sigma \b] \leq \frac{\sqrt{2}}{\sqrt{\pi}} e^{-t^2/2}\,,
    $$
    and in particular,
    $$
    \Pr\b[|X|\geq \sigma \sqrt{\log (2/\pi) + 2K \log n}\b] \leq n^{-K}\,.
    $$
\end{lemma}

\subsection{Proof of Theorem~\ref{thm:gs-pert}}

For a chosen constant $K$, let $\cc0 = 4K \log n$, $\cc2 = \sqrt{\log (2/\pi) + 2K \log n}$ -- that is, $\cc0$ is the constant from the Gaussian norm concentration inequalities (Lem.~\ref{lem:gausnorm}, \ref{lem:gausidnorm}) and $\cc2$ from the Gaussian tail bound (Lem.~\ref{lem:gaustail}). In our argument we will show that all conditions of Thm.~\ref{thm:gs-pert} hold simultaneously with probability at least $1 - n^{-K+4}$, which yields the statement.\footnote{One simply needs to apply the argument with a constant $K':=K+4$.}

\subsubsection{Bounds on Norms $s_i, r_i$}
Recall that we defined $r_i = \|Z_i-\sum_{j=1}^{i-1}\la Z_i, \Zt_j\ra \Zt_j\|$ and $s_i = \|X_i-\sum_{j=1}^{i-1}\la X_i, \Xt_j\ra \Xt_j\|$. Denote $A_i := X_i-\sum_{j=1}^{i-1}\la X_i, \Zt_j\ra \Zt_j$. Define the events
\begin{align}
\Enorm 0 = \Eg &= \b\{\|g\|^2\in [n \pm \cc0 n^{1/2} ] \b\}\quad \text{and}\nonumber\\
\Enorm i &=
\b\{\|X_{i}\|^2 \in [n \pm \cc0 n^{1/2} ] \b\}
 \cap \B\{\|A_i\|^2 \in [(n-i+1) \pm \cc0 (n-i+1)^{1/2} ] \B\} \nonumber\\
 &\qquad\qquad\qquad\cap \B\{s_i^2 \in [(n-i+1) \pm \cc0 (n-i+1)^{1/2} ]\B\} \cap \Enorm {i-1},\quad i \in [d] \,.\nonumber
\end{align}

First, since $g \sim \N(0, I_n)$, by Cor.~\ref{lem:gausidnorm}, we have $\P(\Enorm 0) = \P(\Eg) \geq 1 - n^{-K}$. We claim that Cor. \ref{lem:gausidnorm} then implies
$$\P(\Enorm i) \geq 1 - (1+3i)n^{-K} \geq 1 - n^{-K+1}\,.$$
Indeed, we have $X_i \sim \N(0, I_n)$, so $\P(\|X_{i}\|^2 \in [n \pm \cc0 n^{1/2} ])  \geq 1 - n^{-K}$ for every $i \in [d]$. Moreover, $A_i = X_i - \sum_{j=1}^{i-1} \la X_{i}, \Zt_j\ra \Zt_j$ is a projection of $X_i$ onto a $(n-i+1)$-dimensional subspace, and hence by Cor.~\ref{lem:gausidnorm}, $\P(\|A_i\|^2 \in [(n-i+1) \pm \cc0 (n-i+1)^{1/2} ]) \geq 1 - n^{-K}$ for any $g, X_1,\dots,X_{i-1}$. Similarly, $\P(s_i^2 = \|X_i - \sum_{j=1}^{i-1} \la X_{i}, \Xt_j\ra \Xt_j\|^2 \in [(n-i+1) \pm \cc0 (n-i+1)^{1/2} ]) \geq 1 - n^{-K}$. 

By the union bound applied to $\Enorm {i-1}$ and the three events above, $\P(\Enorm i) \geq 1 - (1+3i)n^{-K}$ for $n\geq d^{1+\epsilon} \geq 1+3d$.

\begin{lemma}[Bound on $r_{i}, s_i$]\label{lem:basis_norms}
    For each $ i\in [d]$, given $\Enorm i$, we have 
    \begin{align}
        \label{e:norm_r}\b|r_i - (n-i+1)^{1/2}\b| &\leq  2 \cc0 + 18 c_u \tsq k^{-1/2} n^{1/2} \mathds{1}_{i\in S}\\
        \label{e:norm_s}\b|s_i - (n-i+1)^{1/2}\b| &\leq 2\cc0\\
        \label{e:norm_s_inv}\b|s_i^{-1} - n^{-1/2}\b| &\leq 16 \cc0 n^{-1}\\
        \label{e:norm_rs}\b|s_i r_i^{-1} - 1\b| &\leq 8 \cc0 n^{-1/2} + 36 c_u \tsq k^{-1/2} \mathds{1}_{i\in S}\,.
    \end{align}

\end{lemma}

\begin{proof}
Recall that we defined $r_i = \|Z_i-\sum_{j=1}^{i-1}\la Z_i, \Zt_j\ra \Zt_j\|.$ Plugging in $Z_i= X_i+\tsq u_i g $, we get
$$Z_{i} - \sum_{j=1}^{i-1} \la Z_{i}, \Zt_j\ra \Zt_j = \underbrace{\b(X_{i} - \sum_{j=1}^i \la X_{i}, \Zt_j\ra \Zt_j\b)}_{A_i} + \underbrace{\b(\tsq u_{i} g  - \sum_{j=1}^{i-1} \la \tsq u_{i} g, \Zt_j\ra \Zt_j\b)}_{B_i}\,,$$
and therefore,
$$
r_i^2 = \|A_i + B_i\|^2 = \|A_i\|^2 + \|B_i\|^2 + 2 \la A_i, B_i\ra\,.
$$
Given $\Enorm{i}$ defined above, we have that $\|A_i\|^2 \in [(n-i+1) \pm \cc0 (n-i+1)^{1/2} ]$ and
$$ \|B_i\|^2 = \|\tsq u_i g - \sum_{j=1}^{i-1} \la \tsq u_{i} g, \Zt_j\ra \Zt_j\|^2 \leq  (\tsq u_i)^2  \|g\|^2 \leq (\tsq u_i)^2 \b( n + \cc0 n^{1/2}\b) \leq 2 \theta u_i^2 n\,,$$
for $n^{1/2} \geq \cc0$. Finally, by Cauchy-Schwarz, for $(n-i+1)^{1/2}\geq n^{1/2}/\sqrt{2} \geq \cc0$,
\begin{align*}
    2|\la A_i, B_i\ra| &\leq 2 \|A_i\| \|B_i\|\\
    &\leq 2 \sqrt{(n-i+1) + \cc0 (n-i+1)^{1/2}}\tsq |u_i|\sqrt{n + \cc0 n^{1/2}} \\
    &\leq 8 \tsq |u_i| n\,.
\end{align*}
Combining these results we obtain that given $\Enorm i$,
\begin{align*}
    \b|r_i^2 - (n-i+1)\b| &\leq \cc0 (n-i+1)^{1/2} + 2\theta u_i^2 n + 8\tsq |u_i| n\\
    &\leq \cc0 n^{1/2} + ( 2 (c_u)^2 \theta k^{-1} + 8 c_u \tsq k^{-1/2})n\mathds{1}_{i\in S}\\
    &\leq \cc0 n^{1/2} + 9c_u\tsq k^{-1/2} n \mathds{1}_{i\in S}\,,
\end{align*}
since $\tsq k^{-1/2} \leq d^{-\epsilon/2} \leq (2c_u)^{-1}$ for sufficiently large $d$. Therefore, since $n-i+1 \geq n - d +1 \geq n/2$, we obtain \eqref{e:norm_r}:
\begin{align*}
    \b|r_i - (n-i+1)^{1/2}\b| &\leq \B|\frac{r_i^2 - (n-i+1)}{r_i + (n-i+1)^{1/2}}\B|\\
    &\leq 2 \frac{\cc0 n^{1/2} +9c_u\tsq k^{-1/2} n \mathds{1}_{i\in S}}{n^{1/2}} \\
    &= 2 \cc0 + 18 c_u \tsq k^{-1/2} n^{1/2} \mathds{1}_{i\in S}\,.
\end{align*}
Given $\Enorm i$, we have 
$$
\b|s_i^2 - (n-i+1)\b| \leq \cc0 (n-i+1)^{1/2} \leq \cc0 n^{1/2}\,,
$$
and by an argument similar to above we establish \eqref{e:norm_s}, \eqref{e:norm_s_inv}:
$$
\b|s_i - (n-i+1)^{1/2}\b| \leq 2\cc0, \qquad \b|s_i^{-1} - n^{-1/2}\b| \leq 16\cc0 n^{-1} \,.
$$
We conclude that since $r_i \geq \frac14 n^{1/2}$, \eqref{e:norm_rs} also holds:
\begin{equation*}\b|s_i r_i^{-1} - 1\b| = \B|\frac{s_i - (n-i+1)^{1/2}}{r_i} + \frac{(n-i+1)^{1/2} - r_i}{r_i}\B| \leq 8 \cc0 n^{-1/2} + 36 c_u \tsq k^{-1/2} \mathds{1}_{i\in S}\,.
\end{equation*}
\end{proof} 

\subsubsection{Induction Argument}

We prove the statement of the theorem by induction on $i.$ 

\subsubsection{Base Case $i=1$.}
Note that $s_1\Xt_1 = X_1$. Hence, writing $\Zt_1$ in the form given in the theorem statement,
\begin{align*}
    \Zt_1 = \frac{1}{r_1} Z_1 = \frac{1}{r_1} (s_1 \Xt_1 + \tsq u_1 g) 
    &= \frac{s_1}{r_1} (\Xt_1 + \tsq n^{-1/2} (u_1+ (n^{1/2}s_1^{-1}-1)u_1)g)
    \\ &=\frac{s_1}{r_1} (\Xt_1 + \tsq n^{-1/2} (u_1+ \alpha_1)g + W_1)\,,
\end{align*}
where we have defined $W_1=0$ and $\alpha_1 := \b(n^{1/2}s_1^{-1}-1\b)u_1$. 
Let $\EE 1 = \Enorm1\subset \Eg$, in which case from \eqref{e:norm_s_inv} in Lem.~\ref{lem:basis_norms}, 
$$|\alpha_1| = |u_1| \b|n^{1/2}s_1^{-1} - 1\b| \leq 16\cc0 c_u n^{-1/2} k^{-1/2} \mathds{1}_{1\in S}\,.$$ It suffices to choose any $\an\geq 0$, $\as \geq 16\cc0 c_u n^{-1/2}k^{-1/2}$ to satisfy the conditions of Thm.~\ref{thm:gs-pert}.

\subsubsection{Induction step $i-1 \rightarrow i$.}
\begin{lemma}[Recursion for $\alpha_i$ and $W_i$]\label{lem:rec}
We can write 
    $$\Zt_{i} = \frac{s_{i}}{r_{i}} \b(\Xt_{i} + \tsq n^{-1/2}(u_{i}+\alpha_{i})g + W_{i}\b),$$ where
\begin{align*}
    \alpha_{i} &= (s_{i}^{-1}n^{1/2} - 1)u_{i} - s_{i}^{-1}\sum_{j=1}^{i-1} \la Z_{i}, \Zt_j\ra s_j r_j^{-1} (u_j + \alpha_j),\quad \text{and}\\
    W_{i} &= - s_{i}^{-1}\sum_{j=1}^{i-1} \b\la \tsq u_{i}g, \Zt_j\b\ra s_jr_j^{-1} \b(\Xt_j + W_j\b)\\
    &\qquad -
    s_{i}^{-1}\sum_{j=1}^{i-1} \la X_{i}, \Zt_j\ra \b((s_jr_j^{-1}-1)\Xt_j + s_jr_j^{-1} W_j\b)\\  
    &\qquad- s_{i}^{-1}\sum_{j=1}^{i-1} \B\la X_{i}, (s_{i}r_{i}^{-1} - 1)\Xt_{i} + s_{i}r_{i}^{-1} \b(\tsq n^{-1/2}(u_{i}+\alpha_{i})g + W_{i} \b)\B\ra \Xt_j\,.
\end{align*}
\end{lemma}

\begin{proof}

$\Zt_{i}, \Xt_{i}$ are the basis vectors on the $i$-th step of the standard GS process, so they satisfy the following recursive relation:
$$
\Zt_{i} = \frac{Z_{i} - \sum_{j=1}^{i-1} \la Z_{i}, \Zt_j\ra \Zt_j}{\|Z_{i} - \sum_{j=1}^{i-1} \la Z_{i}, \Zt_j\ra \Zt_j\|}, \qquad
\Xt_{i} = \frac{X_{i} - \sum_{j=1}^{i-1} \la X_{i}, \Xt_j\ra \Xt_j}{\|X_{i} - \sum_{j=1}^{i-1} \la X_{i}, \Xt_j\ra \Xt_j\|}\,.
$$
In our notation we can rewrite this as
\begin{equation}
    \label{e:GS-Ztilde}
    \Zt_{i} = r_{i}^{-1} \b( Z_{i} - \sum_{j=1}^{i-1} \la Z_{i}, \Zt_j\ra \Zt_j\b), \qquad
\Xt_{i} = s_{i}^{-1}\b(X_{i} - \sum_{j=1}^{i-1} \la X_{i}, \Xt_j\ra \Xt_j\b)\,.
\end{equation}

We aim to rewrite $\Zt_{i}$ in the form
$\Zt_{i} = \frac{s_{i}}{r_{i}} \b(\Xt_{i} + \tsq n^{-1/2}(u_{i}+\alpha_{i})g + W_{i}\b),$
where $\alpha_{i}, W_{i}$ are expressed in terms of $\{ X_j, \Xt_j, Z_j, \Zt_j, \alpha_j, W_j\}_{j=1}^{i-1}, X_{i}.$

We now make the inductive assumption that for all $j\leq i-1$, 
\begin{align*}
    \Zt_j &= \Xt_j + \underbrace{(s_{j}r_{j}^{-1} - 1)\Xt_{j} + s_{j}r_{j}^{-1} \b(\tsq n^{-1/2}(u_{j}+\alpha_{j})g + W_{j} \b)}_{\Delta_j}\,.
\end{align*}
Recalling that $Z_{i} = X_{i} + \tsq u_{i} g$ and
substituting the last displayed equation into~\eqref{e:GS-Ztilde} we get
\begin{align*}
    r_{i}\Zt_{i} &=   Z_{i} - \sum_{j=1}^{i-1} \la Z_{i}, \Zt_j\ra \Zt_j
    \\
    &=   X_{i}+\tsq u_{i}g - \sum_{j=1}^{i-1} \la X_{i}+\tsq u_{i}g, \Xt_j+\Delta_j\ra (\Xt_j+\Delta_j)\\
    &=  X_{i} - \sum_{j=1}^{i-1} \la X_{i}, \Xt_j\ra \Xt_j +\tsq u_{i}g - \sum_{j=1}^{i-1} \la \tsq u_{i}g, \Xt_j+\Delta_j\ra (\Xt_j+\Delta_j)\\
    & \qquad \qquad- \sum_{j=1}^{i-1} \la X_{i}, \Xt_j+\Delta_j\ra \Delta_j  - \sum_{j=1}^{i-1} \la X_{i}, \Delta_j\ra \Xt_j\\
    &=  s_{i} \Xt_{i} +\tsq u_{i}g - \sum_{j=1}^{i-1} \la \tsq u_{i}g, \Zt_j\ra \Zt_j - \sum_{j=1}^{i-1} \la X_{i}, \Zt_j\ra \Delta_j  - \sum_{j=1}^{i-1} \la X_{i}, \Delta_j\ra \Xt_j\,.
\end{align*}
We now extract all terms that include vector $g$:
\begin{align*}
   r_{i} \Zt_{i} 
    &=  s_{i} \Xt_{i}\\ 
    &\qquad+ g\cdot \B(\tsq u_{i} - \sum_{j=1}^{i-1} \la \tsq u_{i}g, \Zt_j\ra s_jr_{j}^{-1} \tsq n^{-1/2} (u_j+\alpha_j) - 
    \sum_{j=1}^{i-1} \la X_{i}, \Zt_j\ra s_jr_{j}^{-1}\tsq n^{-1/2} (u_j+\alpha_j)\B) \\
    &\qquad- \sum_{j=1}^{i-1} \la \tsq u_{i}g, \Zt_j\ra \b(\Zt_j - s_jr_{j}^{-1} \tsq n^{-1/2} (u_j+\alpha_j)g\b)-
    \sum_{j=1}^{i-1} \la X_{i}, \Zt_j\ra \b(\Delta_j - s_jr_{j}^{-1}\tsq n^{-1/2} (u_j+\alpha_j)g\b)\\  
    &\qquad- \sum_{j=1}^{i-1} \la X_{i}, \Delta_j\ra \Xt_j\,.
\end{align*}
Rearranging the terms we bring $\Zt_{i}$ to the desired form:
\begin{align*}
    \Zt_{i} 
    &= r_{i}^{-1}s_{i} \B(  \Xt_{i} + g\cdot\tsq n^{-1/2}  \b(u_{i} + (s_{i}^{-1}n^{1/2}-1)u_{i} - s_{i}^{-1}\sum_{j=1}^{i-1} \la Z_{i}, \Zt_j\ra s_jr_{j}^{-1} (u_j+\alpha_j)\b) \\
    &\qquad\qquad - s_{i}^{-1}\sum_{j=1}^{i-1} \la \tsq u_{i}g, \Zt_j\ra s_jr_j^{-1} \b(\Xt_j + W_j\b)\\
    &\qquad\qquad -
    s_{i}^{-1}\sum_{j=1}^{i-1} \la X_{i}, \Zt_j\ra \b((s_jr_j^{-1}-1)\Xt_j + s_jr_j^{-1} W_j\b)\\  
    &\qquad\qquad- s_{i}^{-1}\sum_{j=1}^{i-1} \b\la X_{i}, (s_{j}r_{j}^{-1} - 1)\Xt_{j} + s_{j}r_{j}^{-1} \b(\tsq n^{-1/2}(u_{j}+\alpha_{j})g + W_{j} \b)\b\ra \Xt_j\B)\,.
\end{align*}
The lemma statement follows.
\end{proof}

We have already defined events $\Enorm d \subseteq \dots \subseteq \Enorm 1 \subseteq \Enorm 0 = \Eg$ that guarantee the desired behavior of $s_i, r_i$ and proved $\Pr[\Enorm d] \geq 1 - n^{-K+1}$. The remainder of the proof is structured as follows. Using the recursive expressions derived in Lem.~\ref{lem:rec}, in Lemmas~\ref{lem:gX}, \ref{lem:gZ} we prove bounds on $\la g, \Xt_j\ra$ and $\la g, \Zt_j\ra$ given a high probability event $\EE{i}^{\la g, \Xt\ra}$. Lemmas~\ref{lem:alpha} and \ref{lem:W} show the bounds on $|\alpha_i|$ and $|\la g, W_i\ra|, \|W_i\|$  given high probability events $\EE{i}^{\alpha}, \EE{i}^{W}$ correspondingly. These high probability events are such that
\begin{enumerate}
    \item for $\star \in \{ \la g, \Xt\ra, \alpha, W\}$, event $\EE{i}^{\star}$ only depends on $g,X_1,\dots,X_i$ and  $\EE{d}^{\star} \subseteq \dots \subseteq \EE{1}^{\star}$,
    \item $\forall \ i \in [d]$, $\EE{i}^{W} \subseteq \EE{i}^{\alpha} \subseteq \EE{i}^{\la g, \Xt\ra} \subseteq \Enorm i$,
    \item $\Pr[\EE{d}^{W}] \geq 1-n^{-K+4}$.
\end{enumerate}
We then set $\EE = \EE{d}^{W}$ to finish the proof.

\begin{lemma}\label{lem:gX}
    Define
    $$
    \EE{j}^{\la g, \Xt\ra} = \b\{|\la g, \Xt_j\ra|\leq 8\cc2 \b\} \cap \EE{j-1}^{\la g, \Xt\ra}\cap \Enorm j\,.
    $$
    Event $\EE{j}^{\la g, \Xt\ra}$ only depends on the randomness in $g,X_1,\dots,X_j$ and satisfies
    $$
    \Pr\b[\EE{j}^{\la g, \Xt\ra} \b] \geq 1 - j n^{-K+1} - 2j n^{-K}\geq 1 - n^{-K+2}\,.
    $$
\end{lemma}
\begin{proof}
    We prove these statements by induction on $j$. The base case $j=1$ follows the same proof as the step, which we present here. From recursive Gram-Schmidt relation for basis vectors $\Xt_j$ we have
    \begin{align*}
        s_j\la g, \Xt_j\ra =  \la g, X_j \ra - \sum_{k=1}^{j-1} \la X_j, \Xt_k\ra \la g, \Xt_k\ra\,.
    \end{align*}
    Assume $\Enorm{j} \subseteq \Eg, \EE{j-1}^{\la g, \Xt\ra}$ hold. Recall that given $\Eg$ we have $\|g\|^2 \leq n + \cc0 n^{1/2} \leq 2n$. $\la g, X_j\ra \sim \N(0, \|g\|^2)$, and therefore, by Lem.~\ref{lem:gaustail},
    $$
    \Pr\b[|\la g, X_j\ra| \geq \sqrt{2}\cc2 n^{1/2}\b] \leq \Pr\b[|\la g, X_j\ra| \geq \cc2 |g|\b] \leq n^{-K}\,.
    $$
    Note that $\sum_{k=1}^{j-1} \la X_j, \Xt_k\ra \la g, \Xt_k\ra \sim \N(0, \sigma_j^2)$, where given $\EE{j-1}^{\la g, \Xt\ra}$,
    $$\sigma_j = \sqrt{\sum_{k=1}^{j-1} \la g,\Xt_k\ra^2} \leq j^{1/2} \cdot 8 \cc2 \,,$$
    and therefore, with probability at least $1-n^{-K}$, 
    $$
    \B|\sum_{k=1}^{j-1} \la X_j, \Xt_k\ra \la g, \Xt_k\ra\B| \leq \cc2 \sigma_j \leq 8 (\cc2)^2 j^{1/2}\,.
    $$
    Finally, by Lem.~\ref{lem:basis_norms}, given $\Enorm j$, $s_j^{-1} \leq 2n^{-1/2}$, and therefore, with probability $\geq 1 - 2n^{-K}$,
    $$
    |\la g, \Xt_j\ra| \leq s_j^{-1}\b|\la g, X_j \ra\b| + s_j^{-1} \B|\sum_{k=1}^{j-1} \la X_j, \Xt_k\ra \la g, \Xt_k\ra\B| \leq 2\sqrt{2}\cc2 + 16 (\cc2)^2 j^{1/2}n^{-1/2}  \leq 8 \cc2 \,,
    $$
    since $j^{1/2}n^{-1/2} \leq d^{-\epsilon/2} = o(1)$. By union bound we have $\Pr[\EE{j}^{\la g, \Xt\ra}|\EE{j-1}^{\la g, \Xt\ra}, \Enorm{j}] \geq 1 - 2n^{-K}$, and since $\Pr[\Enorm{j}] \geq 1 - n^{-K+1}$, by union bound applied to $\EE{j-1}^{\la g, \Xt\ra}, \Enorm{j}$,
    $$
    \Pr[\EE{j}^{\la g, \Xt\ra}] \geq 1 - j n^{-K+1} - 2j n^{-K}\geq 1 - n^{-K+2}\,.
    $$
\end{proof}

\begin{lemma}[Bound on $\la g, \Zt_j\ra$]\label{lem:gZ}
    Given $\EE{j}^{\la g, \Xt\ra}$,
    $$
    \b|\la g, \Zt_j\ra - (\la g, \Xt_j\ra + \tsq n^{1/2} u_j)\b| \leq \begin{cases}
        4 \tsq n^{1/2} \an + 2\wn + 64\cc0\cc2 n^{-1/2}\\
        \qquad= O(\tsq n^{1/2} \an + \wn) + \tO(n^{-1/2}) = o(1) &\text{if } j\not\in S\,,\\
        4 \tsq n^{1/2} \as + 2\ws + 75 (c_u)^2 \theta n^{1/2} k^{-1}\\
        \qquad= O(\tsq n^{1/2} \as + \ws + \theta n^{1/2} k^{-1}) &\text{if } j\in S\,.
    \end{cases}
    $$
\end{lemma}
Substituting the explicit values for $\wn,\ws$ from Lem.~\ref{lem:const_exist}, we obtain:
\begin{mylem}{\ref{lem:gZ_strong}} \boldnormal{(Spike in $\Zt$)}
    Given $\EE{j}^{\la g, \Xt\ra}$,
    $$\b|\la g, \Zt_j\ra - (\la g, \Xt_j\ra + \tsq n^{1/2} u_j)\b| = \begin{cases}
         \tO(\tsq)
        &\text{if } j\not\in S\,,\\
        \tO(\theta n^{1/2} k^{-1} + \tsq n^{-1/2} d k^{-1/2}+\theta^{3/2} n^{1/2} k^{-1/2}),
        &\text{if } j\in S\,.
    \end{cases}
    $$
\end{mylem}
\begin{proof}[of Lemma~\ref{lem:gZ_strong}] The expressions of Lem.~\ref{lem:gZ} contain $\an, \as, \wn, \ws$. Here we use the final expressions for these quantities derived in Lem.~\ref{lem:const_exist} to obtain explicit bounds in terms of $d,k,\theta,n$.
\begin{enumerate}
    \item ($j\not\in S$) From Lem.~\ref{lem:gZ},
    $$
    \b|\la g, \Zt_j\ra - (\la g, \Xt_j\ra + \tsq n^{1/2} u_j)\b| = O(\tsq n^{1/2} \an + \wn) + \tO(n^{-1/2})\,.
    $$
    First, since $\theta \geq \thetastat = \sqrt{k/n}$ and we choose $\an = 16\cc2 \sqrt{c_u} n^{-1/2}$ in Lem.~\ref{lem:const_exist}, we have $O(\tsq n^{1/2} \an) + \tO(n^{-1/2}) = \tO(\tsq)$. We now show that $\wn = \tO(\tsq)$. From Lem.~\ref{lem:const_exist}, we have 
    \begin{align*}
        \wn &= \tO(n^{-1}d^{1/2} + \tsq n^{-1/2} + n^{-1/2}k^{1/2}\ws)\\
        &= \tO(\tsq) + n^{-1/2}k^{1/2}\tO(\tsq n^{-1/2} d k^{-1/2} + \theta)\\
        &= \tO(\tsq)\,,
    \end{align*}
    where in the first line we used $n^{-1}d^{1/2}\ll \tsq$ for $\theta\geq \thetastat$ and in the second we substituted $\ws = \tO(\tsq n^{-1/2} d k^{-1/2} + \theta)$ (Lem.~\ref{lem:const_exist}).
    \item ($j\in S$) From Lem.~\ref{lem:gZ},
    $$
    \b|\la g, \Zt_j\ra - (\la g, \Xt_j\ra + \tsq n^{1/2} u_j)\b| = O(\tsq n^{1/2} \as + \ws + \theta n^{1/2} k^{-1})\,.
    $$
    Substituting $\ws = \tO(\tsq n^{-1/2} d k^{-1/2} + \theta)$ and $\as = \tO(\tsq n^{-1} k^{-1/2} d + \theta k^{-1/2} + \tsq n^{-1/2} \ws)$ (Lem.~\ref{lem:const_exist}), we obtain
    \begin{align*}
        \b|\la g, \Zt_j\ra - (\la g, \Xt_j\ra + \tsq n^{1/2} u_j)\b| &= \tO(\theta n^{-1/2} k^{-1/2} d + \tsq n^{1/2} \theta k^{-1/2} + \tsq n^{-1/2}dk^{-1/2}\\
        &\qquad\qquad+  \theta + \theta n^{1/2} k^{-1})\\
        &= \tO(\theta^{3/2} n^{1/2} k^{-1/2} + \tsq n^{-1/2}dk^{-1/2} +  \theta n^{1/2} k^{-1})\,,
    \end{align*}
    where we used $\tsq \leq 1$ and $\theta\geq \thetastat$.
\end{enumerate}
\end{proof}

\begin{proof}[of Lemma~\ref{lem:gZ}] From our expression for $\Zt_j$~\eqref{eq:Zt}
$$
\la g, \Zt_j\ra = \frac{s_j}{r_j}\b( \la g, \Xt_j\ra + \tsq n^{-1/2}(u_j + \alpha_j)\|g\|^2 + \la g,W_j\ra\b)\,,
$$
and therefore, 
\begin{align*}
\la g, \Zt_j\ra - \b(\la g, \Xt_j\ra + \tsq n^{1/2} u_j \b) &= \b(s_j r_j^{-1} - 1\b)\la g, \Xt_j\ra \\
&\qquad + s_j r_j^{-1} \b( \tsq n^{-1/2}(u_j + \alpha_j)\| g\|^2 - \tsq n^{1/2} u_j \b) \\
&\qquad + \b(s_j r_j^{-1} - 1\b) \tsq n^{1/2} u_j + s_j r_j^{-1} \la g,W_j\ra\,.
\end{align*}
Recall that $\EE{j}^{\la g, \Xt\ra} \subset \EE{j}^{\text{norm}}$. Then from \eqref{e:norm_rs} in  Lem.~\ref{lem:basis_norms}, $\b|s_j r_j^{-1} - 1\b| \leq \Delta = 8\cc0 n^{-1/2} + 36c_u \tsq k^{-1/2}\mathds{1}_{j\in S} \leq 1$. From the high-probability bound on $|\la g, \Xt_j\ra|$ in Lem.~\ref{lem:gX} and the fact that given $\EE{j}^{\text{norm}}$, $|\|g\|^2 - n| \leq \cc0 n^{1/2}$, we have that given $\EE{j}^{\la g, \Xt\ra}$,
$$|\la g, \Xt_j\ra| \leq 8\cc2 ,\quad \b|\tsq n^{-1/2}(u_j + \alpha_j)\| g\|^2 - \tsq n^{1/2} u_j\b| \leq \cc0 \tsq |u_j + \alpha_j| + \tsq n^{1/2}|\alpha_j|\,.$$ 
Combining these bounds with $\|g\|^2 \leq n +\cc0 n^{1/2} \leq 2n$, we obtain that 
\begin{align*}
    \b|\la g, \Zt_j\ra - (\la g, \Xt_j\ra + \tsq n^{1/2} u_j)\b| &\leq \cc0 \tsq |u_j + \alpha_j| + \tsq n^{1/2}|\alpha_j| + |\la g, W_j\ra| \\
    &\qquad+ \Delta \B(8\cc2  + 2 \tsq n^{1/2} |u_j + \alpha_j| + |\la g, W_j\ra|\B)\\
    &\leq (\cc0 \tsq + \tsq n^{1/2} + 2\Delta \tsq n^{1/2})|\alpha_j| + (1 + \Delta)|\la g, W_j\ra| \\
    &\qquad + (\cc0 \tsq + 2\Delta \tsq n^{1/2})|u_j| + 8\cc2 \Delta \\
    &\leq 4 \tsq n^{1/2} |\alpha_j| + 2|\la g, W_j\ra|\\
    &\qquad +c_u(\cc0 \tsq + 2\Delta \tsq n^{1/2})k^{-1/2}\mathds{1}_{j\in S} + 8\cc2 \Delta\,.
\end{align*}

Applying the assumption of induction for $\alpha_j, \la g, W_j\ra$ we get
\begin{itemize}
    \item If $j\not\in S$:
    \begin{align*}
        \b|\la g, \Zt_j\ra - (\la g, \Xt_j\ra + \tsq n^{1/2} u_j)\b| &\leq 4 \tsq n^{1/2} \an + 2\wn + 64\cc0\cc2 n^{-1/2}\\
        &= O(\tsq n^{1/2} \an + \wn) + \tO(n^{-1/2}) = o(1)\,.
    \end{align*}
    \item If $j\in S$: Here, since $\theta \geq \thetastat = \sqrt{k/n}$, $\Delta = 8\cc0 n^{-1/2} + 36c_u \tsq k^{-1/2}\mathds{1}_{j\in S} \leq 37c_u \tsq k^{-1/2}$. Then
    \begin{align*}
        \b|\la g, \Zt_j\ra - (\la g, \Xt_j\ra + \tsq n^{1/2} u_j)\b| &\leq 4 \tsq n^{1/2} \as + 2\ws +c_u\cc0 \tsq k^{-1/2} + 74 (c_u)^2 \theta n^{1/2} k^{-1} \\
        &\qquad+ 296 \cc2 c_u \tsq k^{-1/2}\\
        &\leq 4 \tsq n^{1/2} \as + 2\ws + 75 (c_u)^2 \theta n^{1/2} k^{-1}\\
        &= O(\tsq n^{1/2} \as + \ws + \theta n^{1/2} k^{-1})\,,
    \end{align*}
    where in the last inequality we again used $\theta \geq \sqrt{k/n} \geq k/n \cdot d^{-\epsilon/2}$.%\qedhere
\end{itemize}
\end{proof}

\begin{lemma}[Bound on $\alpha_{i}$]\label{lem:alpha}
We can define an event $\EE{i}^{\alpha}\subseteq \EE{i}^{\la g, \Xt\ra}$ that $g, X_1,\dots,X_{i-1}$ only depends on the randomness of $g,X_1,\dots,X_{i}$ and satisfies 
$$\Pr[\EE{i}^{\alpha}]\geq 1 - in^{-K} - in^{-K+2} \geq 1-n^{-K+3}\,,$$
and given $\EE{i}^{\alpha}$,
    $$|\alpha_{i}| \leq \begin{cases}
        16 \cc2 \sqrt{c_u} n^{-1/2} = \tO(n^{-1/2})        &\text{if } i\not\in S\,,\\
        36c_u\cc2 \tsq n^{-1/2} k^{-1/2} d \an + 49(c_u)^3 \theta k^{-1/2} + 16(c_u)^2 \tsq n^{-1/2} \ws\\
        \quad= \tO(\tsq n^{-1/2} k^{-1/2} d \an) + O(\theta k^{-1/2} + \tsq n^{-1/2} \ws)
        &\text{if } i\in S,\,.
    \end{cases}
    $$
Note that in order for 
$$|\alpha_{i}| \leq \begin{cases}
        \an, & i\not\in S\,,\\
        \as,& i\in S
    \end{cases}
$$
it is sufficient to choose any $\an, \as$ that satisfy
\begin{align}
    \label{e:alpha_cond1}\an &\geq 16 \cc2 \sqrt{c_u} n^{-1/2}\\
    \label{e:alpha_cond2}\as &\geq 36c_u\cc2 \tsq n^{-1/2} k^{-1/2} d \an + 49(c_u)^3 \theta k^{-1/2} + 16(c_u)^2 \tsq n^{-1/2} \ws\,.
\end{align}
\end{lemma}

\begin{proof}
In the base case we have shown that taking $\EE{1}^{\alpha} = \EE{i}^{\la g, \Xt\ra}\subseteq \Enorm 1$ is sufficient to satisfy the conditions of the lemma in case $i=1$. We now prove the statement for general $i>1$. Copying the expression for $\alpha_i$ from Lemma \ref{lem:rec} and plugging in $Z_{i} = X_{i} + \tsq u_{i} g$ yields 
\begin{align*}
    \alpha_{i} &= (s_{i}^{-1}n^{1/2} - 1)u_{i} - s_{i}^{-1}\sum_{j=1}^{i-1} \la Z_{i}, \Zt_j\ra s_j r_j^{-1} (u_j + \alpha_j)\\
    &= \underbrace{u_{i} \cdot (n^{1/2}s_{i}^{-1} - 1)}_{\I} + \underbrace{u_{i} \cdot \tsq s_{i}^{-1}\sum_{j=1}^{i-1} \la g, \Zt_j\ra s_j r_j^{-1}  (u_j + \alpha_j)}_{\II} + \underbrace{s_{i}^{-1}\sum_{j=1}^{i-1} \la X_{i}, \Zt_j\ra s_j r_j^{-1} (u_j + \alpha_j)}_{\III}\,.
\end{align*}
We analyze the magnitude of each of these terms: 
\begin{itemize}
    \item $\I$. From Eq.~\eqref{e:norm_s_inv} in Lem.~\ref{lem:basis_norms}, given $\Enorm {i}$, $|s_{i}^{-1} - n^{-1/2}| \leq 16\cc0 n^{-1}$ and hence 
    \begin{equation}\label{e:Ibound}
        |\I|= |u_{i}| n^{1/2} |s_{i}^{-1} - n^{-1/2}| \leq 16\cc0c_u n^{-1/2} k^{-1/2} \mathds{1}_{i\in S}\,.
    \end{equation}
    \item $\II$. From the bounds on $|s_{i}^{-1} - n^{-1/2}|$ and $|s_j r_j^{-1} - 1|$ (Eq.~\eqref{e:norm_s_inv}, \eqref{e:norm_rs} in Lem.~\ref{lem:basis_norms}), given $\Enorm j$, $|s_j r_j^{-1}| \leq 2$ and $|s_i^{-1}|\leq n^{-1/2} + 16\cc 0 n^{-1} \leq 2 n^{-1/2}$. Additionally, from Lemma \ref{lem:gZ} we have that given $\EE{j}^{\la f, \Xt\ra}$,
    $$
    |\la g, \Zt_j\ra| \leq \begin{cases}
        |\la g, \Xt_j\ra| + o(1) \leq 9 \cc2 &\text{if } j\not\in S\,,\\
        |\la g, \Xt_j\ra| + \tsq n^{1/2}|u_j| + 4\tsq n^{1/2}\as + 2\ws + 75(c_u)^2 \theta n^{1/2} k^{-1}\\
        \qquad \leq 7c_u \tsq n^{1/2}k^{-1/2} + 2\ws &\text{if } j\not\in S\,,
    \end{cases}
    $$
    where we used $\as \ll k^{-1/2}$ and $\max\{\tsq k^{-1/2}, 1\} = o(\tsq n^{1/2} k^{-1/2})$. Substituting these bounds into the expression for $|\II|$, we get
    \begin{align}
        |\II| &\leq 2c_u \tsq n^{-1/2} k^{-1/2} \mathds{1}_{i\in S} \B( \sum_{j\in [i-1]\setminus S} 18 \cc2 \an\nonumber \\
        &\qquad\qquad+ \sum_{j\in [i-1]\cap S} 2(7c_u \tsq n^{1/2} k^{-1/2} + 2 \ws) \cdot 2 c_u k^{-1/2}\B)\nonumber\\
        &= 2c_u \tsq n^{-1/2} k^{-1/2} \mathds{1}_{i\in S} \B( 18 \cc2 d \an + 28 (c_u)^2 \tsq n^{1/2} + 8 c_u k^{1/2} \ws \B)\nonumber\\
        \label{e:IIbound}&=\B(36c_u\cc2 \tsq n^{-1/2} k^{-1/2} d \an  + 56(c_u)^3 \theta k^{-1/2} + 16(c_u)^2 \tsq n^{-1/2} \ws\B) \mathds{1}_{i\in S}
    \end{align}
    \item $\III.$ 
    Observe that conditioned on fixed $\{g,X_j\}_{j=1}^{i-1}$,
     term $\III$ is a Gaussian random variable with variance
    $$
\sigma_i^2:=    s_i^{-2}\sum_{j=1}^{i-1} \Big( \frac{s_j}{r_j}\B)^2 (u_j + \alpha_j)^2\,.
    $$
    Define the event
    \begin{equation}
        \EIII i = \{ |\III| \leq \cc2 \sigma_i \}
    \end{equation}
    and note that by Gaussian Tail Bound (Lemma \ref{lem:gaustail}) $\P( \EIII i) \geq 1- n^{-K}$. 
    Then, letting $\EE{i}^{\alpha} = \EE{i}^{\la g, \Xt\ra} \cap \EIII i \cap \EE{i-1}^{\alpha}$ we have that given $\EE{i}^{\alpha}$, 
    \begin{align}
        |\III| &\leq \cc2 s_i^{-1} \sqrt{\sum_{j=1}^{i-1} s_j^2 r_j^{-2} (u_j + \alpha_j)^2}\nonumber\\
        &\leq 4 \cc2 n^{-1/2} \sqrt{\sum_{j\in [i-1]\setminus S} (u_j + \alpha_j)^2+ \sum_{j\in [i-1]\cap S} (u_j + \alpha_j)^2}\nonumber\\
        &\leq 4 \cc2 n^{-1/2}= \sqrt{d \an^2+ 2c_u k k^{-1}}\nonumber \\
        &\leq 8 \cc2 n^{-1/2} (d^{1/2} \an + \sqrt{2c_u})\nonumber \\
        \label{e:IIIbound}&\leq 16 \cc2 \sqrt{c_u} n^{-1/2} \,,
    \end{align}
    where we applied Lem.~\ref{lem:basis_norms} to bound $s_i^{-1}, s_j r_j^{-1}$ and in the last line we used $d^{1/2}\an = O(d^{1/2}n^{-1/2}) = o(1)$.
\end{itemize}

By union bound, $\Pr[\EE{i}^{\alpha}] \geq 1 - in^{-K} - in^{-K+2}\geq 1 - n^{-K+3}$. We now combine $\I,\II$, and $\III$ to bound $|\alpha_{i}|$ given $ \EE{i}^{\alpha}$ in two cases:
\begin{itemize}
    \item $i\not\in S:$ Since $u_{i}=0$, $|\alpha_{i}| = |\III| \leq 16 \cc2 \sqrt{c_u} n^{-1/2} \log n = \tO(n^{-1/2})$.
    \item $i\in S:$ Combining \eqref{e:Ibound}, \eqref{e:IIbound}, and \eqref{e:IIIbound}, we obtain
    \begin{align}
        |\alpha_{i}| &\leq  |\I| + |\II| + |\III| \nonumber\\
        &\leq 16 \cc0 c_u n^{-1/2} k^{-1/2} + 36c_u\cc2 \tsq n^{-1/2} k^{-1/2} d \an  + \nonumber\\
        &\qquad\qquad+ 56(c_u)^3 \theta k^{-1/2} + 16(c_u)^2 \tsq n^{-1/2} \ws + 16 \cc2 \sqrt{c_u} n^{-1/2}\nonumber\\
        \label{alpha_ineq1}&\leq 36c_u\cc2 \tsq n^{-1/2} k^{-1/2} d \an + 49(c_u)^3 \theta k^{-1/2} + 16(c_u)^2 \tsq n^{-1/2} \ws\,,
    \end{align}
    where in \eqref{alpha_ineq1} we used $\max\{ n^{-1/2} k^{-1/2}, n^{-1/2}\} = o(\theta k^{-1/2})$.\qedhere
\end{itemize}
\end{proof}

\begin{lemma}[Bound on $\la g,W_{i}\ra, \|W_{i}\|.$]\label{lem:W} 
We can define an event $\EE{i}^{W}\subseteq \EE{i}^{\alpha}$, that only depends on the randomness of $g, X_1,\dots, X_{i}$, such that 
$$\P[\EE{i}^{W}] \geq 1 - 4i n^{-K} - i n^{-K+3} \geq 1 - n^{-K+4}$$
and given $\EE{i}^{W}$,
$$|\la g,W_{i}\ra|, \|W_{i}\| \leq \begin{cases}
        \wn, \qquad i\not\in S\,,\\
        \ws,\qquad i\in S\,.
    \end{cases}
$$
where $\wn, \ws$ satisfy:
$\wn = \tOm \B(n^{-1}d^{1/2} + \tsq n^{-1/2} + n^{-1/2} k^{1/2} \ws\B)$ and $\ws = \tOm \B(\wn + \tsq n^{-1/2} d k^{-1/2} + \theta + \tsq n^{-1/2} k^{1/2} \ws^2 \B)$ with explicit constants for $\tOm$ stated in Equations \eqref{W_final_noS} and \eqref{W_final_S}.
\end{lemma}

\begin{proof}
As shown in the base case, setting $\EE{1}^{W} = \EE{1}^{\alpha} \subseteq \Enorm 1$ is sufficient to satisfy the conditions of the lemma in case $i=1$. Recall from Lemma \ref{lem:rec} the expression for $W_i$,
    \begin{align*}
        W_{i} &= -\underbrace{s_{i}^{-1}\sum_{j=1}^{i-1} \la \tsq u_{i}g, \Zt_j\ra s_jr_j^{-1} \b(\Xt_j + W_j\b)}_{A_1}
        - \underbrace{
        s_{i}^{-1}\sum_{j=1}^{i-1} \la X_{i}, \Zt_j\ra \b((s_jr_j^{-1}-1)\Xt_j + s_jr_j^{-1} W_j\b)}_{A_2}\\  
        &\qquad- \underbrace{s_{i}^{-1}\sum_{j=1}^{i-1} \b\la X_{i}, (s_{j}r_{j}^{-1} - 1)\Xt_{j} + s_{j}r_{j}^{-1} \b(\tsq n^{-1/2}(u_{j}+\alpha_{j})g + W_{j} \b)\b\ra \Xt_j}_{A_3}\,.
    \end{align*}
We now analyze both $|\la g,W_{i}\ra|, \|W_{i}\|$ term by term. From Lemmas \ref{lem:basis_norms}, \ref{lem:gX} and \ref{lem:gZ} we have that given $\EE{i}^{\alpha}$,
\begin{gather*}
        |s_i^{-1}| \leq n^{-1/2}+16 \cc0 n^{-1} \leq 2n^{-1/2}, \quad |s_j r_j^{-1}| \leq 1 + 8\cc0 n^{-1/2} + 36c_u \tsq k^{-1/2} \leq 2\\
        |\la g, \Xt_j\ra|,1=|\Xt_j| \leq 8 \cc2, \quad \la g, \Zt_j\ra = \begin{cases}
         9 \cc2,
        &\text{if } j\not\in S\,,\\
        7c_u \tsq n^{1/2}k^{-1/2} + 2\ws,
        &\text{if } j\in S\,.
    \end{cases}
\end{gather*}
By induction assumption $$|\la g,W_{j}\ra|, \|W_{j}\|\leq \begin{cases}
        \wn, \qquad j\not\in S\,,\\
        \ws,\qquad j\in S\,,
\end{cases}$$ for $j< i$. 
\begin{itemize}
        \item $A_1.$ Our goal here is to analyze both $|\la g, A_1\ra|, \|A_1\|$. Substituting the bounds above into $A_1$ we get:
        \begin{align}
             |\la g, A_1\ra|, \|A_1\| &\leq 2c_un^{-1/2} \tsq k^{-1/2} \mathds{1}_{i\in S} \B(\sum_{j \in [i-1]\setminus S} 18 \cc2 (9\cc2 + \wn) \nonumber\\
             &\qquad\qquad+ \sum_{j \in [i-1]\cap S} 2(7c_u \tsq n^{1/2}k^{-1/2} + 2\ws) (9\cc2 + \ws) \B)\nonumber\\
             \label{w_ineq_1}&\leq 2c_un^{-1/2} \tsq k^{-1/2} \mathds{1}_{i\in S} \B( 180 (\cc2)^2 d \\
             &\qquad\qquad + (14c_u \tsq n^{1/2}k^{1/2} + 4k\ws) (9\cc2 + \ws) \B)\nonumber\\
             &\leq 2c_un^{-1/2} \tsq k^{-1/2} \mathds{1}_{i\in S} \B( 180 (\cc2)^2 d + 126c_u\cc2 \tsq n^{1/2}k^{1/2} \nonumber\\
             &\qquad\qquad + 36 \cc2 k\ws + 14c_u \tsq n^{1/2}k^{1/2}\ws + 4k\ws^2\B)\nonumber\\   
             \label{W_A_1}&= \B( 360c_u(\cc2)^2 \tsq n^{-1/2}  k^{-1/2} d + 252(c_u)^2 \cc2 \theta\\
             &\qquad\qquad + 72c_u \cc2  \tsq n^{-1/2} k^{1/2}\ws + 28(c_u)^2 \theta \ws + 8c_u \tsq n^{-1/2} k^{1/2}\ws^2\B) \mathds{1}_{i\in S}\nonumber
        \end{align}
        where in \ref{w_ineq_1} we used $\wn \leq 1$.
        \item $A_2.$ Our goal here is to analyze both $|\la g, A_2\ra|, \|A_2\|$. Observe that conditioned on $g, \{X_j\}_{j=1}^{i-1}$,
        $\la g, A_2\ra$ is a Gaussian random variable with variance $$\sigma_i^2 := s_i^{-2}\sum_{j=1}^{i-1} \b( (s_jr_j^{-1}-1) \la g, \Xt_j\ra + s_j r_j^{-1}\la g,W_j\ra\b)^2\,.$$
        Similarly, $A_2$ is an $n$-dimensional Gaussian random variable with diagonal covariance matrix $\Sigma_i$, where 
        $$
        \tr(\Sigma_i) = s_i^{-2}\sum_{j=1}^{i-1} \b( (s_jr_j^{-1}-1) \Xt_j + s_j r_j^{-1}W_j\b)^2\,.
        $$
        Define the event
        \begin{equation}
            \EE{i}^{A_2} = \{ |\la g,A_2\ra| \leq \cc2 \sigma_i \}\cap \{ |\|A_2\| - \sqrt{\tr(\Sigma_i)}| \leq 2\cc0 \sqrt{\tr(\Sigma_i)}\} \cap \EE{i}^{\alpha}
        \end{equation}
        and note that by Gaussian concentration (Lemmas \ref{lem:gausnorm}, \ref{lem:gaustail}), $\P( \EE{i}^{A_2}) \geq 1- 2n^{-K}$. 
        Substituting in bounds for $s_i, s_jr_j^{-1}, \la g,\Xt_j\ra$ we obtain given $\EE{i}^{A_2}$ as follows:
        \begin{align}
            |\la g, A_2\ra|, \|A_2\| &\leq (\cc2 + 2 \cc0 + 1) \max(\sigma_i, \sqrt{\tr(\Sigma_i)})\nonumber\\
            \label{w_ineq_2}&\leq 8\cc2 n^{-1/2} \B(\sqrt{\sum_{j \in [i-1]\setminus S} (64 \cc0\cc2 n^{-1/2} + 2 \wn)^2}\\
            &\qquad + \sqrt{\sum_{j \in [i-1]\cap S} \b((8\cc0 n^{-1/2} + 36c_u \tsq k^{-1/2})8\cc2 + 2 \ws \b)^2}\B)\nonumber\\
            \label{w_ineq_3}&\leq 8\cc2 n^{-1/2} \B(64 \cc0\cc2 n^{-1/2} d^{1/2} + 2 d^{1/2} \wn\\
            &\qquad + 296 c_u\cc2 \tsq+ 2 k^{1/2} \ws\B)\nonumber\\
            \label{W_A_2}&\leq  512\cc0(\cc2)^2 n^{-1}  d^{1/2} + 16\cc2 n^{-1/2} d^{1/2} \wn\\
            &\qquad + 2368 c_u(\cc2)^2\tsq  n^{-1/2} + 16\cc2 n^{-1/2} k^{1/2} \ws\nonumber\,,
        \end{align}
        where in \eqref{w_ineq_2} we used $\cc2 \geq 2(\cc0+1)$ (we can make such choice of $\cc2$ initially) and in \eqref{w_ineq_3} we used $n^{-1/2} = o(\tsq k^{-1/2})$.
        \item $A_3.$ Our goal here is to analyze both $|\la g, A_3\ra|, \|A_3\|$. 
        Observe that conditioned on $g, \{X_j\}_{j=1}^{i-1}$, $\la g, A_3\ra$ is a Gaussian random variable with variance 
        $$\sigma_i^2 = s_i^{-2}\sum_{j=1}^{i-1} \b( (s_jr_j^{-1}-1) \Xt_j + s_j r_j^{-1}(\tsq n^{-1/2}(u_j+\alpha_j)g + W_j)\b)^2 \la g, \Xt_j \ra^2\,.$$
        Similarly, $A_3$ is an $n$-dimensional Gaussian random variable with covariance $\Sigma_i$, where 
        $$\tr(\Sigma_i) := s_i^{-2}\sum_{j=1}^{i-1} \b( (s_jr_j^{-1}-1) \Xt_j + s_j r_j^{-1}(\tsq n^{-1/2}(u_j + \alpha_j)g + W_j)\b)^2\,.$$
        Define the event
        \begin{equation}
            \EE{i}^{A_3} = \{ |\la g,A_3\ra| \leq \cc2 \sigma_i \}\cap \{ | \|A_3\| - \sqrt{\tr(\Sigma_i)}| \leq 2\cc0 \sqrt{\tr(\Sigma_i)}\} \cap \EE{i}^{\alpha}
        \end{equation}
        and note that by Gaussian concentration (Lemmas \ref{lem:gausnorm} and~\ref{lem:gaustail}) $\P( \EE{i}^{A_3}) \geq 1- 2n^{-K}$. Then, similarly to the argument for bounding $A_2$,
        \begin{align}
            |\la g, A_3\ra|, \|A_3\| &\leq 2 \cc2 \max(\sigma_i, \sqrt{\tr(\Sigma_i)})\\
            &\leq 8 \cc2 n^{-1/2} \B( \sqrt{\sum_{j \in [i-1]\setminus S} \b((64\cc0 \cc2 n^{-1/2} + 2(2\tsq \an + \wn))8\cc2 \b)^2}\nonumber\\
            &\qquad +\sqrt{\sum_{j \in [i-1]\cap S} \b((296 \cc2 c_u \tsq k^{-1/2} + 2(4c_u \tsq k^{-1/2}  + \ws))8\cc2 \b)^2}\B)\nonumber\\
            \label{w_ineq_4}&\leq 8 \cc2 n^{-1/2}\B( d^{1/2} \b(64\cc0 \cc2 n^{-1/2} + 2\wn\b)8\cc2\\
            &\qquad +k^{1/2} \b(296 \cc2 c_u \tsq k^{-1/2} + 2 \ws\b)8\cc2\B)\nonumber\\
            \label{W_A_3}&\leq 4096\cc0 (\cc2)^3 n^{-1} d^{1/2} + 128 (\cc2)^2 d^{1/2} n^{-1/2} \wn \\
            &\qquad + 18944 (\cc2)^3 c_u \tsq n^{-1/2} + 128 (\cc2)^2 k^{1/2}n^{-1/2} \ws \nonumber
        \end{align}
        where in \eqref{w_ineq_4} we used $\tsq \an = \tsq O(n^{-1/2}) = o_{n^{O(\epsilon)}}(n^{-1/2})$.
    \end{itemize}
    Now let $\EE{i}^{W} = \EE{i}^{A_2}\cap\EE{i}^{A_3}\cap\EE{i-1}^{W}\cap \EE{i}^{\alpha}$. By union bound, $\Pr[\EE{i}^{W}] \geq 1 - 4i n^{-K} - i n^{-K+3} \geq 1 - n^{-K+4}$. We now combine bounds \eqref{W_A_1}, \eqref{W_A_2}, and \eqref{W_A_3} to get the desired bound in two cases: 
    \begin{itemize}
        \item $i\not\in S:$
        \begin{align}
            |\la g,W_{i}\ra|, \|W_{i}\| &\leq |A_2| + |A_3|\nonumber\\
            &\leq \B(512\cc0(\cc2)^2 n^{-1}  d^{1/2} + 16\cc2 n^{-1/2} d^{1/2} \wn\nonumber\\
            &\qquad + 2368 c_u(\cc2)^2\tsq  n^{-1/2} + 16\cc2 n^{-1/2} k^{1/2} \ws\B)\nonumber\\
            &\qquad+\B(4096\cc0 (\cc2)^3 n^{-1} d^{1/2}  + 128 (\cc2)^2 d^{1/2} n^{-1/2} \wn \nonumber\\
            &\qquad+ 18944 (\cc2)^3 c_u \tsq n^{-1/2}  + 128 (\cc2)^2 k^{1/2}n^{-1/2} \ws \B)\nonumber\\
            \label{W_final_noS}&\leq 4097\cc0 (\cc2)^3 n^{-1} d^{1/2}+ 129 (\cc2)^2 d^{1/2} n^{-1/2} \wn \\
            &\qquad+ 18945 (\cc2)^3 c_u \tsq n^{-1/2}  + 129 (\cc2)^2 k^{1/2}n^{-1/2} \ws \nonumber\\
            &= \tO(n^{-1}d^{1/2} + d^{1/2}n^{-1/2}\wn + \tsq n^{-1/2} + k^{1/2}n^{-1/2} \ws)\nonumber\,.
        \end{align}
        where we used $\log n \geq 4096(\cc2)^{-1}$.
        \item $i\in S:$ 
        \begin{align}
            |\la g,W_{i}\ra|, \|W_{i}\| &\leq |A_1| + |A_2| + |A_3|\nonumber\\
            \label{W_final_S}&\leq \B( 360c_u(\cc2)^2 \tsq n^{-1/2}  k^{-1/2} d + 252(c_u)^2 \cc2 \theta \\
             &\qquad + 72c_u \cc2  \tsq n^{-1/2} k^{1/2}\ws + 28(c_u)^2 \theta \ws + 8c_u \tsq n^{-1/2} k^{1/2}\ws^2\B)\nonumber\\
             &\qquad + \B(
             4097\cc0 (\cc2)^3 n^{-1} d^{1/2} + 129 (\cc2)^2 d^{1/2} n^{-1/2} \wn\nonumber\\
            &\qquad+ 18945 (\cc2)^3 c_u \tsq n^{-1/2} + 129 (\cc2)^2 k^{1/2}n^{-1/2} \ws
             \B)\nonumber\\
             &= \tO(\tsq n^{-1/2}k^{-1/2}d +\theta + \tsq n^{-1/2} k^{-1/2} \ws + \theta \ws + \tsq n^{-1/2} k^{1/2} \ws^2\nonumber\\
             &\qquad+n^{-1}d^{1/2} + d^{1/2}n^{-1/2}\wn + \tsq n^{-1/2} + k^{1/2}n^{-1/2} \ws)\nonumber \qedhere\,.
        \end{align}
    \end{itemize}
\end{proof}

We now let $\EE i := \EE{i}^{W}$ and finish the proof by showing the existence of $\an, \as, \wn, \ws$ that satisfy all the requirements. 

\begin{lemma}[Existence of $\an, \as, \wn, \ws$.]\label{lem:const_exist} There exist $\an, \as, \wn, \ws,$ s.t. $\an = \tO(n^{-1/2}), \as = o(k^{-1/2})), \wn = o(1), \ws = o(\min\{n^{1/4}, n^{1/2}k^{-1/2}\})$ and all of the conditions of Base Case and Lemmas \ref{lem:alpha}, \ref{lem:W} hold.
\end{lemma}

\begin{proof}
    Recall the conditions on $\an, \as, \wn, \ws$ we established throughout the proof: 
    \begin{itemize}
        \item (Theorem~\ref{thm:gs-pert} Conditions) $\an = \tO(n^{-1/2}), \as = o(k^{-1/2})), \wn = o(1)$, $\ws = o(\min\{n^{1/4}, n^{1/2}k^{-1/2}\})$.
        \item (Base Case) $\an\geq 0$, $\as \geq 16\cc0 c_u n^{-1/2}k^{-1/2}$.
        \item (Lemma~\ref{lem:alpha}) $\an \geq 16\cc2 \sqrt{c_u} n^{-1/2} $, $\as \geq 36c_u\cc2 \tsq n^{-1/2} k^{-1/2} d \an + 49(c_u)^3 \theta k^{-1/2} + 16(c_u)^2 \tsq n^{-1/2} \ws$.
        \item (Lemma~\ref{lem:W}) $\wn = \tOm\B(n^{-1}d^{1/2} + \tsq n^{-1/2} + n^{-1/2} k^{1/2} \ws\B)$ and $\ws = \tOm \B(\wn + \tsq n^{-1/2} d k^{-1/2} + \theta + \tsq n^{-1/2} k^{1/2} \ws^2 \B)$ with explicit constants for $\tOm$ stated in equations \eqref{W_final_noS}, \eqref{W_final_S}.
    \end{itemize}
    We set
    \begin{align*}
        \an &= 16\cc2 \sqrt{c_u} n^{-1/2} = \widetilde{\Theta}(n^{-1/2})\\
        \ws &= \max\{C_0' \tsq n^{-1/2} d k^{-1/2}, C_1'\theta\} = o\b(\min\{n^{1/4}, n^{1/2}k^{-1/2}\}\b)\\
        \as &= \max \{36c_u\cc2 \tsq n^{-1/2} k^{-1/2} d \an , 49(c_u)^3 \theta k^{-1/2}, 16(c_u)^2 \tsq n^{-1/2} \ws\} = o(k^{-1/2})\\
        \wn &= \max\{C_0 n^{-1}d^{1/2}, C_1\tsq n^{-1/2}, C_2 n^{-1/2}k^{1/2} \ws\}\b) = o(1)\,,
    \end{align*}
    where $C_0, C_1, C_2 = \tO(1)$ are the corresponding coeffitiens from eq.~\eqref{W_final_noS} and $C_0', C_1' = \tO(1)$ from eq.~\eqref{W_final_S}. The conditions of Thm.~\ref{thm:gs-pert} can be easily verified by substituting $\theta \leq \min\{k/\sqrt{n}, \sqrt{d/n}\} \ll 1$ and $k\leq d \ll n$.

    Base Case conditions and those of Lem.~\ref{lem:alpha} are immediately satisfied. To verify the conditions of Lem.~\ref{lem:W}, it is sufficient to show that for the chosen values
    $$
    \ws \gg \wn \quad \text{and} \quad \ws \ll \theta^{-1/2} n^{1/2}k^{-1/2}\,.
    $$
    Indeed, since $\theta\geq \thetastat$ and $n\geq d\geq k$, $$\wn = \tO(n^{-1}d^{1/2} + \tsq n^{-1/2} + n^{-1/2}k^{1/2} \ws) = o(\tsq n^{-1/2} d k^{-1/2}) + o(\ws) \ll \ws\,,$$
    and 
    $$
    \ws = \tO(\tsq n^{-1/2}dk^{-1/2} + \theta) = o(\theta^{-1/2}n^{1/2}k^{-1/2})\,,
    $$
    which concludes the proof. 
\end{proof}

\appendix

\section{Definitions of Average-Case Reductions for Detection}\label{sec:avg_case}

In this section we formally define the statistical detection problems and associated average-case reductions. Analogous definitions for recovery problems and reductions are in Sec.~\ref{sec:avg_case_recovery}. 

\subsection{Detection Problems}\label{subsec:detect_problem}
We identify two main ingredients of a planted problem: a \emph{collection of planted distributions} and a \emph{problem task}. In this section we focus on the detection task.

\paragraph{Collections of Planted Distributions.} We will define a planted problem on a collection of planted distributions $\PP {\mu,N} {\cU}$ over $\R^{D_N}$, where
\begin{align*}
    \PP {\mu} {\cU} = \b\{P_u^{\mu, N}\b\}_{u \in \cU}\,.
\end{align*}
Here $N$ is the main sequence index (in our case this will be dimension), $\mu\in \cM$ is a fixed problem parameter independent of $N$, $u\in \cU$ is the planted signal from an allowed collection $\cU = \cU_N$, and $P_u^{\mu, N}$ is a distribution over $\R^{D_N}$ corresponding to parameter $\mu$ and the planted signal $u$.
For both Spiked Wigner and Spiked Covariance Model we take $d$ as the sequence index and consider $\cU_d \in \{\ksparse d, \ksparseflat d\}$. For $\sc$ we consider $\PP {\mu,d} {\cU_d}$, where $\mu = (\alpha, \beta, \gamma)$ and for $\sw$ we consider $\QQ {\nu, d} {\cU_d}$, where $\nu = (\alpha, \beta)$.

Both models describe \textit{families} of such collections, as parameters $\mu, \nu$ vary. Each fixed point of $\mu$ or $\nu$ corresponds to a specific planted problem, whose complexity might differ from the one of a different planted problem with parameters $\mu'$ or $\nu'$. Moreover, every value of the dimensionality parameter $d$ implies a separate collection of distributions.

For simplicity, we often conflate the notation for a problem defined on a collection of distributions and the collection of distributions itself, i.e. $\PP {\mu, N} {\cU}$, when it is clear whether we consider detection or recovery from the context.

\paragraph{Detection Task.}

Fix $\cU = \cU_N$ and consider a collection of planted distributions $\PP {\mu, N} {\cU}$ on $\R^{D_N}$. Let $f:\cM\to\cM$ be a function mapping a parameter $\mu$ to a corresponding ``null hypothesis" parameter $\mu^0 = f(\mu)$. Denote the collections of distributions induced by $\mu$ and $\mu^0=f(\mu)$ as follows:
$$
\PP {\mu, N} {\cU, H_0} := \PP {\mu^0, N} {\cU},\quad \text{and} \quad \PP {\mu, N} {\cU, H_1} := \PP {\mu, N} {\cU}\,.
$$
The detection problem for $\PP {\mu, N} {\cU}$ is a task of distinguishing between $\PP {\mu, N} {\cU, H_0}$ and $\PP {\mu, N} {\cU, H_1}$ and is defined as follows.

\begin{definition}[Detection Problem]\label{def:detection_problem}
    Given a sample $X$ from an unknown distribution $P_{u^{\star}} \in \b\{\PP {\mu} {\cU, H_0} \bigcup \PP {\mu} {\cU, H_1}\b\}$, the detection problem is the task of distinguishing between events $H_0 = \{P_{u^{\star}} \in \PP {\mu} {\cU, H_0}\}$ and $H_1 = \{P_{u^{\star}} \in \PP {\mu} {\cU, H_1}\}$. In particular, an algorithm $\mathcal{A}: \R^{D_N} \to \{0,1\}$ solves the detection problem with error $\delta_N$ if the Type $I+II$ error does not exceed $\delta_N$, i.e.,
$$
\sup_{P_{u^{\star}} \in \PP {\mu,N} {\cU, H_0}} \P_{X \sim P_{u^\star}} \b[\mathcal{A}(X) = 1\b] + \sup_{P_{u^{\star}} \in \PP {\mu,N} {\cU, H_1}} \P_{X \sim P_{u^\star}} \b[\mathcal{A}(X) = 0\b] \leq \delta_N\,.
$$

\end{definition}

For both Spiked Wigner and Covariance Models, a natural detection problem is to distinguish between the \emph{null hypothesis}, where the signal-to-noise ratio (SNR) parameter $\lambda$ (corresp. $\theta$) is zero, and the \emph{planted hypothesis}, where $\lambda$ (corresp. $\theta$) take some nonzero value. Formally, for parameters $\mu = (\alpha, \beta, \gamma)$ for $\sc$ and $\nu = (\alpha, \beta)$ for $\sw$, we consider functions $\mu_0 = f(\mu) = (\alpha, -\infty, \gamma)$, $\nu_0 = g(\nu) = (\alpha, -\infty)$. In correspondence to the definition above, the detection task for $\sc(\mu)$ is to distinguish between $H_0: \sc(\mu_0)$ and $H_1: \sc(\mu)$ and for $\sw(\nu)$ between $H_0: \sw(\nu_0)$ and $H_1: \sw(\nu)$. 

Moreover, we say that an algorithm $\mathcal{A}$ solves the $\sc(\mu)$ (or $\sw(\nu)$) detection problem with Type $I+II$ error $\delta$ if in the definition above we have 
$$
\limsup_{N\to\infty} \delta_N \leq \delta\,.
$$

\subsection{Average-Case Reductions in Total Variation}\label{subsec:average-case}
We now formally define an average-case reduction between two detection problems. We start with a definition for mapping $\PP {\mu, N} {\cU}$ to $\QQ {\nu, N} {\cU'}$ for fixed dimensionality parameter $N$.

\begin{definition}[Average-Case Reduction for Detection from $\PP {\mu, N} {\cU}$ to $\QQ {\nu, N} {\cU'}$]\label{def:avg_case_detect}

Fix parameters $\mu, \nu$ and two collections of planted signals $\cU = \cU_N, \cU' = \cU'_N$. Let $\PP {\mu,N} {\cU}$ and $\QQ {\nu, N} {\cU'}$ be two collections of planted distributions over $\R^{D_N}$ and $\R^{D'_N}$ respectively. Consider the two induced detection problems with corresponding hypothesis $\b\{ \PP {\mu,N} {\cU, H_0}, \PP {\mu,N} {\cU, H_1} \b\}$ and $\b\{\QQ {\nu, N} {\cU', H_0}, \QQ {\nu, N} {\cU', H_1}\b\}$.

We say that a $O(N^C)$-time (possibly randomized) algorithm $\mathcal{A}^{\text{red}}: \R^{D_N} \to \R^{D'_N}$ for $C=O(1)$ is an \emph{average-case reduction with TV-error $\epsilon_N$ from $\PP {\mu, N} {\cU}$ to $\QQ {\nu, N} {\cU'}$} if for all $u \in \cU_N$, there exists some $u' \in \cU'_N$ satisfying:
    \begin{align*}
        \text{if } X \sim P^{\mu, N}_{u,H_0} \in \PP {\mu, N} {\cU,H_0},  \text{ then } & \tv(\mathcal{A}^{\text{red}}(X), {Q}^{\nu, N}_{u', H_0})\leq \epsilon_N\,, \\
        \text{if } X \sim P^{\mu, N}_{u,H_1} \in \PP {\mu, N} {\cU, H_1}, \text{ then } & \tv(\mathcal{A}^{\text{red}}(X), {Q}^{\nu, N}_{u', H_1})\leq \epsilon_N \,.
    \end{align*}

\end{definition}

The following lemma states the complexity implications of such a reduction. 

\begin{lemma}[Implication of Reduction for Detection from $\PP {\mu, N} {\cU}$ to $\QQ {\nu, N} {\cU'}$]\label{lem:red_detect0}
    Let $\PP {\mu, N} {\cU}$, $\QQ {\nu,N} {\cU'}$ be two detection problems defined as above. Assume there exists an $O(N^C)$-time average-case reduction for detection from $\PP {\mu, N} {\cU}$ to $\QQ {\nu,N} {\cU'}$ with parameter $\epsilon_N$ from the definition above. If there is an $O(N^M)$-time algorithm for detection for $\QQ {\nu,N} {\cU'}$ achieving Type $I+II$ error at most $\delta_N$, there exists a $O(N^{M} + N^C)$-time algorithm for detection for $\PP {\mu,N} {\cU}$ achieving Type $I+II$ error at most $2\epsilon_N + \delta_N$.
\end{lemma}

\begin{proof} 
    Given a sample $X$ from $P_u \in \b\{\PP {\mu, N} {\cU, H_0} \bigcup \PP {\mu, N} {\cU, H_1} \b\}$ and the $O(N^M)$-time detection algorithm $\mathcal{B}:\R^{D'_N} \to \{0, 1\}$, consider an algorithm $\mathcal{A}:\R^{D_N}$ defined as 
    $$
    \mathcal{A} = \mathcal{B} \circ \mathcal{A}^{\text{red}}(X)\,.
    $$
    Clearly, the runtime of $\mathcal{A}$ is a sum of the runtimes of the two subroutines $\mathcal{B}, \mathcal{A}^{\text{red}}$: $O(N^{M} + N^{C})$.
    From the definition of total variation distance,
    \begin{align*}
        \B| \P_{X' \sim {Q}^{\nu, N}_{u',H_0}}\b[ \mathcal{B}(X') = 1\b] - \P_{X \sim {P}^{\mu, N}_{u,H_0}}\b[ \mathcal{B}\b(\mathcal{A}^{\text{red}}(X)\b) = 1\b] \B| 
        &\leq \tv\b({Q}^{\nu, N}_{u',H_0}, \mathcal{A}^{\text{red}}(X) \b) 
        \\& \leq \epsilon_N, \text{ for } X\sim {P}^{\mu, N}_{u,H_0} 
        \\
        \B| \P_{X' \sim {Q}^{\nu, N}_{u',H_1}}\b[ \mathcal{B}(X') = 0\b] - \P_{X \sim {P}^{\mu, N}_{u,H_1}}\b[ \mathcal{B}\b(\mathcal{A}^{\text{red}}(X)\b) = 0\b] \B| 
        &\leq \tv\b({Q}^{\nu, N}_{u',H_1}, \mathcal{A}^{\text{red}}(X) \b) 
        \\&\leq \epsilon_N, \text{ for } X\sim {P}^{\mu, N}_{u,H_1}\,.
    \end{align*}
    This yields the following bound on the Type $I+II$ error for algorithm $\mathcal{A}$: 
    $$
    \P_{X \sim {P}^{\mu, N}_{u,H_0}}\b[ \mathcal{B}\b(\mathcal{A}^{\text{red}}(X)\b) = 1\b] + \P_{X \sim {P}^{\mu, N}_{u,H_1}}\b[ \mathcal{B}\b(\mathcal{A}^{\text{red}}(X)\b) = 0\b] \leq \delta_N + 2\epsilon_N\,,
    $$
    which concludes the proof.
\end{proof}

Notably, while Def.~\ref{def:avg_case_detect} of an average-case reduction is specific to a particular value of the dimensionality parameter $N$, we study the complexity of problems $\sc(\mu), \sw(\nu)$, where both $\sc(\mu)$ and $\sw(\nu)$ describe a \emph{sequence} of collections of distributions with $d = 1,2,\dots$. To address this, we define an average-case reduction between problems $\cP^\mu_{\cU} = \b\{\PP {\mu, N} {\cU_N}\b\}_{N=1}^{\infty}$ and $\cQ^\nu_{\cU'} = \b\{\QQ {\nu,N} {\cU'_N}\b\}_{N=1}^{\infty}$, where $\mu, \nu$ are independent of the dimensionality parameter $N$. 

\begin{definition}[Average-Case Reduction for Detection, Sequence]\label{def:avg_case_points}
Let $\cP^\mu_{\cU}  = \b\{\PP {\mu, N} {\cU_N}\b\}_{N=1}^{\infty}$ and $\cQ^\nu_{\cU'}  = \b\{\QQ {\nu,N} {\cU'_N}\b\}_{N=1}^{\infty}$ for fixed parameters $\mu, \nu$ independent of $N$. 
We say that an $O(N^{C})$-time algorithm $\mathcal{A} = \{\mathcal{A}_N:\R^{D_N}\to \{0,1\}\}$ is an \emph{average-case reduction for detection from $\cP^{\mu}_{\cU}$ to $\cQ^{\nu}_{\cU'}$} if
    for any sequence $\{\mu_i\}_{i=1}^{\infty}$, such that $\lim_{i\to\infty} \mu_i = \mu$, there exist a sequences $\{\nu_i, N_i\}_{i=1}^{\infty}$ and an index $i_0$, such that 
    \begin{enumerate}
        \item $\lim_{i\to\infty} \nu_i = \nu$ and $\lim_{i\to\infty} N_i = \infty$;
        \item for any $i > i_0$, $\mathcal{A}_{N_i}$ is an average-case reduction for detection from $\PP {\mu, N_i} {\cU_{N_i}}$ to $\QQ {\nu_i, N_i} {\cU'_{N_i}}$ with the total variation parameter $\epsilon_{N_i}$ from Def.~\ref{def:avg_case_detect}; 
        \item $\lim_{i\to\infty} \epsilon_{N_i} = 0$. 
    \end{enumerate}
Note that fully analogous definition for recovery is presented in Def.~\ref{def:avg_case_points_rec}.
\end{definition}

\subsection{Implications of Average-Case Reductions for Phase Diagrams}

We now apply Lemma~\ref{lem:red_detect0} to the Def.~\ref{def:avg_case_points} of an average-case reduction to obtain an implication of average-case reductions to
phase diagrams. We state the following lemma for the case of reductions between $\sc_{\gamma}$ and $\sw$ to simplify the notation\footnote{One can easily generalize the statement to the case of a reduction between any planted detection problems $\PP {\mu} {\cU}$ and $\QQ {\nu} {\cU'}$ with well-defined ``Hard" regions in their phase diagrams.}.

\begin{lemma}[Implication of Reduction for Detection for Phase Diagrams]\label{lem:red_detect}
    Consider the detection problems associated with $\sc_{\gamma}(\mu), \sw(\nu)$ for $\mu \in \Omc_{\gamma}, \nu \in \Omw$ and assume there exists an $O(d^C)$-time average-case reduction for detection from $\sc_{\gamma}(\mu)$ to $\sw(\nu)$ (Def.~\ref{def:avg_case_points}).
    Fix $\epsilon > 0$ and assume there is a $O(d^M)$-time algorithm for detection for $\sw(\nu')$ for all $\nu' \in B_{\epsilon}(\nu)$ achieving asymptotic Type $I+II$ error $\delta$. Then there exists a $O(d^{M} + d^C)$-time algorithm for detection for $\sc_{\gamma}(\mu)$ achieving asymptotic Type $I+II$ error at most $\delta$.

    An analogous implication holds for average-case reduction for detection from $\sw(\nu)$ to $\sc_{\gamma}(\mu)$.
\end{lemma}

\begin{proof}
    Consider any sequence $\{\mu_i = (\alpha_i ,\beta_i, \gamma)\}_{i=1}^{\infty}$, such that $\lim_{i\to\infty} \mu_i = \mu$ and $\beta_i \leq \beta$ (i.e. the SNR in $\mu_i$ is at most the SNR in $\mu$). From Def.~\ref{def:avg_case_points}, there exist sequences $\{\nu_i, d_i, \epsilon_{d_i}\}_{i=1}^{\infty}$ and an index $i_0$, such that $\lim_{i\to\infty} \nu_i = \nu$, $\lim_{i\to\infty} \epsilon_{d_i} = 0$, $\lim_{i\to\infty} d_i = \infty$ and for any $i > i_0$, there exists an average-case reduction for detection from $\sc_{\gamma} (\mu)$ in dimension $d_i$ to $\sw(\nu)$ in dimension $d_i$ with total variation parameter $\epsilon_{d_i}$ from Def.~\ref{def:avg_case_detect}. 
    
    Consider index $i'$, such that for all $i > i'$ we have $\nu_i \in B_{\epsilon}(\nu)$.
    From Lemma~\ref{lem:red_detect0}, we conclude that since there exists an $O(d^M)$-time algorithm for detection for $\sw(\nu_i)$ in dimension $d_i$ achieving an error $\delta_{d_i}$, there exists an $O(d^M+d^C)$-time algorithm for detection for $\sc_{\gamma} (\mu_i)$ in dimension $d_i$ achieving an error $\delta_{d_i} + 2\epsilon_{d_i}$. However, since $\beta_i \leq \beta$, the detection problem for $\mu_i$ is strictly \emph{harder} than the one for $\mu$, and therefore, there exists an $O(d^M+d^C)$-time algorithm for detection for $\sc_{\gamma} (\mu)$ in dimension $d_i$ achieving an error $\delta_{d_i} + 2\epsilon_{d_i}$. One, for example, could add additional noise to an instance of $\sc_{\gamma}(\mu)$ and apply the algorithm for $\sc_{gamma}(\mu_i)$.
    
    Finally, note that $$\lim_{i\to\infty} \delta_{d_i} + 2\epsilon_{d_i} \leq \delta + 0 = \delta\,,$$  
    which concludes the proof.
\end{proof}

\subsection{Average-Case Reductions on the Computational Threshold} 

As motivated in Sec.~\ref{sec:intro_avg_case} and \ref{sec:intro_equiv}, hardness results for points arbitrarily close to the computational thresholds of $\sc$ and $\sw$ have implications for the whole hard regimes $\Omc$ and $\Omw$. To capture this idea, we introduce the following definition of average-case reduction between the points on the computational thresholds $\CCc$ and $\CCw$. We emphasize that we do not consider complexity of the points \emph{exactly on the computational threshold}, rather points \emph{arbitrarily close to the thresholds but still inside the hard regions $\Omc$ and $\Omw$}. We state the appropriate complexity implications in Lemma~\ref{lem:red_comp_thr_impl}.

\begin{definition}[Average-Case Reduction on the Computational Threshold]\label{def:avg_case_points_threshold}
    Fix $\gamma\geq 1$ and parameters $\mu \in \CCc_{\gamma}, \nu \in \CCw$ at the computational thresholds of $\sc$ and $\sw$. We say that there is an average-case reduction for detection from $\sc(\mu)$ to $\sw(\nu)$ if for any sequence $\{\mu_i\}_{i=1}^{\infty}$, such that $\mu_i \in \Omc_{\gamma}$ and $\lim_{i\to\infty} \mu_i = \mu$, there exists a sequence $\{\nu_i\}_{i=1}^{\infty}$ and index $i_0$, such that
    \begin{enumerate}
        \item $\nu_i \in \Omw$ for all $i$; 
        \item $\lim_{i\to \infty} \nu_i = \nu$;
        \item for any $i > i_0$, there exists an average-case reduction for detection from  $\sc(\mu_i)$ to $\sw(\nu_i)$ (Def.~\ref{def:avg_case_points}).
    \end{enumerate}
    A reduction from from $\sw(\nu)$ to $\sc(\mu)$ is defined analogously.

\end{definition}

The following is a direct consequence of Lemma~\ref{lem:red_detect} - that is, implication of average-case reduction for detection for phase diagrams.

\begin{lemma}[Implication of Reduction on the Comp. Thr. for Phase Diagrams]\label{lem:red_comp_thr_impl}
Fix $\gamma\geq 1$ and consider parameters $\mu \in \CCc_{\gamma}$ and $\nu \in \CCw$ on the computational thresholds of $\sc$ and $\sw$. Assume there exists an $O(d^C)$-time average-case reduction for detection from $\sc_{\gamma}(\mu)$ to $\sw(\nu)$ (Def.~\ref{def:avg_case_points_threshold}). Fix $\epsilon > 0$ and assume there is a $O(d^M)$-time algorithm for detection for $\sw(\nu')$ for all $\nu' \in B_{\epsilon}(\nu)\cap \Omw$ achieving asymptotic Type $I+II$ error $\delta$. Then, for every $\epsilon' > 0$ there exists $\mu' \in B_{\epsilon'}(\mu) \cap \Omc_{\gamma}$ and an $O(d^{M} + d^C)$-time algorithm for detection for $\sc_{\gamma}(\mu')$ achieving asymptotic Type $I+II$ error at most $\delta$.

    An analogous implication holds for average-case reduction for detection from $\sw(\nu)$ to $\sc_{\gamma}(\mu)$.
\end{lemma}

\section{Definitions of Average-Case Reductions for Recovery} \label{sec:avg_case_recovery}

In this section we formally define the statistical recovery problems and associated average-case reductions. We extends our definitions for average-case reductions for detection from Sec.~\ref{sec:avg_case}. 

\subsection{Recovery Problems}\label{subsec:recover_problem}

Similarly to detection problems, the recovery problems are defined on a collection of planted distributions $\PP {\mu,N} {\cU}$ over $\R^{D_N}$. The difference lies in the task - signal recovery instead of the signal detection.

\paragraph{Recovery Task.} 
Fix parameter $\mu$ and a collection of allowed signals $\cU = \cU_N$. Let $\PP {\mu,N} {\cU}$ be a collection of planted distributions over $\R^{D_N}$ and let $\ell_N:\cU_N\times \widehat{\cU_N}\to [0,+\infty]$ be a loss function. 

\begin{definition}[Recovery Problem]\label{def:recovery_problem}
     Given a sample $X\in \R^{D_N}$ from an unknown distribution $P_{u^{\star}}^{\mu, N} \in \PP {\mu,N} {\cU}$, a \textbf{recovery problem on $\PP {\mu,N} \cU$ with a loss function $\ell_N$} is a task of producing an estimate $\widehat{u^{\star}} \in \widehat{\cU_N}$, such that the expected loss $\ell_N$ is minimized. In particular, an algorithm $\mathcal{A}: \R^{D_N} \to \widehat{\cU_N}$ solves the recovery problem with expected loss at most $\epsilon_N$ if
$$ \sup_{u^{\star} \in \cU_N} \E_{X\sim P_{u^{\star}}^{\mu, N}} \ell_N(u^{\star}, \widehat{u^{\star}} = \mathcal{A}(X)) \leq \delta_N\,,$$
for some $\delta_N \geq 0$.
\end{definition}

While most of our results for $\sw$ and $\sc$ are applicable to an arbitrary choice of the loss function $\ell_d$ and the threshold $l^{\star}_d$, our reduction at the computational threshold in Section~\ref{sec:gs_reduction} slightly alters the signal $u$, preserving the support and the strength. Thus, we slightly restrict the recovery problems to the following setting. 

We assume $\widehat{\cU_d} = \{u\in \mathbb{S}^{d-1}: \|u\|_0 \leq k\}$ - that is, the recovery algorithm $\mathcal{A}: \R^{D_d} \to \widehat{\cU_d}$ produces a $k$-sparse unit $\widehat{u^{\star}}$. Moreover, we let
$$
\ell_d(u^{\star},\widehat{u^{\star}}) = 1 - \b(\la u^{\star},\widehat{u^{\star}}\ra\b)^2\,,
$$
and let $l^{\star}_d = 1 - k^{-1/2} d^{2\delta}$ for some small $\delta > 0$, i.e., we say that $\mathcal{A}$ solves the recovery problem if 
$$ \sup_{u^{\star} \in \cU_d} \E_{X\sim P_{u^{\star}}^{\mu}} \ell_d(u^{\star}, \widehat{u^{\star}} = \mathcal{A}(X)) \leq \epsilon_d \leq 1 - k^{-1/2} d^{2\delta}\,.$$

This choice of $\ell_d$ aligns with the most common approach to planted rank-1 signal estimation - solving $\text{arg}\max_{u: \|u\|_2 = 1, \|u\|_1 = k} {u^T \Sigma u}$, where $\Sigma$ comes from either $\sw$ or the empirical covariance matrix of $\sc$ samples.

\cite{cai2016optimal} show that the best achievable loss $\ell_d^{\min}$ by (any) algorithm for $\sc$ is
$$
\b(\ell_d^{\min}\b)^2 \leq \inf_{\widehat{u^{\star}}} \E \b(\ell_d(u^{\star},\widehat{u^{\star}})\b)^2 = O(1) \frac{(\theta+1) k}{n \theta^2}\,.
$$
For $\theta = \sqrt{k^{1+\eta}/n} = \thetastat \cdot k^{\eta/2}$ with $\eta\in (0,1)$, we have $\ell_d^{\star} = O(k^{-\eta/2}) = o(1)$, and therefore, our restriction for an algorithm to achieve expected loss at most $ l_d^{\star} = 1 - k^{-1/2} d^{2\delta} = 1 - o(1)$ is mild. 

Moreover, we say that an algorithm $\mathcal{A}$ solves the $\sw$ (or $\sc$) recover problem with loss $\delta$ if in the definition above we have 
$$
\limsup_{N\to\infty} \delta_N \leq \delta\,.
$$

\subsection{Average-Case Reductions in Total Variation}\label{subsec:average-case-rec}
We now formally define an average-case reduction between two recovery problems. We start with a definition for mapping $\PP {\mu, N} {\cU}$ to $\QQ {\nu, N} {\cU'}$ for fixed dimensionality parameter $N$.

\begin{definition}[Average-Case Reduction for Recovery from $\PP {\mu, N} {\cU}$ to $\QQ {\nu, N} {\cU'}$]\label{def:avg_case_recover}

    Fix parameters $\mu, \nu$ and two collections of planted signals $\cU = \cU_N, \cU' = \cU'_N$. Let $\PP {\mu,N} {\cU}$ and $\QQ {\nu,N} {\cU'}$ be two collections of planted distributions over $\R^{D_N}$ and $\R^{D'_N}$ respectively and let $\ell_N:\cU \times \widehat{\cU} \to [0, +\infty]$, $\ell'_N:\cU' \times \widehat{\cU'} \to [0, +\infty]$ be two loss functions. Consider two recovery problems: a recovery problem for $\PP {\mu, N} {\cU}$ with a loss function $\ell_N$ and a loss threshold $\ell_N^{\star}$ and for $\QQ {\nu,N} {\cU'}$ with a loss function $\ell'_N$ and a loss threshold ${\ell'}_N^{\star}$. 
    
    We say that an $O(N^{C})$-time (potentially randomized) reduction algorithm $\mathcal{A}^{\text{red}}: \R^{D_N} \to \R^{D'_N}$ is an \textbf{average-case reduction for recovery} from $\PP {\mu} {\cU}$ to $ \QQ {\nu} {\cU'}$ with parameters $\epsilon_N, f, C', C=O(1)$ if:
    \begin{enumerate}
        \item(Reduction Algorithm) The reduction algorithm $\mathcal{A}^{\text{red}}: \R^{D_N} \to \R^{D'_N}$ is such that for all $u \in \cU$, there exists some $u' \in \cU'$ satisfying
        $$
        \tv(\mathcal{A}^{\text{red}}(X), {Q}^{\nu}_{u'})\leq \epsilon_N\,,
        $$
        where $X \sim P^{\mu, N}_u \in \PP {\mu, N} {\cU}$, ${Q}^{\nu, N}_{u'} \in \QQ {\nu, N} {\cU'}$ for some $\epsilon_N > 0$ %, and $\lim_{N\to \infty} \epsilon_N = 0$.
        \item(Recovery Algorithm) There exists an $O(N^{C'})$-time algorithm $\mathcal{A}^{\text{rec}}: \R^{D_N}\times \widehat{\cU'_N} \to \widehat{\cU_N}$ for $C'=O(1)$, such that if $X, u, u' = u'(X)$ are defined as above and $\widehat{u'} \in \widehat{\cU'}$ is such that $\ell'_N(u', \widehat{u'}) \leq \delta_N \leq {\ell'_N}^{\star}$, then
        $$
        \sup_{u\in \cU} \E_{X \sim P_{u}^{\mu, N}} \ell_N\b(u, \mathcal{A}^{\text{rec}}(X, \widehat{u'})\b) \leq \eta_N\,,
        $$
        where $\eta_N = f(\delta_N)$.
    \end{enumerate}

\end{definition}

The following Lemma states the complexity implications of existence of such a reduction. 

\begin{lemma}[Implication of Reduction for Recovery from $\PP {\mu, N} {\cU}$ to $\QQ {\nu, N} {\cU'}$]\label{lem:red_recover0}
    Let $\PP {\mu, N} {\cU}, \QQ {\nu, N} {\cU'}$ be two recovery problems defined as above and assume there exists an $O(N^{C})$-time average case reduction for recovery with parameters $\epsilon_N, f, C'$ from the definition above. Then, if there is an $O(N^{M})$-time algorithm for recovery for $\QQ {\nu, N} {\cU'}$, that with probability at least $1 - \psi_N$ achieves loss at most $\delta_N$ for the loss function $\ell'_N$, there exists an $O(N^{M} + N^{C} + N^{C'})$-time algorithm for recovery for $\PP {\mu, N} {\cU}$ that with probability at least $1-\epsilon_N - \psi_N$ achieves loss at most $\eta_N = f(\delta_N)$ for the loss function $\ell_N$.
\end{lemma}
\begin{proof}
    Given a sample $X$ from $P_u^{\mu, N} \in \PP {\mu, N} {\cU}$ and the $O(N^{M})$-time recovery algorithm $\mathcal{B}:\R^{D'_N} \to \widehat{\cU'_N}$, consider an algorithm $\mathcal{A}:\R^{D_N}$ defined as
    $$
    \mathcal{A} = \mathcal{A}^{\text{rec}}\B(X, \mathcal{B} \circ \mathcal{A}^{\text{red}}(X)\B)\,.
    $$
    Clearly, the runtime of $\mathcal{A}$ is a sum of the runtimes of the three subroutines $\mathcal{B}, \mathcal{A}^{\text{red}}, \mathcal{A}^{\text{rec}}$: $O(N^{M} + N^{C} + N^{C'})$.
    From the coupling property of total variation and the guarantee for recovery algorithm $\mathcal{B}$, with probability at least $1-\epsilon_N - \psi_N$, $\mathcal{B} \circ \mathcal{A}^{\text{red}}(X) = \widehat{u'}$, where $\ell'_N(u', \widehat{u'}) \leq \delta_N$. Then, from the guarantees for the recovery algorithm, with probability at least $1-\epsilon_N - \psi_N$, $$\ell_N\B(u, \mathcal{A}^{\text{rec} }\B(X, \mathcal{B} \circ \mathcal{A}^{\text{red}}(X)\B)\B) = \ell_N(u, \mathcal{A}^{\text{rec}}(X, \widehat{u'})) \leq \eta_N = f(\delta_N)\,,$$
    which concludes the proof.
\end{proof}

Similarly to the detection case we define $\cP^\mu_{\cU} = \b\{\PP {\mu, N} {\cU_N}\b\}_{N=1}^{\infty}$ and $\cQ^\nu_{\cU'} = \b\{\QQ {\nu,N} {\cU'_N}\b\}_{N=1}^{\infty}$, where $\mu, \nu$ are independent of the dimensionality parameter. We simply adapt the Def.~\ref{def:avg_case_points} substituting the detection reduction with the recovery one.

\begin{definition}[Average-Case Reduction for Recovery]\label{def:avg_case_points_rec}
Let $\cP^\mu_{\cU}  = \b\{\PP {\mu, N} {\cU_N}\b\}_{N=1}^{\infty}$ and $\cQ^\nu_{\cU'}  = \b\{\QQ {\nu,N} {\cU'_N}\b\}_{N=1}^{\infty}$ for fixed parameters $\mu, \nu$ independent of $N$. 
We say that an $O(N^{C})$-time algorithm $\mathcal{A} = \{\mathcal{A}_N:\R^{D_N}\to \{0,1\}\}$ is an \emph{average-case reduction for recovery from $\cP^{\mu}_{\cU}$ to $\cQ^{\nu}_{\cU'}$} if
    for any sequence $\{\mu_i\}_{i=1}^{\infty}$, such that $\lim_{i\to\infty} \mu_i = \mu$, there exist a sequences $\{\nu_i, N_i\}_{i=1}^{\infty}$ and an index $i_0$, such that 
    \begin{enumerate}
        \item $\lim_{i\to\infty} \nu_i = \nu$ and $\lim_{i\to\infty} N_i = \infty$;
        \item for any $i > i_0$, $\mathcal{A}_{N_i}$ is an average-case reduction for recovery from $\PP {\mu, N_i} {\cU_{N_i}}$ to $\QQ {\nu_i, N_i} {\cU'_{N_i}}$ with parameters $\epsilon_{N_i}, f, C'$ from Def.~\ref{def:avg_case_recover};
        \item $\lim_{i\to\infty} \epsilon_{N_i} = 0$;
    \end{enumerate}
\end{definition}

\subsection{Implications of Average-Case Reductions for Phase Diagrams}

We now apply Lemma~\ref{lem:red_recover0} to the Def.~\ref{def:avg_case_points_rec} of an average-case reduction to obtain an implication of existence of average-case reductions 
phase diagrams. Similarly to the detection case, we state the following lemma for the case of reductions between $\sc_{\gamma}$ and $\sw$ to simplify the notation, noting that the lemma can be easily generalized.

\begin{lemma}[Implication of Reduction for Recovery for Phase Diagrams]\label{lem:red_recover}
    Consider the recovery problems associated with $\sc_{\gamma}(\mu), \sw(\nu)$ for $\mu \in \Omc_{\gamma}, \nu \in \Omw$ and assume there exists an $O(d^C)$-time average-case reduction for recovery from $\sc_{\gamma}(\mu)$ to $\sw(\nu)$ (Def.~\ref{def:avg_case_points}) with parameters $f, C'$. Moreover, assume that $f(\delta') \leq \eta$ for every $\delta' \leq \delta$.
    Fix $\epsilon > 0$ and assume there is a $O(d^M)$-time algorithm for recovery for $\sw(\nu')$ for all $\nu' \in B_{\epsilon}(\nu)$ achieving the loss $\delta$ with probability at least $1-\psi$. Then there exists a $O(d^{M} + d^C + d^{C'})$-time algorithm for recovery for $\sc_{\gamma}(\mu)$ achieving the loss at most $\eta$ with probability at least $1-\psi$.

    A analogous implication holds for an average-case reduction for recovery from $\sw(\nu)$ to $\sc_{\gamma}(\mu)$.
\end{lemma}
\begin{proof}
    The proof follows exactly the proof for analogous lemma for detection: Lemma~\ref{lem:red_detect}
\end{proof}

Finally, the definition for the average-case reduction for recovery at the computational threshold is defined in exactly the same way as the one for recovery, see Def.~\ref{def:avg_case_points_threshold} and has implications analogous to Lemma~\ref{lem:red_comp_thr_impl} as a consequence of Lemma~\ref{lem:red_recover}.

\section{Deferred Proofs}

\subsection{Proof of Corollary~\ref{cor:clone_reduction}}\label{proof:clone_reduction}

\begin{mycor}{\ref{cor:clone_reduction}}\textnormal{(Canonical Reduction from $\sc$ to $\sw$, Corollary of Thm.~\ref{t:bipartite})} With canonical parameter correspondence $\mu = (\alpha, \beta, \gamma) \leftrightarrow \nu = f(\mu)=(\alpha, \beta + \gamma/2)$~\eqref{eq:param}, there exists a $O(d^{2+\gamma})$-time average-case reduction for both detection and recovery from $\sc(\mu)$ to $\sw(\nu=f(\mu))$ for $\gamma \geq 2, \mu\in\Omc_{\gamma}, \nu\in\Omw$.
\end{mycor}

\begin{proof}
    Fix $\gamma > 2$. From Theorem~\ref{t:bipartite}, there exists a $O(d^2n=d^{2+\gamma})$-time reduction algorithm $\mathcal{A}^{\text{Red}} = \CloneCov : \R^{n\times d} \to \R^{d\times d}$ mapping an instance of $\sc(d, k, \theta < \thetacomp, n=d^{\gamma})$ with a unit signal vector $u \in\ksparse d$ to an instance of $\sw(d,k,\lambda = \theta\sqrt{n})$ with the same signal vector $u$. 
    
    First, we show that an $\CloneCov$ is an average-case reduction from $\sc(d, k, \theta, n=d^{\gamma})$ to $\sw(d,k,\lambda = \theta\sqrt{n})$ for detection (Def.~\ref{def:avg_case_detect}) and recovery (Def.~\ref{def:avg_case_recover}). Note that the algorithm succeeds for $\theta = \lambda = 0$, so the existence of an average-case reduction for detection follows. Moreover, since the algorithm $\CloneCov$ preserves the signal vector $u$ intact, there is no need for a separate recovery algorithm stage, and we establish an average-case reduction for recovery with the same estimation loss functions for $\sc$ and $\sw$. 

    According to Definition~\ref{def:avg_case_points}, we obtain an existence of the average-case reduction for both recovery and detection from $\sc\b(\mu = (\alpha, \beta, \gamma)\b)$ to $\sw\b(\nu = (\alpha, \beta+\gamma/2)\b)$ for $\gamma > 2$. Indeed, for any sequence $\mu_i = (\alpha_i, \beta_i, \gamma)$ converging to $mu$ we consider a sequence $\nu_i = (\alpha_i, \beta_i+\gamma/2)$ and since the reduction from $\sc(\mu_i)$ to $\sw(\nu_i)$ is established, the existence of an average-case reduction from $\sc(\mu)$ to $\sw(\nu)$ follows. 

    We now extend the result from $\gamma>2$ to the case of $\gamma = 2$. The goal is to prove that there exist average-case reductions for recovery and detection from $\sc(d, k, \theta < \thetacomp, n=d^{\gamma=2})$ to $\sw(d,k,\lambda = \theta\sqrt{n})$. We do this in two steps: 
    $$
    \sc(d, k, \theta, n=d^{\gamma}) \to \sc(d, k, \theta' = \theta d^{\eta}, n'=d^{\gamma + \eta}) \to \sw(d,k,\lambda' = \theta'\sqrt{n'})\,,
    $$
    where we choose any small $\eta>0$ and have $\theta'\approx \theta$.
    As established above, since $\gamma + \eta > 2$, $\CloneCov$ is an average-case reduction for both detection and recovery for the second step in the chain. 
    
    We now establish a simple average-case reduction that achieves the first step - that is, slightly increases the number of samples $n$ relative to dimension $d$, keeping the SNR $\theta$ roughly the same. Given a sample $Z = X + \tsq g u^\top \in \R^{n\times d}$ from $\sc(d, k, \theta, n=d^{\gamma}$, consider two clones $Z\up1, Z\up2 = \textsc{GaussClone}(Z)$. From Lemma~\ref{lem:cloning-guarantee},
    $$
    Z\up1 = X\up1 + \tsq gu^\top, Z\up2 = X\up2 + \tsq gu^\top\,,
    $$
    where $X\up1, X\up2 \sim N(0, I_n)^{\otimes d}$ are independent. We now randomly rotate $Z\up2$, i.e. consider ${Z'}\up2 = U Z\up2$ for $U\sim\text{Unif}(\mathcal{O}(n))$ and let a matrix $Z' \in \R^{2n \times d}$ be such that the first $n$ rows form $Z\up 1$ and the last $n$ rows form ${Z'}\up 2$. Since $Z\up 1, {Z'}\up2$ are tow independent samples from $\sc(d, \theta/2, k, n)$, it is easy to verify that the resulting matrix $Z'$ is a random sample from $\sc(d, \theta/2, k, 2n)$. Repeating this procedure $\log_2 (d^{\eta})$ times we obtain an algorithm that maps an instance of $\sc(d,\theta,k,n)$ to an instance $\sc(d,\theta d^{-\eta}, k, n d^{\eta})$ in $0$ total variation. Moreover, it preserves the signal vector $u$ intact, thus resulting in an average-case reduction for both detection and recovery. This completes the first step of the two step process. It is easy to verify that sequentially combining two average-case reductions that preserve the signal $u$ the same results in an average-case reduction for both detection and recovery. 

    Finally, fix any $\alpha, \beta < \betacomp$. For any $\eta > 0$, we have obtained an average-case reduction for detection and recovery from $\sc\b(\mu = (\alpha, \beta, \gamma)\b)$ to $\sw\b(\nu = (\alpha, \beta'=\beta + \gamma/2 - \eta/2)\b)$. Since we can take an arbitrarily small $\eta$, for any sequences $\{\mu_i = (\alpha_i, \beta_i, \gamma), d_i\}_i$, such that $\lim_{i\to\infty} \mu_i = \mu$ and $\lim_{i\to\infty} d_i = \infty$, we can define a sequence $\{\nu_i = (\alpha_i, \beta'_i=\beta_i + \gamma/2 - \eta_{d_i}/2)\}_i$, and observe that there exists an average-case reduction for both detection and recovery from $\sc(d_i, k_i=d_i^{\alpha_i}, \theta_i= d_i^{\beta_i}, n_i = d_i^{\gamma})$ to $\sw(d_i,k_i,\lambda_i = \theta_i\sqrt{n_i})$ and $\lim_{i\to\infty} \nu_i = \nu$. Then by Definition~\ref{def:avg_case_points}, there exists an average-case reduction from $\mu = (\alpha, \beta, \gamma)$ to $\nu = (\alpha, \beta+\gamma/2)$ for $\gamma = 2$.
\end{proof}

\subsection{Proof of Corollary~\ref{cor:gs_reduction}}\label{proof:gs_reduction}

\begin{mycor}{\ref{cor:gs_reduction}}\textnormal{(Canonical Reduction from $\sc$ to $\sw$ at the Computational Threshold)}

With canonical parameter correspondence $\mu = (\alpha, \beta, \gamma) \leftrightarrow \nu = f(\mu)=(\alpha, \beta + \gamma/2)$~\eqref{eq:param}, there exists an $O(d^{2+\gamma})$-time average-case reduction at the computational threshold for both detection and recovery\footnote{As described in Section~\ref{sec:avg_case}, we prove the reduction for recovery for a slightly restricted class of potential recovery algorithms for $\sw$, which we believe to be a very mild restriction. The recovery algorithm in this case also finishes in $O(d^{2+\gamma})$ and given a recovery algorithm for $\sw$ with loss $<{\ell'}^{\star}$, is guaranteed to output an estimate for $\sc$ with loss at most ${\ell'}^{\star} + o(1)$.} from $\sc(\mu)$ to $\sw(\nu)$ for $\alpha \leq 1/2, \gamma\geq1, \mu \in \CCc_{\gamma}, \nu \in \CCw$.

\end{mycor}

\begin{proof}
    Fix $\epsilon > 0, \alpha\leq 1/2$, and let $n = d^{\gamma = 1 + \epsilon}, k = d^{\alpha}, \theta = \delta\thetacomp(d,k,n)$ for $\delta \leq (\log n)^{-1}$. Define constants $A_{\alpha, \epsilon}, C_{\alpha, \epsilon}, K_{\alpha, \epsilon}, \psi$ as in the statement of Theorem~\ref{t:wishSSBM}. From Theorem~\ref{t:wishSSBM},
    there exists a $O(d^{2+\gamma})$-reduction algorithm $$\mathcal{A}^{\text{Red}} = \WishartToSSBM(\cdot, 2K_{\alpha, \epsilon}, \psi) : \R^{n\times d} \to \R^{d\times d}$$ that maps an instance of $\sc(d, k, \theta=\delta\thetacomp, n)$ with a unit signal vector $u \in\ksparseflat d$ to an instance of $\sw(d,k,\lambda)$ with a signal vector $u'\in \ksparseflat d$ and $\lambda \in [\phi \lamcomp(d,k), \lamcomp(d,k)]$ for some $\phi \geq \delta^{2A_{\alpha, \epsilon}} \cdot C(A_{\alpha, \epsilon}) \log^{-1/2} n$. 

    First, it is easy to verify that the algorithm $\mathcal{A}^{\text{Red}}$ defined as above for some value $\theta > 0$ successfully maps the distributions in the absence of signal:
    $$
    \tv \b( \mathcal{A}^{\text{Red}}\b(N(0,I_n)^{\otimes d}\b), W \b) \to 0 \quad \text{as} \quad n\to\infty\,,
    $$
    where $W\sim \GOE(d)$, which proves that $\mathcal{A}^{\text{Red}}$ satisfies the null-hypothesis requirement of Definition~\ref{def:avg_case_detect}. Moreover, since given a Spiked Covariance sample with a signal vector $u\in \ksparseflat d$ as an input $\mathcal{A}^{\text{Red}}$ outputs an instance of Spiked Wigner with a signal vector $u' \in \ksparseflat d$, we conclude the existence of an average-case reduction for detection according to Definition~\ref{def:avg_case_detect}.

    We now prove that $\mathcal{A}^{\text{Red}}$ can be used to establish an average-case reduction for recovery. We modify our reduction algorithm slightly to only act on half of our initial sample $Z$. Given $Z\in \R^{n\times d}$ a sample from Spiked Covariance Model ~\eqref{e:SCM2}, let $Z\up 1, Z\up 2\in \R^{n/2 \times d}$ denote the first and the last halves of the rows of $Z$. Note that $Z\up1, Z\up2$ are two independent instances of $\sc(d, k, \theta, n/2)$. Let $\alpha' = \alpha, \epsilon' = \log_d (n/2) - 1, \delta' = \thetacomp(d, k, n/2) / \theta$ and define new constants $A_{\alpha', \epsilon'}, C_{\alpha', \epsilon'}, K_{\alpha', \epsilon'}, \psi'$ as in Theorem~\ref{t:wishSSBM}. We let $$\mathcal{A}_{\text{recovery}}^{\text{Red}}(Z) := \WishartToSSBM(Z\up1, 2K_{\alpha', \epsilon'}, \psi')\,,$$
    where $Z\up 1$ consists of the first $n/2$ rows of $Z$ as described above. From Theorem~\ref{t:wishSSBM}, $Y = \mathcal{A}_{\text{recovery}}^{\text{Red}}(Z)$ is close in total variation to an instance of $\sw(d, k, \lambda)$ with a unit vector signal $u' \in \ksparseflat, u'_i = |u_i|$ and $\lambda \in [\phi \lamcomp(d,k), \lamcomp(d,k)]$ for some $\phi \geq \delta^{2A_{\alpha', \epsilon'}} \cdot C(A_{\alpha', \epsilon'}) \log^{-1/2} (n/2)$.

    We now define $\mathcal{A}^{\text{Rec}}: \R^{n\times d} \times \R^d \to \R^d$ -- an estimation part of the recovery average-case reduction. The idea is to use the remaining rows $Z\up2$ and the sparsity information in $\widehat{u'}$ to recover~$u$.

    As discussed in Section~\ref{sec:avg_case_recovery}, we assume that $\widehat{u'} \in \mathbb{S}^{d-1}, \|\widehat{u'}\|_1 \leq k$ and the recovery algorithm for the Spiked Wigner instance satisfies  $$\ell'(u', \widehat{u'}) = 1 - \b(\la u', \widehat{u'} \ra\b) ^2 \leq 1 - k^{-1/2} d^{2\delta}\,.$$
    It is easy to verify that we then have  $\supp(\widehat{u'})\cap\supp(u')=  \Omega(k^{1/2} d^{\delta})$.

    Then, selecting the rows of $Z\up 2$ that are in the support of $\widehat{u'}$ creates an instance of $\sc$ of dimension $d' = k=d^{\alpha}$, sparsity parameter at least $k' = k^{1/2} d^{\delta} = d^{\alpha/2 + \delta}$ and $\theta' \approx d^{\alpha/2 + \delta} / \sqrt{n}$. Now for $k' \geq \sqrt{d'}$, the computational threshold is $\thetacomp(d',k',n) = \sqrt{d'/n} \ll \theta'$, and therefore, a spectral algorithm solves the estimation problem in time $O(d^2n)$, achieving loss at most $o(1)$ for the new problem, for example, using the recovery algorithm in \cite{yuan2013truncated}. It is easy to verify that this achieves $\ell(u, \widehat{u}) \leq \ell'(u', \widehat{u'}) + o(1)$, where $\widehat{u}$ is the output of the corresponding spectral algorithm.

    For every $\delta \leq (\log n)^{-1}$, we have established average-case reductions for both detection and recovery from $\sc(d, k, \theta=\delta\thetacomp, n)$ to $\sw(d,k,\lambda)$ for $\lambda \in [\phi \lamcomp(d,k), \lamcomp(d,k)]$ for some $\phi \geq \delta^{2A_{\alpha, \epsilon}} \cdot C(A_{\alpha, \epsilon}) \log^{-1/2} n$. Note that for $n = d^{\gamma}$, $$\lim_{d\to\infty} \log_d (\log n)^{-1} = 0\,.$$ Thus, for every $\eta > 0, \eta' > 0$, for sufficiently large $d > d_{\eta}$, 
    we have a reduction from $\sc\b(\mu = (\alpha, \beta = \betacomp - \eta, \gamma)\b)$ to $\sw\b( \nu = (\alpha, \beta' = \betacomp + \gamma/2 - \eta'') \b)$, where $\betacomp = \min\{\alpha-\gamma/2, 1/2-\gamma/2\}$,  and $ \eta'' = - \log_d \phi(\delta = d^{-\epsilon}) \leq 2A_{\alpha, \epsilon}\eta + \eta'$. 
    Then, for every sequence of parameters $\mu_i \to \mu$, we can choose sufficiently small d$\eta_i, \eta'_i$ and $d_i > d_{\eta_i}$ to obtain a reduction from $\sc(\mu_i)$ to $\sw(\nu_i)$ in dim. $d_i$, where $\nu_i = \alpha_i, \beta'_i = \betacomp + \gamma/2 - \eta''_i$. Since $\eta''_i \to 0$ as $\eta_i, \eta'_i \to 0$, by Definition~\ref{def:avg_case_points}, there exists an average-case reduction from $\mu = (\alpha, \beta, \gamma)$ to $\nu = (\alpha, \beta+\gamma/2)$ for $\gamma > 1$.

    It is now left to prove analogous statement for $\gamma = 1$. The proof parallels that of Corollary~\ref{cor:clone_reduction}, where we have similarly used a result for all $\gamma > 2$ to obtain the reductions for $\gamma = 2$.
\end{proof}

\newpage 

\bibliographystyle{alpha}
\bibliography{GB_BIB}

\end{document}